\PassOptionsToPackage{table,xcdraw}{xcolor}
\documentclass[onefignum,onetabnum]{siamonline190516}

\usepackage[table]{xcolor}

\usepackage[utf8]{inputenc}
\usepackage[english]{babel}
\usepackage{graphicx}
\usepackage{fancyhdr}
\usepackage[utf8]{inputenc}
\usepackage{amsfonts}
\usepackage{amssymb}
\usepackage{amsmath}
\usepackage{amsfonts}
\usepackage{amssymb}
\usepackage{mathtools}
\usepackage{enumitem}
\usepackage[percent]{overpic}
\usepackage{tikz}
\usepackage{mathrsfs}
\usepackage{xargs}
\usepackage[colorinlistoftodos,prependcaption,textsize=tiny]{todonotes}
\newcommandx{\at}[2][1=]{\todo[linecolor=red,backgroundcolor=red!25,bordercolor=red,#1]{#2}}

\usepackage{algorithm}
\usepackage{algorithmic}

\usepackage{mathtools,booktabs}

\usepackage{varwidth}
\usepackage{hyperref}
\usepackage{cleveref}
\usepackage{cite}
\definecolor{rev1}{HTML}{cb270f}
\definecolor{rev2}{HTML}{1c8235}

\DeclareMathOperator{\argmin}{arg\!\min}

\usepackage{ifthen,tabularx,graphicx,multirow}

\usepackage{rotating}
\usepackage{etoolbox}
\usepackage{tcolorbox}

\newcolumntype{Y}{>{\centering\arraybackslash}X}
\definecolor{ddg}{rgb}{0.2, 0.2, 0.2}
	\definecolor{llg}{rgb}{0.95, 0.95, 0.95}
	\tikzstyle{stuff_fill}=[rectangle,draw,fill=pink,minimum size=0.5em]

\ifpdf
  \DeclareGraphicsExtensions{.eps,.pdf,.png,.jpg}
\else
  \DeclareGraphicsExtensions{.eps}
\fi

\usepackage{enumitem}
\setlist[enumerate]{leftmargin=.5in}
\setlist[itemize]{leftmargin=.5in}

\newsiamremark{remark}{Remark}
\newsiamremark{hypothesis}{Hypothesis}
\crefname{hypothesis}{Hypothesis}{Hypotheses}
\newsiamthm{claim}{Claim}

\title{WARPd: A linearly convergent first-order method for inverse problems with approximate sharpness conditions\thanks{Submitted to the editors \today.
\funding{This work was supported by a Research Fellowship at Trinity College, Cambridge.}}}
\author{Matthew J. Colbrook\thanks{Centre Sciences des Données, École Normale Supérieure. (\email{m.colbrook@damtp.cam.ac.uk})}}
\headers{WARPd: Linear convergence with approx. sharpness}{Matthew Colbrook}

\usepackage{amsopn}

\usepackage{arydshln}

\usepackage{pdfpages}
\begin{document}

\linespread{0.992}\selectfont{}

\maketitle
\vspace{-1mm}
\begin{abstract}
Reconstruction of signals from undersampled and noisy measurements is a topic of considerable interest. Sharpness conditions directly control the recovery performance of restart schemes for first-order methods without the need for restrictive assumptions such as strong convexity. However, they are challenging to apply in the presence of noise or approximate model classes (e.g., approximate sparsity). We provide a first-order method: Weighted, Accelerated and Restarted Primal-dual (WARPd), based on primal-dual iterations and a novel restart-reweight scheme. Under a generic approximate sharpness condition, WARPd achieves stable linear convergence to the desired vector. Many problems of interest fit into this framework. For example, we analyze sparse recovery in compressed sensing, low-rank matrix recovery, matrix completion, TV regularization, minimization of $\|Bx\|_{l^1}$ under constraints ($l^1$-analysis problems for general $B$), and mixed regularization problems. We show how several quantities controlling recovery performance also provide explicit approximate sharpness constants. Numerical experiments show that WARPd compares favorably with specialized state-of-the-art methods and is ideally suited for solving large-scale problems. We also present a noise-blind variant based on the Square-Root LASSO decoder. Finally, we show how to unroll WARPd as neural networks. This approximation theory result provides lower bounds for stable and accurate neural networks for inverse problems and sheds light on architecture choices. Code and a gallery of examples are made available online as a MATLAB package.
\end{abstract}

\vspace{-1mm}


\vspace{-1mm}

\vspace{-2mm}

\section{Introduction} Reconstruction from undersampled measurements is a key problem in signal and image processing, machine learning, statistics, computer vision, and a variety of other fields. In this paper, we consider the following canonical linear inverse problem:
\begin{equation}\setlength\abovedisplayskip{5pt}\setlength\belowdisplayskip{5pt}
\label{eq:inv_prob0}
\mbox{Given measurements $b = A \varkappa + e \in \mathbb{C}^m$, recover $\varkappa  \in \mathbb{C}^N$.}
\end{equation}
Here, $A \in \mathbb{C}^{m \times N}$ represents a sampling model ($m < N$) and $e\in\mathbb{C}^N$ models noise or perturbations. Note that $\varkappa$ could correspond to a vectorized image or matrix.\footnote{We have used the notation $\varkappa$ to avoid confusion with $x$ used as a dummy variable below.} Over the last few decades there has been an explosion in nonlinear reconstruction techniques for \eqref{eq:inv_prob0} (see \cite{rudin1992nonlinear,jin17,arridge2019solving,hastie2015statistical,beck2009fast,daubechies2004iterative,chambolle2016introduction,Mallat09,candes2009exact,chandrasekaran2012convex,recht2010guaranteed,fannjiang2020numerics,adcock2021compressive} for a very incomplete list). For example, the field of compressed sensing shows that, under certain conditions, accurate reconstruction is possible if $\varkappa$ is (approximately) sparse \cite{candes2006robust,donoho2006compressed,candes2006stable}. A popular approach to recover $\varkappa$ is to solve an optimization problem of the following form:
\begin{equation}\setlength\abovedisplayskip{5pt}\setlength\belowdisplayskip{5pt}
\label{main_problem}
\min_{x\in \mathbb{C}^N} \mathcal{J}(x)+\|Bx\|_{l^1}\quad \text{s.t.}\quad\|Ax-b\|_{l^2}\leq \epsilon.
\end{equation}
Throughout this paper, $\mathcal{J}$ denotes a seminorm (e.g., a regularizer, which depends on prior assumptions about the signal $\varkappa$) and $B\in\mathbb{C}^{N\times q}$ is a generic matrix. For example, $\|Bx\|_{l^1}$ could correspond to the popular TV-seminorm $\|x\|_{\mathrm{TV}}$ \cite{chambolle2010introduction} (see \cref{sec:TVTVTVT}) or a sum $\|Wx\|_{l^1}+\lambda\|x\|_{\mathrm{TV}}$ for general $W$ (see \cref{sec:num_shear_TGV}). The formulation in \eqref{main_problem} with the constraint $\|Ax-b\|_{l^2}\leq \epsilon$ is often theoretically preferred over other variations (such as $\|Ax-b\|_{l^2}^2$ in the objective function) because a reasonable estimate of $\epsilon$ may be known \cite{becker2011nesta}. {The case of unknown $\epsilon$ and replacing the constraint $\|Ax-b\|_{l^2}\leq \epsilon$ by a term $\lambda^{-1}\|Ax-b\|_{l^2}$ in the objective function, where $\lambda$ scales \textit{independently} of the noise, is treated in \cref{sqre_root}.}

Due to the large interest in solving \eqref{main_problem} and similar problems, there is a long list of algorithms (see \cref{sec:prev_work}), with a particular emphasis on first-order methods\footnote{Due to their dimensionality, many large-scale optimization models have rendered second-order methods computationally impractical (typically large systems of linear equations are solved to compute Newton steps). Thus, efficient and accelerated first-order algorithms have become essential for tackling numerous problems.} for large scale problems. The goal is to design simple schemes (e.g., matrix/vector multiplications) that produce approximate solutions efficiently. Solving \eqref{main_problem} is a notoriously difficult challenge, with common issues being non-smoothness of $\mathcal{J}$, the analysis term $\|Bx\|_{l^1}$, the constraint $\|Ax-b\|_{l^2}\leq \epsilon$ etc. First-order methods typically need a very large number of iterations when high accuracy is required (see \cite{nesterov2003introductory} for optimal convergence rates for different classes of objective functions). There are also many works on the limits of recovering $\varkappa$ via solutions of \eqref{main_problem}, and this is intimately linked to numerical performance. It has been observed empirically that recovery problems \eqref{eq:inv_prob0} that are easier to solve theoretically (e.g., larger $m$) often lead to optimization problems \eqref{main_problem} that are easier/more efficient to solve numerically \cite{donoho2008fast}. In some cases, this has led to algorithms with accelerated convergence guarantees \cite{roulet2020computational}.

This paper provides a general framework for the accelerated (linear) convergence and stable solution of \eqref{eq:inv_prob0}. Our only assumption is an inequality of the form\footnote{See \cref{sec:extensions} for extensions, such as different norms for measuring the error.}
\begin{equation}\setlength\abovedisplayskip{5pt}\setlength\belowdisplayskip{5pt}
\label{assumption}
\|\hat x-x\|_{l^2}\! \leq\!  C_1\! \Big[\underbrace{\mathcal{J}\! (\hat x)+\|B\hat x\|_{l^1}-\mathcal{J}\! (x)-\|Bx\|_{l^1}}_{\text{objective function difference}}+C_2\underbrace{\left(\|A\hat x-b\|_{l^2}\! -\epsilon\right)}_{\text{feasibility gap}}+\underbrace{c(x,b)}_{\!\!\!\!\!\!\text{approx. term}\!\!\!\!\!\!}\Big],\quad \! \! \forall x,\hat  x\in \mathbb{C}^N.
\end{equation}
Here $C_1$ and $C_2$ are constants and $c(x,b)$ should be understood as a small approximation term. For example, in the case of sparse recovery considered in \cref{application1} and taking $x=\varkappa$, $c(\varkappa,b)$ measures the distance of $\varkappa$ to sparse vectors and contains a term proportional to the noise level $\epsilon$ (see \cref{main_theorem1}). More generally, \eqref{assumption} is much weaker than typical assumptions for acceleration such as strong convexity, and can be considered an \textit{approximate} \L{}ojasiewicz-type inequality \cite{roulet2020sharpness}. We discuss its links to other error bounds in \cref{sec:prev_work}. It turns out that many of the recovery results for \eqref{eq:inv_prob0} in the literature are proven via such an inequality or local versions restricted to specific vectors. We provide analysis of several examples below.

Given \eqref{assumption}, we provide an iterative algorithm: \textbf{W}eighted, \textbf{A}ccelerated and \textbf{R}estarted \textbf{P}rimal-\textbf{d}ual (WARPd), based on primal-dual iterations and a novel restart-reweight scheme. Our main convergence result is summarized in the following theorem.

\vspace{-1mm}
\begin{theorem}[Uniform stable recovery with linear convergence]
\label{main_theorem}
Suppose that \eqref{assumption} holds. Let $L$ be an upper bound for $\sqrt{\|A\|^2+\|B\|^2}$, $\tau\in(0,1)$ (step size), $\nu\in(0,1)$ and $\delta>0$. Then for any $n\in\mathbb{N}$ and any pair $(\varkappa,b)\in\mathbb{C}^{N}\times\mathbb{C}^m$ such that $\|A\varkappa-b\|\leq \epsilon$ and $c(\varkappa,b)\leq\delta$,
\begin{equation}\setlength\abovedisplayskip{3pt}\setlength\belowdisplayskip{3pt}
\label{key_theorem_bound}
\|\phi_{n}(b)- \varkappa\|_{l^2}\leq C_1\left(\frac{\delta}{1-\upsilon}+\upsilon^nC_2\|b\|_{l^2}\right),
\end{equation}
where $\phi_n(b)$ denotes the output of \text{WARPd} in \cref{alg:RR_primal_dual}.\vspace{-2mm}
\end{theorem}\vspace{-2mm}

The total number of inner iterations is a multiple of $n$. We show that the optimal choice of $\upsilon$ is $\exp(-1)$, for which $\sim LC_1\sqrt{C_2^2+q}\cdot\log({C_2\|b\|_{l^2}}/{\delta})$ total inner iterations are required to balance the two terms on the right-hand side of \cref{key_theorem_bound}. In other words, \cref{main_theorem} demonstrates \textit{linear (or exponential) convergence down to the error bound $\sim C_1\delta$}. Note that the barrier $C_1\delta$ between solutions of \eqref{main_problem} and $\varkappa$ in \eqref{eq:inv_prob0} is to be expected from the $c(\cdot,b)$ term in \eqref{assumption} when the objective function difference and feasibility gap vanish. Moreover, the convergence result is stable in perturbations to $\varkappa$ or $b$, with stability governed by $c(\varkappa,b)\leq\delta$. Finally, each iteration of WARPd only requires a few matrix-vector operations and applying the proximal map of $\mathcal{J}$. In particular, we do not assume anything on the matrices $A$ and $B$ (e.g., we do not assume that $A^*A$ is an orthogonal projector or that $B$ is diagonal). Together with the acceleration, this makes WARPd very computationally efficient.

Many problems of interest satisfy a version of \eqref{assumption} and there is great flexibility in our framework. To be concrete, we explicitly analyze the following examples:
\begin{itemize}[leftmargin=0.8in]
	\item[\cref{application1}:] Sparse recovery, using the robust null space property (in levels) to obtain \eqref{assumption}.
	\item[\cref{application2}:] Low-rank matrix recovery, using the Frobenius-robust rank null space property to obtain \eqref{assumption}.
	\item[\cref{sec:mat_com_exm}:] Matrix completion, using the existence of approximate dual certificates to obtain a local version of \cref{assumption}.
	\item[\cref{application3}:] Examples with non-trivial matrix $B$ including $l^1$-analysis with frames (using a generalization of the restricted isometry property to obtain \eqref{assumption}) and total variation minimization (using the restricted isometry property to obtain \eqref{assumption}).
\end{itemize}
Comprehensive numerical experiments demonstrate that WARPd compares favorably with state-of-the-art methods. We also consider a variant WARPdSR in \cref{sqre_root} that covers the case of unknown $\epsilon$ and replaces the constraint $\|Ax-b\|_{l^2}\leq \epsilon$ by a term $\lambda^{-1}\|Ax-b\|_{l^2}$ in the objective function, where $\lambda$ scales \textit{independently} of the noise. Some further extensions and adaptations are also discussed in \cref{sec:extensions}. For example, one can replace the $l^1$-norm in the $\|Bx\|_{l^1}$ term by any norm whose dual unit ball has a simple projection.

\subsection{Accurate and stable neural networks (NNs)}
\label{sec:NN_unravel_alg}

Given the current intense interest in deep learning (DL), it is not surprising that numerous DL-based methods are now being proposed for the above and similar problems (see \cite{wang2018image,jin17,hammernik2018learning,mccann2017convolutional,arridge2019solving,bubba2019learning,kobler2020total} for a small sample). There is ample evidence that DL has the potential to achieve state-of-the-art results in numerous applications. However, a current challenge is that many DL-based methods lack theoretical foundations regarding reconstruction guarantees, convergence rates, stability analysis, and other basic numerical analysis questions. The stability question is particularly alarming, with empirical evidence that current DL techniques typically lead to unstable methods for inverse problems (e.g., ``adversarial attacks'') \cite{huang2018some,antun2020instabilities,finlayson2019adversarial}. For example, this is a problem in real-world clinical practice. Facebook and NYU's 2019 FastMRI challenge reported that networks that performed well in standard image quality metrics were prone to false negatives, failing to reconstruct small but physically relevant image abnormalities \cite{knoll2020advancing}. Subsequently, the 2020 FastMRI challenge \cite{muckley2020state} focused on pathologies and ``AI-generated hallucinations.'' AI-generated hallucinations pose a serious danger in applications such as medical imaging. The big problem, therefore, is to compute/train NNs that are both accurate and stable \cite{devore2020neural,adcock2020gap,colbrookNN2021can}.

In light of this, we consider unrolling WARPd as a NN. Unrolling iterative algorithms as NNs is an increasingly popular method \cite{mccann2017convolutional,monga2019algorithm} and is particularly well-suited to scenarios where it is difficult to collect large training samples. Naive unrolling of first-order iterative methods typically provides slow $\mathcal{O}(\delta+n^{-1})$ (or $\mathcal{O}(\delta+n^{-2})$ in certain regimes) convergence guarantees in the number of hidden layers $n$.\footnote{There are exceptions, such as \cite{chen2018theoretical, liu2018alista} for LISTA (a learned version of ISTA) ensuring the existence of NN with linear convergence towards the minimizer. Yet, neither \cite{chen2018theoretical} nor \cite{liu2018alista} use the theoretically correct weights, as these can only be computed as solutions of intractably large optimization problems. It is also unclear whether the needed assumptions on $A$ hold in practice.} Instead, we gain convergence $\mathcal{O}(\delta+\exp(-n))$, providing lower bounds on what is achievable in terms of stability and accuracy of a NN. The following theorem provides the approximation theory result.\footnote{At no point is this paper do we train a neural network.}

\begin{theorem}
\label{main_theorem_NN}
Let $L$ be an upper bound for $\sqrt{\|A\|^2+\|B\|^2}$ and $\delta>0$. Suppose that \eqref{assumption} holds and that the proximal map of $\mathcal{J}$ can be approximated to the required accuracy described by $\mu$ in \eqref{stab_bound} via a NN of width bounded by a constant times $m+N+q$ and depth $M$. We provide a NN $\phi$ of width bounded by a constant times $m+N+q$ and depth bounded by a constant times $MLC_1\sqrt{C_2^2+q}\cdot\log\left({C_2\|b\|_{l^2}}/{\delta}\right)$ such that the following uniform stable recovery guarantee holds. For any pair $(\varkappa,b)\in\mathbb{C}^{N}\times\mathbb{C}^m$ such that $\|A\varkappa-b\|\leq \epsilon$ and $c(\varkappa,b)\leq\delta$,
\begin{equation}\setlength\abovedisplayskip{5pt}\setlength\belowdisplayskip{5pt}
\label{key_theorem_bound_NN}
\|\phi(b)- \varkappa\|_{l^2}\lesssim C_1\delta.
\end{equation}
\end{theorem}

The key points are: (a) the total number of parameters and depth of the NN only depend logarithmically on the error tolerance $\delta$ (\textit{accuracy} and \textit{efficiency}), and (b) the recovery guarantee is stable in the $l^2$-norm (in terms of the data $b$ and the bound $c(\varkappa,b)\leq\delta$) for the model class described by $c(\varkappa,b)\leq\delta$ (\textit{stability}). This result provides lower bounds for what is achievable in terms of stable and accurate neural networks.

Regarding the approximation of the proximal map of $\mathcal{J}$, for the examples in \cref{application1,application3} this can be achieved exactly using a fixed depth (so we can take $M=\mathcal{O}(1)$ in \cref{main_theorem_NN}). For the examples of low-rank matrix recovery and matrix completion in \cref{application2,sec:mat_com_exm}, the proximal map is computed via a partial singular value decomposition. This is typically achieved via iterative methods, which can be unrolled as recurrent NNs. The precise number of iterations is heavily dependent on the matrix and singular values/vectors that are sought. Finally, we point out (see also \cref{sec:complexity}) that one can obtain similar results where the matrices $A$ and $B$ are only known approximately, and the non-linear maps in each layer are only applied approximately.

\subsection{Connections with previous work}
\label{sec:prev_work}

Additional to this section, we provide connections with previous work that are specific to each of \cref{application1,application2,sec:mat_com_exm,application3} throughout the paper. We do not cover here the vast literature on NN techniques, which was discussed in \cref{sec:NN_unravel_alg}.

\textbf{First-order methods:} There are numerous specialized algorithms for various instances of \eqref{main_problem} and closely related problems \cite{beck2009fast,becker2011nesta,daubechies2004iterative,figueiredo2007gradient,van2009probing}, as well as general-purpose solvers \cite{becker2011templates}. A common approach is to apply some form of smoothing and use Nesterov's acceleration \cite{nesterov2005smooth}, which achieves an objective function suboptimality of $\delta$ in $\mathcal{O}(\delta^{-1/2})$ steps for the smoothed problem, in combination with techniques such as continuation for the smoothing parameter. Higher values of smoothing improve numerical performance of underlying solvers but at the expense of accuracy, and balancing this precise trade-off is difficult \cite{becker2011templates}. We will not attempt to survey this vast area but point the reader to \cite{chambolle2016introduction,beck2017first}. More generally, the complexity of first-order methods is usually controlled by smoothness assumptions on the objective function, such as Lipschitz continuity of its gradient. Additional assumptions on the objective function such as strong and uniform convexity provide, respectively, linear and faster polynomial rates of convergence \cite{nesterov2003introductory}. For example, using variants of the classical strong convexity assumption, linear convergence results have been obtained for LASSO \cite{agarwal2012fast,zhou2015ell_1}. However, strong or uniform convexity are often too restrictive in many applications. For results on asymptotic linear convergence of standard methods (e.g., proximal gradient method) for certain continuously differentiable (but non strongly convex) objective functions, see \cite{liang2014local,necoara2019linear,zhou2017unified}.

\textbf{\L{}ojasiewicz-type inequalities:} Achieving linear convergence for restarted first-order methods typically requires a \L{}ojasiewicz-type or ``sharpness'' inequality such as
\begin{equation}\setlength\abovedisplayskip{5pt}\setlength\belowdisplayskip{5pt}
\label{loja}\gamma d(\hat x,X^*)^{\nu}\leq f(\hat x)-f^*,
\end{equation}
also known as a H\"olderian error bound, with knowledge of $\gamma$ and $\nu$ \cite{freund2018new,roulet2020computational,roulet2020sharpness}. Here $f$ is the objective function (with optimal value $f^*$) and $d(\cdot,X^*)$ denotes the distance to the set of minimisers. For example, Nemirovskii and Nesterov \cite{nemirovskii1985optimal} linked a ``strict minimum'' condition similar to \eqref{loja} with faster convergence rates using restart schemes for smooth objective functions. H\"olderian error bounds where first introduced by Hoffman \cite{hoffman1952approximate} to study systems of linear inequalities, and extended  to convex optimization in \cite{robinson1975application,mangasarian1985condition,auslender1988global,burke1993weak,burke2002weak}. \L{}ojasiewicz proved that \eqref{loja} holds generically for real analytic and subanalytic functions \cite{lojasiewicz1963propriete}, which was extended to nonsmooth subanalytic convex functions by Bolte, Daniilidis, and Lewis \cite{bolte2007lojasiewicz}.

There is, however, a key difference between \eqref{assumption} and \eqref{loja}, and hence also between the restart scheme of this paper and the above cited work. In \eqref{assumption}, we only assume \textit{approximate} control of the distance via the objective function difference - this is reflected by the parameter $\delta$ in \cref{main_theorem} and the term $c(x,b)$ in \eqref{assumption}. For the type of problems we consider, this gives us greater generality and allows us to tackle the case of \textit{noisy measurements}, as well as prove \textit{robustness} of our results (e.g., when considering sparse recovery, we cover approximately sparse vectors). However, it also means that the vector $\varkappa$ can only be recovered approximately to order $\delta$. Curiously, numerical experiments below demonstrate that we continue to achieve linear convergence to a solution of \cref{main_problem}, which suggests that a combination of restarting and reweighting can take advantage of properties analogous to \eqref{assumption} around minimizers.

For further use of \L{}ojasiewicz-type inequalities for first-order methods (e.g., assessing asymptotic rates of convergence), see \cite{bolte2014proximal,bolte2017error,attouch2010proximal,frankel2015splitting}. Further works on restart schemes include \cite{fercoq2016restarting}, which showed that generic restart schemes can offer linear convergence given a rough estimate of the behavior of the function around its minimizers, and \cite{o2015adaptive}, which developed a heuristic analysis for restarts based on ripples or bumps in the trace of the objective value.

\textbf{The example of sparse recovery:} The use of \eqref{assumption} is closely related to \cite{benlectures}, who were one of the first to realize how key assumptions in compressed sensing -- such as the \emph{robust nullspace property} -- help bound the error of the approximation to a minimizer (produced by an optimization algorithm) in terms of error bounds on the approximation to the objective function. For example, \cite{roulet2020computational} achieves linear convergence, using the restarted NESTA algorithm \cite{becker2011nesta}, for exact recovery (noiseless) of real-valued sparse vectors if $A$ satisfies the null space property of order $s$. Under this assumption, if $x$ is $s$-sparse and $A\hat x=Ax$, then one has
\begin{equation}\setlength\abovedisplayskip{5pt}\setlength\belowdisplayskip{5pt}
\label{jbakdcabv}
\|\hat x-x\|\lesssim \|\hat x\|_1-\|x\|_1.
\end{equation}
The restart scheme in \cite{roulet2020computational} is based on a careful reduction in the smoothing parameter, chosen by analyzing a combination of the error bounds for NESTA and \eqref{jbakdcabv}. Though our methods are completely different (e.g., we do not rely on smoothing, and we must take into account the additional error term owing to the approximate \L{}ojasiewicz-type inequality), for the specific case of sparse recovery discussed in \cref{application1}, our results can be considered a generalization of \cite{roulet2020computational} to allow measurement noise, approximate sparsity, and structured compressed sensing.

Finally, the author of the current paper developed a simplified version of the restart scheme used in WARPdSR (see \cref{sqre_root}) based on the Square-Root LASSO decoder for the specific case of sparse recovery ($B=0$ and $\mathcal{J}(x)=\|x\|_{l^1_w}$, see \cref{application1}) from Fourier and binary measurements in \cite{colbrookNN2021can}. The outcome was stable and accurate NNs, where unrolled iterations led to Fast Iterative REstarted NETworks (FIRENETs). \cref{main_theorem_NN} continues in this direction and provides foundations for stable and accurate NNs for a much wider class of problems. It was also shown in \cite{colbrookNN2021can} that there are fundamental computability barriers for solving $l^1$ minimization if conditions such as \eqref{assumption} are not met (here, we mean computing a minimizing vector as opposed to vectors that nearly minimize the objective function).

\textbf{Primal-dual algorithms:} WARPd uses iterations of Chambolle and Pock's primal-dual algorithm \cite{chambolle2016ergodic,chambolle2011first} and a novel restart-reweight scheme. The primal-dual hybrid gradient (PDHG) algorithm is a popular method to solve saddle point problems \cite{esser2010general,pock2009algorithm,chambolle2018stochastic}. The linear convergence of primal-dual methods under different conditions is widely studied. For example, see \cite{daskalakis2017training} for bilinear problems (with a focus on training GANs) and \cite{valkonen2017acceleration} for partially strongly convex functions. Recently, \cite{applegate2021faster} developed an adaptive restart scheme for PDHG applied to linear programming and showed linear convergence.

\subsection{Notation} We use $\|\cdot\|_{l^p}$ to denote the standard $l^p$-norm of vectors and $\langle \cdot,\cdot\rangle$ to denote the standard inner product on $\mathbb{C}^n$. Given a linear operator $A$ between Banach spaces, we denote the operator norm of $A$ by $\|A\|$. For a seminorm $\mathcal{J}$, we define $\|\mathcal{J}\|=\sup_{\|x\|_{l^2}=1}\|\mathcal{J}(x)\|$. Given a lower semi-continuous convex function $f$ from a Hilbert space $\mathcal{H}$ to $[-\infty,\infty]$, we use the proximal operator $\mathrm{prox}_f(v)=\mathrm{argmin}_{x\in\mathcal{H}}f(x)+\frac{1}{2}\|v-x\|^2$. Throughout, $a\lesssim b$ will mean there is a constant  $C$ (independent of all relevant parameters) such that $a\leq Cb$. Finally, given a matrix $M$ with singular values $\sigma_1(M)\geq \sigma_2(M)\geq...\geq \sigma_{r}(M)$, we denote by $\|M\|_{p}$ the Schatton $p$-norm of $M$, which is the $l^p$-norm of the sequence of singular values $\{\sigma_j(M)\}$.

\subsection{Outline of paper}

In \cref{THEALG}, we introduce WARPd, prove its convergence properties (e.g., \cref{main_theorem}), discuss its computational complexity and practice, provide a variation (WARPdSR) suitable for noise-blind recovery problems (unknown $\epsilon$) and prove \cref{main_theorem_NN}. \Cref{application1} analyzes the example of sparse recovery, \cref{application2} analyzes the example of (approximately) low-rank matrix recovery, \cref{sec:mat_com_exm} analyzes the example of matrix completion, and \cref{application3} analyzes the examples of $l^1$-analysis and total variation minimization. Numerical examples are given throughout the paper and code is available at \url{https://github.com/MColbrook/WARPd}. For brevity, proofs of the theoretical results we derive in \cref{application1,application2,sec:mat_com_exm,application3} as well as of \cref{main_theoremb} are given in the supplementary materials.

\section{The accelerated algorithm}
\label{THEALG}
We begin with the primal-dual iterations in \cref{theory1} and then describe the restart scheme in \cref{theory2}. \cref{main_theorem} provides the error bounds for WARPd described in \cref{alg:RR_primal_dual}. Computational considerations are given in \cref{sec:complexity} and we prove \cref{main_theorem_NN} in \cref{sec:NN_unn}. In \cref{sqre_root}, we provide a variation, WARPdSR, based on replacing the constraint $\|Ax-b\|_{l^2}\leq \epsilon$ in \eqref{main_problem} with an additional data fitting term $\|Ax-b\|_{l^2}$ in the objective function. This is well suited to noise-blind recovery problems (unknown $\epsilon$) and provides an elegant means to bound dual variables for warm restarts. Finally, we discuss extensions in \cref{sec:extensions}.

\subsection{Primal-dual iterations}
\label{theory1}
Our starting to point is to recast the problem \eqref{main_problem} as an equivalent saddle point problem. Let $u=[\mathrm{real}(x); \hspace{1mm}\mathrm{imag}(x)]=[x_1; \hspace{1mm}x_2]\in\mathbb{R}^{2N}$ denote the primal variable so that $x=x_1+ix_2$, and $\hat b=[\mathrm{real}(b); \hspace{1mm}\mathrm{imag}(b)]\in\mathbb{R}^{2m}$. Define the matrices
$$\setlength\abovedisplayskip{5pt}\setlength\belowdisplayskip{5pt}
K_1\!=\!\begin{pmatrix}
\mathrm{real}(A)\!\!\!\!    & -\mathrm{imag}(A)\\
\mathrm{imag}(A)\!\!\!\!    & \mathrm{real}(A)
\end{pmatrix}\!\in\!\mathbb{R}^{2m\times 2N},\!\!  \quad K_2\!=\!\begin{pmatrix}
\mathrm{real}(B)\!\!\!\!    & -\mathrm{imag}(B)\\
\mathrm{imag}(B)\!\!\!\!   & \mathrm{real}(B)
\end{pmatrix}\!\in\!\mathbb{R}^{2q\times 2N},\!\!  \quad K\!=\!\begin{pmatrix}
K_1\\
K_2
\end{pmatrix}.
$$
$K_1$ and $K_2$ correspond to $A$ and $B$, when viewed as linear maps on the corresponding real vector spaces, and $\|K_1\|=\|A\|$, $\|K_2\|=\|B\|$, $\|K\|\leq \sqrt{\|A\|^2+\|B\|^2}$. Finally, let
$$\setlength\abovedisplayskip{5pt}\setlength\belowdisplayskip{5pt}
j(u)=\mathcal{J}(x_1+ix_2), \quad \|u\|_{l^1_{\mathbb{C}}}=\|x_1+ix_2\|_{l^1}.
$$
With this notation in hand, the problem \eqref{main_problem} is equivalent to
$$\setlength\abovedisplayskip{5pt}\setlength\belowdisplayskip{5pt}
\min_{u\in\mathbb{R}^{2N}}j(u)+\|K_2u\|_{l^1_{\mathbb{C}}}\quad \text{s.t.} \quad\|K_1u-\hat b\|_{l^2}\leq \epsilon.
$$
Via dualizing the feasibility condition and $\|K_2u\|_{l^1_{\mathbb{C}}}$, we obtain the saddle point problem
\begin{equation}\setlength\abovedisplayskip{5pt}\setlength\belowdisplayskip{5pt}
\label{saddleiighujb}
\min_{u\in\mathbb{R}^{2N}}\max_{y_1\in\mathbb{R}^{2m},y_2\in\mathbb{R}^{2q}}\mathcal{L}(u,y) := \langle K_1u-\hat b,y_1 \rangle +\langle K_2u,y_2 \rangle+ j(u) - \epsilon\|y_1\|_{l^2}-\chi_{\mathcal{B}_\infty}(y_2).
\end{equation}
Here, $\langle\cdot,\cdot\rangle$ denotes the usual inner product, $\chi_S$ denotes the indicator function of a set $S$ (taking the value $0$ on $S$ and $+\infty$ otherwise) and $\mathcal{B}_\infty$ denotes the complex closed unit $l^\infty$ ball.

To solve \eqref{saddleiighujb}, we use a primal-dual algorithm \cite{chambolle2016ergodic,chambolle2011first}. We use $y=(y_1,y_2)^\top$ to denote the dual variables and start by setting $u^{(0)}=u_0$ and $y^{(0)}=0$. We then iterate via
\begin{equation}\setlength\abovedisplayskip{5pt}\setlength\belowdisplayskip{5pt}\label{eq:iterCP}
\begin{split}
u^{(k+1)}&\! =\! \argmin_{u \in \mathbb{R}^{2N}} j(u) +\frac{1}{2\tau_1}\|u-(u^{(k)}-\tau_1 K^* y^{(k)})\|^2_{l^2}\! =\! \mathrm{prox}_{\tau_1j}(u^{(k)}-\tau_1 K^* y^{(k)})\\
y^{(k+1)}&\! =\! \argmin_{y \in \mathbb{R}^{2m+2q}} \!\epsilon\|y_1\|_{l^2}\!+\!\langle \hat b,y_1\rangle\! +\!\chi_{\mathcal{B}_\infty}\!(y_2) \!+\!\frac{1}{2\tau_2}\|y\!-\! y^{(k)}\!-\!\tau_2 K (2u^{(k+1)}\!-\! u^{(k)})\|_{l^2}^2\\
&\! =\! \left(\gamma_{\tau_2\epsilon}(y^{(k)}_1+\tau_2 K_1 (2u^{(k+1)}-u^{(k)})-\tau_2 \hat b),\varsigma_1(y^{(k)}_2+\tau_2 K_2 (2u^{(k+1)}-u^{(k)}))\right)^\top,
\end{split}
\end{equation}
where $\tau_1,\tau_2>0$ denote proximal step sizes and we define the functions
\begin{equation}\setlength\abovedisplayskip{5pt}\setlength\belowdisplayskip{5pt}
\label{simple_function1}
\gamma_\rho(y_1):=\max\left\{0,1-{\rho}/{\|y_1\|_{l^2}}\right\}y_1, \quad [\varsigma_\rho(z)]_j=\min\left\{1,{\rho}/{|z_j|}\right\}z_j,
\end{equation}
for $z=y_2$ written in complex form. To obtain the final line in \eqref{eq:iterCP}, we use the well-known proximal maps of the $l^2$-norm and $\chi_{\mathcal{B}_\infty}$. We use $\mathrm{PD}_{\tau_1,\tau_2}$ to denote the exact updates so that
\begin{equation}\setlength\abovedisplayskip{5pt}\setlength\belowdisplayskip{5pt}\label{PDE_dfgg}
(x^{(k+1)},y^{(k+1)})=\mathrm{PD}_{\tau_1,\tau_2}(x^{(k)},y^{(k)}).\end{equation}
For notational convenience, we define
\begin{equation}\setlength\abovedisplayskip{5pt}\setlength\belowdisplayskip{5pt}
\label{def:GGG}
G_{\eta}(\hat x,x,b):=\underbrace{\mathcal{J}(\hat x)+\|B\hat x\|_{l^1}-\mathcal{J}(x)-\|B x\|_{l^1}}_{\text{objective function difference}}+\eta\underbrace{\left(\|A\hat x-b\|-\epsilon\right)}_{\text{feasibility gap}},
\end{equation}
for multiplier $\eta\geq0$. Note that \eqref{assumption} allows us to bound the distance between vectors $x$ and $\hat x$ in terms of $G_{C_2}(\hat x,x,b)$ and $c(x,b)$. Hence we would like to control the size of $G_{C_2}(\hat x,x,b)$. As a first step, \cref{approx_PD_alg} provides an explicit bound for $G_{\eta}$ (see \eqref{p3_key_bd}) for inexact (see \eqref{p3_key_bd0}) primal-dual updates without restarts. We treat inexact updates to provide stability results, cover the case of inexact information regarding $A$ and $B$, and to cover cases where the proximal map of $\mathcal{J}$ is applied approximately (see \cref{sec:mat_alg_comp}).

\begin{theorem}[Stable bounds on $G_{\eta}$ for inexact primal-dual updates]
\label{approx_PD_alg}
Suppose that the step sizes $\tau_1$ and $\tau_2$ satisfy $\tau_1\tau_2({\|A\|^2+\|B\|^2})< 1$. Let $x_0\in \mathbb{C}^m$ and $u^{(0)}=[\mathrm{real}(x_0); \hspace{1mm}\mathrm{imag}(x_0)]$. Set $\tilde u^{(0)}=u^{(0)}$ and $\tilde y^{(0)}=y^{(0)}=0$. Suppose that $\left(\tilde u^{(k)},\tilde y^{(k)}\right)$ are such that
$$\setlength\abovedisplayskip{5pt}\setlength\belowdisplayskip{5pt}
\|(\tilde u^{(k)},\tilde y^{(k)})^\top-\mathrm{PD}_{\tau_1,\tau_2}(\tilde u^{(k-1)},\tilde y^{(k-1)})^\top\|_{l^2}\leq \epsilon_k,\quad k\geq 1.
$$
In other words, each primal-dual iterate is approximately applied/computed to accuracy $\epsilon_k$. Define the four ergodic averages (where $(u^{(k)},y^{(k)})$ denote the exact updates)
$$\setlength\abovedisplayskip{5pt}\setlength\belowdisplayskip{5pt}
U^{(n)}=\frac{1}{n}\sum_{k=1}^nu^{(k)},\quad \tilde U^{(n)}=\frac{1}{n}\sum_{k=1}^n \tilde u^{(k)}, \quad Y^{(n)}=\frac{1}{n}\sum_{k=1}^n y^{(k)},\quad \tilde Y^{(n)}=\frac{1}{n}\sum_{k=1}^n \tilde y^{(k)},
$$
and let $X_n,\tilde X_n\in\mathbb{C}^N$ denote the complexifications of $U^{(n)},\tilde U^{(n)}\in\mathbb{R}^{2N}$ respectively. Then
\begin{equation}\setlength\abovedisplayskip{5pt}\setlength\belowdisplayskip{5pt}
\label{p3_key_bd0}
\left\|\left(X_n-\tilde X_n,Y^{(n)}-\tilde Y^{(n)}\right)^\top\right\|_{l^2}\!\leq \left[\sqrt{\frac{\tau_1+\tau_2}{1-\tau_1\tau_2 (\|{A}\|^2+\|B\|^2)}}\sqrt{\tau_1^{-1}+\tau_2^{-1}}\right]\frac{1}{n}\sum_{k=1}^n\sum_{j=1}^k\epsilon_j.
\end{equation}
Moreover, for any $\eta\geq 0$ and any feasible $x\in\mathbb{C}^{N}$ (i.e., $\|Ax-b\|_{l^2}\leq\epsilon$),
\begin{equation}\setlength\abovedisplayskip{5pt}\setlength\belowdisplayskip{5pt}
\label{p3_key_bd}
 \underbrace{\mathcal{J}(X_n)\! +\! \|BX_n\|_{l^1}\! -\! \mathcal{J}(x)\! -\! \|Bx\|_{l^1}\! +\! \eta \left(\|AX_n-b\|_{l^2}-\epsilon\right)}_{G_{\eta}(X_n,x,b)}\leq\frac{1}{n}\left(\frac{\|x_0-x\|_{l^2}^2}{\tau_1}+\frac{\eta^2+q}{\tau_2}\right).
\end{equation}
\end{theorem}

\begin{proof}
Recall first the definition of $\mathcal{L}$ in \eqref{saddleiighujb}. Since $\tau_1\tau_2\|K\|^2\leq \tau_1\tau_2({\|A\|^2+\|B\|^2})< 1$, Theorem 1 and remark 2 of \cite{chambolle2016ergodic} show that for any $u\in\mathbb{R}^{2N}$ and any $y\in\mathbb{R}^{2m+2q}$,
\begin{equation}\setlength\abovedisplayskip{4pt}\setlength\belowdisplayskip{4pt}
\label{saddle_bound}
\mathcal{L}\left(U^{(n)},y\right)-\mathcal{L}\left(u,Y^{(n)}\right)\leq \frac{\|u_0-u\|_{l^2}^2}{n\tau_1}+\frac{\|y_0-y\|_{l^2}^2}{n\tau_2}.
\end{equation}
Let $x=x_1+ix_2$ be feasible and $y_1$ be parallel to $K_1U^{(n)}-\hat b$ with $\|y_1\|_{l^2}=\eta$. Writing out the difference on the left-hand side of \eqref{saddle_bound} and simplifying, we see that \eqref{saddle_bound} now yields
\begin{equation}\setlength\abovedisplayskip{4pt}\setlength\belowdisplayskip{4pt}\label{saddle_bound0}
\begin{split}
&\eta\!\left(\!\|K_1U^{(n)}-\hat b\|_{l^2}-\epsilon\!\right)+\mathcal{J}(X_n)-\langle K_1u-\hat b,Y^{(n)}_1\rangle-\mathcal{J}(x)+\epsilon\|Y^{(n)}_1\|_{l^2}\\
&\quad\quad\!+\langle K_2U^{(n)},y_2 \rangle-\chi_{\mathcal{B}_\infty}(y_2)-\langle K_2u,Y^{(n)}_2 \rangle+\chi_{\mathcal{B}_\infty}(Y^{(n)}_2)\leq \frac{\|u_0-u\|_{l^2}^2}{n\tau_1}\!+\!\frac{\eta^2+\|y_2\|_{l^2}^2}{n\tau_2}.
\end{split}
\end{equation}
Choose $y_2$ of complex $l^\infty$-norm one such that $\langle K_2U^{(n)},y_2 \rangle=\|BX_n\|_{l^1}.$ To see why this is possible, note that $\hat u\mapsto\|\hat u\|_{l^1_{\mathbb{C}}}$ is convex and lower semi-continuous. Hence, by the Fenchel--Moreau theorem, it is equal to its biconjugate and (after composing with a linear map)
\begin{equation}\setlength\abovedisplayskip{4pt}\setlength\belowdisplayskip{4pt}
\label{FM_theorem}
\|B\hat x\|_{l^1} = \sup_{y_2\in\mathbb{R}^{2q}} \langle K_2\hat u,y_2\rangle-\chi_{\mathcal{B}_\infty}(y_2).
\end{equation}
Since $\|y_2\|_{l^2}^2\leq q$, it follows that \eqref{saddle_bound0} reduces to
\begin{equation}\setlength\abovedisplayskip{4pt}\setlength\belowdisplayskip{4pt}\label{saddle_bound0000}
\begin{split}
&\eta\left(\|K_1U^{(n)}-\hat b\|_{l^2}-\epsilon\right)+\mathcal{J}(X_n)+\|BX_n\|_{l^1}-\mathcal{J}(x)\\
&\quad-\langle K_1u-\hat b,Y^{(n)}_1\rangle+\epsilon\|Y^{(n)}_1\|_{l^2}-\langle K_2u,Y^{(n)}_2 \rangle+\chi_{\mathcal{B}_\infty}(Y^{(n)}_2)\!\leq\! \frac{\|u_0-u\|_{l^2}^2}{n\tau_1}+\frac{\eta^2+q}{n\tau_2}.
\end{split}
\end{equation}
By the Cauchy--Schwartz inequality, since $\|K_1u-\hat b\|_{l^2}\leq\epsilon$ ($x$ is feasible), $
-\langle K_1u-\hat b,Y^{(n)}_1\rangle+\epsilon\|Y^{(n)}_1\|_{l^2}\geq 0.$ Moreover, since $\chi_{\mathcal{B}_\infty}(Y^{(n)}_2)$ must be finite and using \eqref{FM_theorem}, $-\langle K_2u,Y^{(n)}_2 \rangle\geq -\|Bx\|_{l^1}.$ Hence \eqref{saddle_bound0000} reduces to \eqref{p3_key_bd} upon complexification and it suffices to prove \eqref{p3_key_bd0}.

Let $v=(u,y)^{\top}$, and define the matrix (acting on the vectorized form of the variables)
$$\setlength\abovedisplayskip{5pt}\setlength\belowdisplayskip{5pt}
M_{\tau_1\tau_2}=\begin{pmatrix}
\frac{1}{\tau_1}I & -K^*\\
-K & \frac{1}{\tau_2}I
\end{pmatrix}\in\mathbb{R}^{2(m+q+N)\times 2(m+q+N)}.
$$
$M_{\tau_1\tau_2}$ is positive definite by the assumption $\tau_1\tau_2 ({\|{A}\|^2+\|B\|^2})<1$, and hence induces a norm denoted by $\|\cdot\|_{\tau_1\tau_2}$. We can write the iterations defined by $\mathrm{PD}_{\tau_1,\tau_2}$ as (see \cite[Sec.\ 3]{chambolle2016ergodic})
$$\setlength\abovedisplayskip{5pt}\setlength\belowdisplayskip{5pt}
0\in \mathcal{T}v^{(k+1)}+(v^{(k+1)}-v^{(k)}), \quad \text{with}\quad\mathcal{T}:=M_{\tau_1\tau_2}^{-1}\begin{pmatrix}
\partial j & K^* \\
-K & \partial h^* 
\end{pmatrix},
$$
where $h^*(y)=\epsilon\|y_1\|_{l^2}+\langle \hat b,y_1\rangle+\chi_{\mathcal{B}_\infty}(y_2)$. It follows that
$$\setlength\abovedisplayskip{5pt}\setlength\belowdisplayskip{5pt}
v^{(k+1)}=\left[I+\mathcal{T}\right]^{-1}v^{(k)}.
$$
The multi-valued operator $\mathcal{T}$ is maximal monotone with respect to the inner product induced by $M_{\tau_1\tau_2}$ \cite{rockafellar1976monotone}, and hence the iterations are non-expansive in the norm $\|\cdot\|_{\tau_1\tau_2}$.

We have that
$$\setlength\abovedisplayskip{5pt}\setlength\belowdisplayskip{5pt}
\|(u,y)^{\top}\|^2_{\tau_1\tau_2}\leq\frac{\|u\|_{l^2}^2}{\tau_1}+\frac{\|y\|_{l^2}^2}{\tau_2}+2\|{K}\|\|u\|_{l^2}\|y\|_{l^2}\leq \left(\frac{\|{K}\|}{\nu}+\tau_1^{-1}\right)\|u\|_{l^2}^2+\left(\|{K}\|\nu+\tau_2^{-1}\right)\|y\|_{l^2}^2,
$$
for any $\nu> 0$ by the AM--GM inequality. Choosing $\nu=\tau_2 \|{K}\|$ and using $\tau_1\tau_2 \|{K}\|^2<1$,
\begin{equation}\setlength\abovedisplayskip{5pt}\setlength\belowdisplayskip{5pt}
\label{equiv_norm1}
\|(u,y)^{\top}\|^2_{\tau_1\tau_2}\leq(\tau_1^{-1}+\tau_2^{-1})\|(u,y)^{\top}\|_{l^2}^2.
\end{equation}
A similar calculation yields that
\begin{equation}\setlength\abovedisplayskip{5pt}\setlength\belowdisplayskip{5pt}
\label{equiv_norm2}
\|(u,y)^{\top}\|_{l^2}^2\leq \frac{\tau_1+\tau_2}{1-\tau_1\tau_2 \|{K}\|^2}\|(u,y)^{\top}\|^2_{\tau_1\tau_2}.
\end{equation}
It follows that
\begin{align*}\setlength\abovedisplayskip{5pt}\setlength\belowdisplayskip{5pt}
&\|\!(u^{(k)},y^{(k)})^{\top}-(\tilde u^{(k)},\tilde y^{(k)})^{\top}\!\|_{\tau_1\tau_2}\!\!\leq \|(\tilde u^{{(k)}},\tilde y^{(k)})^{\top}-\mathrm{PD}_{\tau_1,\tau_2}(\tilde u^{(k-1)},\tilde y^{(k-1)})^{\top}\|_{\tau_1\tau_2}\\
&\quad\quad\quad\quad\quad\quad\quad\quad\quad\quad\quad\quad\quad\quad+\|\mathrm{PD}_{\tau_1,\tau_2}(u^{(k-1)},y^{(k-1)})^{\top}-\mathrm{PD}_{\tau_1,\tau_2}(\tilde u^{(k-1)},\tilde y^{(k-1)})^{\top}\|_{\tau_1\tau_2}\\
&\quad\leq \|(\tilde u^{{(k)}},\tilde y^{(k)})^{\top}-\mathrm{PD}_{\tau_1,\tau_2}(\tilde u^{(k-1)},\tilde y^{(k-1)})^{\top}\|_{\tau_1\tau_2} +\|(u^{(k-1)},y^{(k-1)})^{\top}-(\tilde u^{(k-1)},\tilde y^{(k-1)})^{\top}\|_{\tau_1\tau_2}\\
&\quad\leq \epsilon_k\sqrt{\tau_1^{-1}+\tau_2^{-1}}\!+\!\|\!(u^{(k-1)},y^{(k-1)})^{\top}\!-\!(\tilde u^{(k-1)},\tilde y^{(k-1)})^{\top}\!\|_{\tau_1\tau_2},
\end{align*}
where we have used the triangle inequality in the first inequality and the fact that the iterates are non-expansive in $\|\cdot\|_{\tau_1\tau_2}$ in the second inequality. Iterating and using \eqref{equiv_norm2}, we have
$$\setlength\abovedisplayskip{5pt}\setlength\belowdisplayskip{5pt}
\|(u^{(k)},y^{(k)})^{\top}-(\tilde u^{(k)},\tilde y^{(k)})^{\top}\|_{l^2}\leq \sqrt{\frac{\tau_1+\tau_2}{1-\tau_1\tau_2 \|{K}\|^2}}\sqrt{\tau_1^{-1}+\tau_2^{-1}}\sum_{j=1}^k\epsilon_j.
$$
Since $\|K\|^2\leq\|A\|^2+\|B\|^2$, this proves \eqref{p3_key_bd0} and hence finishes the proof of the theorem.
\end{proof}

\subsection{The restart scheme}
\label{theory2}

With \cref{approx_PD_alg} in hand, we can now describe the accelerated scheme. The idea is to take advantage of the different orders of positive homogeneity on either side of \eqref{p3_key_bd}. Together with \eqref{assumption}, this allows a decrease in the relevant gap $G_{\eta}$ (defined in \eqref{def:GGG}) by a constant factor for a fixed number of iterations. For convenience, define
$$\setlength\abovedisplayskip{5pt}\setlength\belowdisplayskip{5pt}
\widehat{\mathcal{J}}(x):=\mathcal{J}(x)+\|Bx\|_{l^1}.
$$
We begin by describing three steps used to obtain a key inequality \eqref{ProofC2}.

\vspace{1mm}

\textbf{Step 1:} First, consider the iterations described in \cref{approx_PD_alg} with $n=k$, but with rescaled input $b/\beta$ and $x_0/\beta$, and $\epsilon$ in the theorem rescaled to $\epsilon/\beta$ for a given $k \in \mathbb{N}$, and $\beta>0$ (both of which are explicitly defined below). We assume that each of the $\epsilon_j\leq \mu$ for some $\mu>0$ and denote the corresponding map (the computed $\tilde X_{k}$) as
$$\setlength\abovedisplayskip{5pt}\setlength\belowdisplayskip{5pt}
\Psi_k=\Psi_k\left(\frac{b}{\beta}, \frac{x_0}{\beta},\frac{\epsilon}{\beta}\right).
$$
\cref{approx_PD_alg} ensures the existence of a vector $\psi_k$ (the exact iterates $ X_{k}$) satisfying
\begin{equation}\setlength\abovedisplayskip{5pt}\setlength\belowdisplayskip{5pt}
\label{needed0}
\left\| \psi_{k} - \Psi_k  \right\|_{l^2}  \leq \left[\sqrt{\frac{\tau_1+\tau_2}{1-\tau_1\tau_2 (\|{A}\|^2+\|B\|^2)}}\sqrt{\tau_1^{-1}+\tau_2^{-1}}\right]\frac{(k+1)\mu}{2},
\end{equation}
(where we have computed the double sum in \eqref{p3_key_bd0}) along with the following bound (where we take $\eta=C_2$) for any feasible $x \in \mathbb{C}^{N}$:
\begin{equation}\setlength\abovedisplayskip{5pt}\setlength\belowdisplayskip{5pt}
\label{eq:scaled}
\widehat{\mathcal{J}}(\psi_{k}) -\widehat{\mathcal{J}}\left(\frac{x}{\beta} \right) +C_2\left(\left\| {A}\psi_k - \frac{b}{\beta}\right\|_{l^2} - \frac{\epsilon}{\beta}\right) 
\leq \frac{1}{k}\left( \frac{\|x-x_0\|_{l^2}^2}{\beta^2\tau_1} + \frac{C_2^2+q}{\tau_2}\right).
\end{equation}

\vspace{1mm}

\textbf{Step 2:} We now seek to convert the bound \eqref{eq:scaled} to a corresponding bound for $\Psi_{k}$ instead of $\psi_{k}$. Given $L\geq \sqrt{\|A\|^2+\|B\|^2}$ and $\tau\in(0,1)$, we choose $\tau_1^{-1}=\tau_2^{-1}\leq \tau^{-1}L$ with $\tau^{-1}L\lesssim\tau_1^{-1}$. Using \eqref{needed0}, we have that
$$\setlength\abovedisplayskip{5pt}\setlength\belowdisplayskip{5pt}
\widehat{\mathcal{J}}(\Psi_{k})\leq \widehat{\mathcal{J}}(\psi_{k})+\widehat{\mathcal{J}}(\Psi_{k}-\psi_{k})\leq \widehat{\mathcal{J}}(\psi_{k})+\left\|\widehat{\mathcal{J}}\right\|\left[\sqrt{\frac{\tau_1+\tau_2}{1-\tau_1\tau_2 L^2}}\sqrt{\tau_1^{-1}+\tau_2^{-1}}\right]\frac{(k+1)\mu}{2}.
$$
Since $\|A\Psi_{k}-A\psi_{k}\|\leq L\| \psi_{k} - \Psi_k  \|_{l^2}$, it follows after some tedious but straightforward calculations that a suitable $\mu$ can be chosen with the scaling
\begin{equation}\setlength\abovedisplayskip{5pt}\setlength\belowdisplayskip{5pt}
\label{stab_bound}
\mu\sim L(C_2^2+q){\sqrt{1-\tau^2}}({k^2\tau (\|\widehat{\mathcal{J}}\|+C_2L)})^{-1}, \quad \text{such that}
\end{equation}
\begin{equation}\setlength\abovedisplayskip{5pt}\setlength\belowdisplayskip{5pt}
\label{eq:scaled00}
\widehat{\mathcal{J}}(\Psi_{k}) -\widehat{\mathcal{J}}\left(\frac{x}{\beta} \right) +C_2\left(\left\| {A}\Psi_k - \frac{b}{\beta}\right\|_{l^2} - \frac{\epsilon}{\beta}\right) 
\leq \frac{L}{k\tau}\left( \frac{\|x-x_0\|_{l^2}^2}{\beta^2} + 2(C_2^2+q)\right).
\end{equation}

\vspace{1mm}

\textbf{Step 3:} Finally, we rescale the bound \eqref{eq:scaled00}. Define the map $H_k^\beta:\mathbb{C}^{m}\times\mathbb{C}^{N}\rightarrow \mathbb{C}^{N}$ by
$$\setlength\abovedisplayskip{5pt}\setlength\belowdisplayskip{5pt}
H_{k}^\beta(b,x_0)=\beta\cdot\Psi_k\left(\frac{b}{\beta},\frac{x_0}{\beta},\frac{\epsilon}{\beta}\right).
$$
Multiplying \eqref{eq:scaled00} by $\beta$ and using that the seminorm $\widehat{\mathcal{J}}$ is positive homogenous of degree $1$, 
$$\setlength\abovedisplayskip{5pt}\setlength\belowdisplayskip{5pt}
\widehat{\mathcal{J}}\left(H_{k}^\beta(b,x_0)\right) -\widehat{\mathcal{J}}\left({x} \right) +C_2\left(\left\| {A}H_{k}^\beta(b,x_0) - b\right\|_{l^2} - {\epsilon}\right)\leq \frac{L}{k\tau}\left(\frac{\|x-x_0\|_{l^2}^2}{\beta}+2{\beta}(C_2^2+q)\right).
$$
Upon combining this with \eqref{assumption} to bound $\|x-x_0\|_{l^2}^2$, we obtain the key inequality
\begin{equation}\setlength\abovedisplayskip{5pt}\setlength\belowdisplayskip{5pt}
\label{ProofC2}
G_{C_2}\left(H_{k}^\beta(b,x_0),x,b\right)\leq \frac{LC_1^2}{\tau k\beta}\left[c(x,b)+G_{C_2}(x_0,x,b)\right]^2+\frac{2{L}{\beta}}{k}(C_2^2+q)\tau^{-1}.
\end{equation}

\vspace{1mm}

Our algorithm takes advantage of this inequality for each restart. The full algorithm is described in \cref{alg:RR_primal_dual} (the inner iterations are described in \cref{alg:inner_iterations}), where, for simplicity, we have taken $\mu=0$ corresponding to exact primal-dual iterations. \cref{main_theorem} summarizes the convergence result, and the proof shows how to choose optimal $k$ and $\beta$.

\begin{algorithm}[t!]
\textbf{Input:} Data $b\in\mathbb{C}^m$, initial vector $x_0\in\mathbb{C}^N$, function handles for $A$, $A^*$, $B$ and $B^*$, number of iterations $k\in\mathbb{N}$, proximal step sizes $\tau_1>0$ and $\tau_2>0$, $\epsilon>0$, and seminorm $\mathcal{J}$. \\
\vspace{-4mm}
\begin{algorithmic}[1]
\STATE Initiate with $x^{(0)}=x_0$, $z_1^{(0)}=0\in\mathbb{C}^m$, $z_2^{(0)}=0\in\mathbb{C}^q$ and $X_0=0$.
\STATE For $j=0,...,k-1$ compute
\begin{equation*}\setlength\abovedisplayskip{5pt}\setlength\belowdisplayskip{5pt}
\begin{split}
x^{(j+1)}&=\mathrm{prox}_{\tau_1 \mathcal{J}}\left(x^{(j)}-\tau_1 A^* z_1^{(j)}-\tau_1 B^* z_2^{(j)}\right)\\
z_1^{(j+1)}&=\gamma_{\tau_2\epsilon}\!\left(z_1^{(j)}+\tau_2 A (2x^{(j+1)}-x^{(j)})-\tau_2 b\right)\!,\quad z_2^{(j+1)}=\varsigma_1\!\left(z_2^{(j)}+\tau_2 B (2x^{(j+1)}-x^{(j)})\right)\!,
\end{split}
\end{equation*}
and update the ergodic average $X_{j+1}=\frac{1}{j+1}\left(jX_{j}+x^{(j+1)}\right).$
\end{algorithmic} \textbf{Output:} $\texttt{InnerIt}\left(b,x_0,A,B,k,\tau_1, \tau_2,\epsilon,\mathcal{J}\right)=X_k.$
\caption{Inner iterations of primal-dual updates \eqref{eq:iterCP} written in complex form. The $z^{(j)}$ correspond to the complexification of the real dual vectors in \eqref{eq:iterCP}, $\gamma_\rho$ and $\varsigma_\rho$ are defined in \eqref{simple_function1}. At any one time, only the current ergodic average, two primal and four dual variables need to be stored.}\label{alg:inner_iterations}
\end{algorithm}

\begin{algorithm}[t!]
\textbf{Input:} $C_1$ and $C_2$ such that \eqref{assumption} holds, $L$ (upper bound for $\sqrt{\|A\|^2+\|B\|^2}$), $\tau\in(0,1)$, $\nu\in(0,1)$, $\epsilon>0$, $\delta>0$, function handles for $A$, $A^*$, $B$ and $B^*$, seminorm $\mathcal{J}$ and data $b\in\mathbb{C}^m$. \\
\vspace{-4mm}
\begin{algorithmic}[1]
\STATE Set $\epsilon_0=C_2\|b\|_{l^2}$. For $j=1,...,n-1$ compute $\epsilon_j=\upsilon\left(\delta+\epsilon_{j-1}\right)$.
\STATE Set $
k=\left\lceil\frac{2LC_1\sqrt{C_2^2+q}}{\upsilon\tau}\right\rceil,\quad \beta_j=\frac{C_1(\delta+\epsilon_{j-1})}{\sqrt{C_2^2+q}},\quad \text{for $j=1,...,n$ (recall that $B\in\mathbb{C}^{q\times N}$).}
$
\STATE Set $\phi_0(b)=0$ (or any other initial approximation) and for $j=1,...,n$, compute
\begin{equation}\setlength\abovedisplayskip{5pt}\setlength\belowdisplayskip{5pt}
\label{RRPD}
\phi_j(b)=\beta_j\cdot\texttt{InnerIt}\left(\frac{b}{\beta_{j}},\frac{\phi_{j-1}(b)}{\beta_{j}},A,B,k,\tau L^{-1}, \tau L^{-1},\frac{\epsilon}{\beta_j},\mathcal{J}\right).
\end{equation}
\end{algorithmic} \textbf{Output:} $\phi_n(b)\in\mathbb{C}^N$.
\caption{\textbf{WARPd:} Accelerated algorithm for the solution of \eqref{main_problem} and recovery of desired vector. We have removed the factors of $\sqrt{2}$ in the definition of $k$ and $\beta_j$ in the proof of \cref{main_theorem} (this is equivalent to taking $\mu=0$). The updates in \eqref{RRPD} correspond to restarted and reweighted primal-dual iterations (performed by the routine $\texttt{InnerIt}$).}\label{alg:RR_primal_dual}
\end{algorithm}

\begin{proof}[Proof of \cref{main_theorem}]
First, we specify the choices of $k$ and $\beta$. Recall that we restrict to $(\varkappa,b)$ with $c(\varkappa,b)\leq \delta$. Suppose that $G_{C_2}(x_0,\varkappa,b)\leq \epsilon_0$. Since $\varkappa$ is feasible, \eqref{ProofC2} becomes
$$\setlength\abovedisplayskip{5pt}\setlength\belowdisplayskip{5pt}
G_{C_2}\left(H_{k}^\beta(b,x_0),\varkappa,b\right)\leq \frac{LC_1^2}{\tau k\beta}(\delta+\epsilon_0)^2+\frac{2{L}{\beta}}{k}(C_2^2+q)\tau^{-1}.
$$
Optimizing the right-hand side leads to the choice $\beta={C_1(\delta+\epsilon_0)}/{\sqrt{2(C_2^2+q)}}$ and
$$\setlength\abovedisplayskip{5pt}\setlength\belowdisplayskip{5pt}
G_{C_2}\left(H_{k}^\beta(b,x_0),\varkappa,b\right)\leq \frac{2LC_1\sqrt{2(C_2^2+q)}}{\tau k}(\delta+\epsilon_0).
$$
For a given $\upsilon\in(0,1)$, we define
$$\setlength\abovedisplayskip{5pt}\setlength\belowdisplayskip{5pt}
k(\upsilon,\tau)=\left\lceil{2LC_1\sqrt{2(C_2^2+q)}}/({\upsilon\tau})\right\rceil,\quad \beta(\upsilon,\tau,\epsilon_0)={C_1(\delta+\epsilon_0)}/{\sqrt{2(C_2^2+q)}}.
$$
This ensures that $G_{C_2}\left(H_{k}^\beta(b,x_0),\varkappa,b\right)\leq \upsilon\left(\delta+\epsilon_0\right)
$ whenever $G_{C_2}(x_0,\varkappa,b)\leq \epsilon_0$.

We are now ready to describe the restart scheme. Note that $
G_{C_2}(0,\varkappa,b)\leq C_2\|b\|_{l^2}$. Given $n\in\mathbb{N}$, set $\epsilon_0=C_2\|b\|_{l^2}$ and for $j=1,...,n-1$ set $\epsilon_j=\upsilon\left(\delta+\epsilon_{j-1}\right).$ By summing a geometric series, this implies $
\epsilon_n\leq \frac{\upsilon\delta}{1-\upsilon}+\upsilon^nC_2\|b\|_{l^2}.$ We define $\phi_n(b)$ iteratively as follows. We set
$$\setlength\abovedisplayskip{5pt}\setlength\belowdisplayskip{5pt}
\phi_1(b)=H_{k(\upsilon,\tau)}^{\beta(\upsilon,\tau,\epsilon_0)}(b,0),\quad\phi_{j}(b)=H_{k(\upsilon,\tau)}^{\beta(\upsilon,\tau,\epsilon_{j-1})}(b,\phi_{j-1}(b))\text{ for }j=2,...,n.
$$
The choice of $\epsilon_j$ and the above argument inductively shows that $G_{C_2}\left(\phi_{j}(b),\varkappa,b\right)\leq\epsilon_j$. Hence, $G_{C_2}\left(\phi_{n}(b),\varkappa,b\right)\leq\epsilon_n\leq \frac{\upsilon\delta}{1-\upsilon}+\upsilon^nC_2\|b\|_{l^2}.$ Combining with \eqref{assumption}, we see that \eqref{key_theorem_bound} holds.
\end{proof}

For $T= kn$ inner iterations (as $\tau\uparrow 1$) and the choice of $k$ in \cref{alg:RR_primal_dual}, the error term $
\upsilon^n=\exp(Tk^{-1}\log(\upsilon))$ is equal to $\exp(T\lceil{2LC_1\sqrt{C_2^2+q}}/{\upsilon}\rceil^{-1}\log(\upsilon)).
$ If we ignore the ceiling function, the optimal choice $\upsilon=e^{-1}$ is found via differentiation. This choice yields
\begin{equation}\setlength\abovedisplayskip{5pt}\setlength\belowdisplayskip{5pt}
\label{conv_rate99}
\upsilon^n= \exp\left(-T\left\lceil{2eLC_1\sqrt{C_2^2+q}}\right\rceil^{-1}\right)
\end{equation}
and linear convergence in the total number of inner iterations $T$. Suppose that we want $\upsilon^nC_2\|b\|_{l^2}\sim\delta$ in order to balance the two terms on the right-hand side of \eqref{key_theorem_bound}, then
$$\setlength\abovedisplayskip{5pt}\setlength\belowdisplayskip{5pt}
T\sim LC_1\sqrt{C_2^2+q}\cdot\log\left({C_2\|b\|_{l^2}}/{\delta}\right),
$$
which only grows logarithmically with the precision $\delta^{-1}$, as stated after \cref{main_theorem}.

\subsection{Computational complexity and remarks}
\label{sec:complexity}

Let $C_{A}$, $C_{A^*}\!$, $C_B$ and $C_{B^*}$ denote the computational cost of applying $A$, $A^*\!\!$, $B$ and $B^*$ respectively, and let $C_\mathcal{J}$ denote the cost of applying the proximal map of $\mathcal{J}$. The cost per inner iteration of \cref{alg:RR_primal_dual} is
$$\setlength\abovedisplayskip{5pt}\setlength\belowdisplayskip{5pt}
C_{A}+C_{A^*}+C_{B}+C_{B^*}+C_\mathcal{J}+\mathcal{O}(m+N+q),
$$
since applying $\gamma_\rho$ and $\varsigma_1$ involves only vector operations. For simple $\mathcal{J}$, such as those considered in \cref{application1}, $C_\mathcal{J}=\mathcal{O}(N)$. In compressed sensing applications, it is common for $A$ to be a submatrix of a (rescaled) unitary operator that admits a fast transform for matrix-vector products. Similarly, in $l^1$-analysis problems, $B$ and $B^*$ often admit fast transforms. In this case, the cost per iteration is bounded by a small multiple of $N$ (and possibly logarithmic terms). Hence, each iteration is extremely fast. In the more general case, such as the nuclear norm in \cref{application2,sec:mat_com_exm}, where a singular value decomposition needs to be computed to apply $\mathrm{prox}_{\tau_1 \mathcal{J}}$, $C_\mathcal{J}$ can be larger than $\mathcal{O}(N)$. However, the algorithm is still scalable to large problems and competitive with state-of-the-art methods (see \cref{sec_low_rk_num,sec:mat_num_comp}).

After applying the rescaling in \cref{alg:RR_primal_dual}, the {\textit{relative error}} bound needed for the primal dual iterations in our algorithm scales no worse than ${\delta}({LC_1\sqrt{C_2^2+q}(\|\widehat{\mathcal{J}}\|+C_2L)})^{-1}.$ This is useful in scenarios where the proximal map of $\mathcal{J}$ can only be applied approximately. Moreover, in certain cases, we may not know the matrices $A$ or $B$ exactly, or they have been stored to a finite precision. We can absorb this additional error into the error bounds for inexact computation in \cref{approx_PD_alg}. In a similar fashion, all of the algorithms in this paper can be executed on a Turing machine with almost identical error bounds. This is important for the computability of solutions of \cref{main_problem} to a given accuracy (e.g., see \cite{opt_big} and its numerical experiments). However, we have taken the usual convention throughout of proving results in exact arithmetic and providing stability bounds such as \eqref{stab_bound}.

In the following sections, we discuss how to select the constants $C_1$ and $C_2$ in different scenarios. For cases where $\|A\|$ and $\|B\|$ are unknown, we use the power method (applied to $A^*A$ and $B^*B$) to find a suitable $L$. This computation incurs a one-off upfront cost which is usually only as expensive as a few iterations of \texttt{InnerIt}. Practically, we found that \cref{alg:RR_primal_dual} performed better if the initial dual variables in \texttt{InnerIt} were selected as the final dual variables of the previous operation of \texttt{InnerIt} (as opposed to zero). \cref{approx_PD_alg} can be adapted accordingly by bounding the dual variables (the only change is to the final term on the right-hand side of \eqref{p3_key_bd}). We omit the details and instead provide an alternative technique in \cref{sqre_root}, where the dual variables are bounded using the dual of a data fitting term. Before discussing this technique, we prove \cref{main_theorem_NN}.

\subsection{Unrolling \cref{alg:RR_primal_dual} as a stable and accurate NN}
\label{sec:NN_unn}
To capture standard architectures used in practice, such as skip connections, we consider the following definition of a NN. Without loss of generality and for ease of exposition, we also work with complex-valued NNs. Real-valued NNs can realize such networks by splitting into real and imaginary parts. A NN is a mapping $\phi \colon \mathbb{C}^{m}\rightarrow\mathbb{C}^{N}$ that can be written as a composition
$$\setlength\abovedisplayskip{5pt}\setlength\belowdisplayskip{5pt}
\phi(y)=[V_T\circ\rho_{T-1}\circ V_{T-1}\circ\cdots \circ V_2\circ\rho_1\circ V_1](y),\quad \text{where:}
$$
\begin{itemize}[leftmargin=15pt]
	\item Each $V_j$ is an affine map $\mathbb{C}^{N_{j-1}}\rightarrow\mathbb{C}^{N_{j}}$ given by $V_j(x)=W_jx+b_j(y)$ where $W_j\in\mathbb{C}^{N_{j}\times N_{j-1}}$ and $b_j(y)=R_jy+c_j$ are affine functions of the input $y$.
	\item Each $\rho_j\colon \mathbb{C}^{N_j}\rightarrow \mathbb{C}^{N_j}$ is one of two forms:	
	\begin{enumerate}[leftmargin=9pt]
		\item[(i)] There exists an index set $I_j\subset\{1,...,N_j\}$ such that $\rho_j$ applies a non-linear function $f_j:\mathbb{C}\rightarrow\mathbb{C}$ element-wise on the input vector's components with indices in $I_j$:
		$$\setlength\abovedisplayskip{5pt}\setlength\belowdisplayskip{5pt}
		\rho_j(x)_k=\begin{dcases}
f_j(x_k),&\text{if }k\in I_j\\
x_k,&\text{otherwise}.
\end{dcases}
		$$
		\item[(ii)] There exists a function $f_j:\mathbb{C}\rightarrow\mathbb{C}$ such that, after decomposing the input vector $x$ as $(x_0,X^{\top},Y^{\top})^{\top}$ for scalar $x_0$ and $X\in\mathbb{C}^{m_j},Y\in\mathbb{C}^{N_j-1-m_j}$, we have
		\begin{equation*}\setlength\abovedisplayskip{5pt}\setlength\belowdisplayskip{5pt}
		\rho_j:\begin{pmatrix}
x_0\\
X\\
Y
\end{pmatrix}\rightarrow\begin{pmatrix}
0\\
{f}_j(x_0)X\\
Y
\end{pmatrix}.
	\end{equation*}
	\end{enumerate}
\end{itemize}
The affine dependence of $b_j(y)$ on $y$ allows skip connections from the input to the current level as in definitions of feed-forward NNs \cite[p. 269]{shalev2014understanding}, and the above architecture has become standard \cite{jin17, hammernik2018learning,devore2020neural}. The use of non-linear functions of the form (ii) may be re-expressed using non-linear functions of the form (i) and the following standard element-wise squaring trick:
$$\setlength\abovedisplayskip{5pt}\setlength\belowdisplayskip{5pt}
{f}_j(x_0)X=\frac{1}{2} \left[[f_j(x_0)\textbf{1}+X]^2\! -\! f_j(x_0)^2\textbf{1}\! -\! X^2\right],
$$
where $\textbf{1}$ denotes a vector of ones of the same size as $X$. The key observation is that the basic operations of \cref{alg:inner_iterations} can be unrolled as NNs. For example, $x\mapsto \gamma_\rho(x)$ can be executed via (L denotes affine maps and NL non-linear maps)
\begin{equation*}\setlength\abovedisplayskip{5pt}\setlength\belowdisplayskip{5pt}
\begin{split}
\! x\! \! \xrightarrow[]{\text{L}}\!  \begin{pmatrix}
 x \\
 x 
\end{pmatrix} \! \! \xrightarrow[]{\text{NL}}\! \begin{pmatrix}
\! |x_1|^2\! \! \\
\! \vdots\! \! \\
\! |x_{m}|^2\! \! \\
\! x\! \! 
\end{pmatrix}\! \! \xrightarrow[]{\text{L}}\! \begin{pmatrix}
\sum_{j=1}^{M}|x_j|^2\\
x
\end{pmatrix}\!\!  \xrightarrow[]{\text{NL}}\! 
\begin{pmatrix}
\!  0 \! \\
\!  \max\left\{0,1\! -\! \frac{\rho}{\|x\|_{l^2}}\right\}x\!  
\end{pmatrix}\!\!  \xrightarrow[]{\text{L}}\! 
\max\! \left\{0,1\! -\! \frac{\rho}{\|x\|_{l^2}}\right\}\! x.
\end{split}
\end{equation*}
The second arrow applies pointwise modulus squaring (type (i) above), and the penultimate arrow applies a non-linear map (type (ii) above). Similarly, $\varsigma_\rho$ (as well as $\vartheta$ from \cref{alg:inner_iterationsB}) can be unrolled as NNs of a fixed depth and width of order $\mathcal{O}(m+N+q)$.

\begin{proof}[Proof of \cref{main_theorem_NN}]
Under the assumptions, we see that each iteration in \cref{alg:inner_iterations} (now with the appropriate change of parameters to encompass inexact primal-dual iterates as in the proof of \cref{main_theorem}) can be executed by a NN of width $\mathcal{O}(m+N+q)$ and depth $\mathcal{O}(M)$. This follows via the unrolling of $\gamma_\rho$ and $\varsigma_\rho$, the approximation of the proximal map of $\mathcal{J}$ and concatenation of NNs. Similarly, the rescalings and operations in \cref{alg:RR_primal_dual} can be combined into a NN. The result now follows from \cref{main_theorem} and the analysis after that shows that $\mathcal{O}(LC_1\sqrt{C_2^2+q}\cdot\log\left({C_2\|b\|_{l^2}}/{\delta}\right))$ inner iterations are required to achieve \eqref{key_theorem_bound_NN}.
\end{proof}

\subsection{Noise-blind recovery: Replacing the constraint with a data fitting term}
\label{sqre_root}

We now discuss a variation of WARPd based on the following \textit{unconstrained} optimization problem
\begin{equation}\setlength\abovedisplayskip{5pt}\setlength\belowdisplayskip{5pt}
\label{main_problem2}
\min_{x\in \mathbb{C}^N} \lambda\Big[\mathcal{J}(x)+\|Bx\|_{l^1}\Big]+\|Ax-b\|_{l^2} \quad \text{(with $\lambda>0$)}.
\end{equation}
The optimization problem in \eqref{main_problem2} differs from its LASSO-type cousin by replacing the conventional $\|Ax-b\|_{l^2}^2$ term with $\|Ax-b\|_{l^2}$. In the case of sparse recovery ($B=0$ and $\mathcal{J}(x)=\|x\|_{l^1_w}$ - see \cref{application1}), this is known as the Square-Root LASSO (SR-LASSO) decoder. It was introduced in \cite{belloni2011square}, see also \cite{belloni2014pivotal,adcock2019correcting}. In particular, SR-LASSO allows an optimal parameter choice for $\lambda$ that is \textit{independent of the noise level} \cite[Table 6.1]{adcock2021compressive} and is therefore well suited to noise-blind recovery problems. This property also holds for the algorithm we describe, WARPdSR. Moreover, there is an additional benefit. SR-LASSO allows an elegant bound on the size of dual variables, and hence allows an easier analysis with additional dual variable restarts (see discussion at the end of \cref{sec:complexity}).

\begin{algorithm}[t!]
\textbf{Input:} Data $b\in\mathbb{C}^m$, initial vector $x_0\in\mathbb{C}^N$, initial dual vector $z_0\in\mathbb{C}^{m+q}$, function handles for $A$, $A^*$, $B$ and $B^*$, number of iterations $k\in\mathbb{N}$, proximal step sizes $\tau_1>0$ and $\tau_2>0$, $\lambda>0$, and seminorm $\mathcal{J}$. \\
\vspace{-4mm}
\begin{algorithmic}[1]
\STATE Initiate with $x^{(0)}=x_0$, $z^{(0)}=z_0=(z_{1}^{(0)},z_{2}^{(0)})$, $X_0=0\in\mathbb{C}^N$ and $Z_0=0\in\mathbb{C}^{m+q}$.
\STATE For $j=0,...,k-1$ compute
\begin{equation*}\setlength\abovedisplayskip{5pt}\setlength\belowdisplayskip{5pt}
\begin{split}
x^{(j+1)}&=\mathrm{prox}_{\lambda\tau_1 \mathcal{J}}\left(x^{(j)}-\tau_1 A^* z_1^{(j)}-\tau_1 B^* z_2^{(j)}\right)\\
z_1^{(j+1)}&=\vartheta\left(z_1^{(j)}+\tau_2 A (2x^{(j+1)}-x^{(j)})-\tau_2 b\right),\quad z_2^{(j+1)}=\varsigma_{\lambda}\!\left(z_2^{(j)}+\tau_2 B (2x^{(j+1)}-x^{(j)})\right)\!
\end{split}
\end{equation*}
and update the ergodic averages
$$\setlength\abovedisplayskip{5pt}\setlength\belowdisplayskip{5pt}
X_{j+1}=\frac{1}{j+1}\left(jX_{j}+x^{(j+1)}\right), \quad Z_{j+1}=\frac{1}{j+1}\left(jZ_{j}+z^{(j+1)}\right).
$$
\end{algorithmic} \textbf{Output:} $\texttt{InnerItSR}\left(b,x_0,z_0,A,B,k,\tau_1, \tau_2,\lambda,\mathcal{J}\right)=(X_k,Z_k).$
\caption{Inner iterations of primal-dual updates. The $z^{(j)}$ correspond to the complexification of the dual vectors, $\vartheta(y):=\min\{1,{\|y\|_{l^2}}^{-1}\}y$ and $\varsigma_\rho$ is defined in \eqref{simple_function1}.}\label{alg:inner_iterationsB}
\end{algorithm}

Throughout this section, we replace the assumption \cref{assumption} by
\begin{equation}\setlength\abovedisplayskip{5pt}\setlength\belowdisplayskip{5pt}
\label{assumption2}
\|\hat x-x\|_{l^2} \leq  \hat{C}_1 \big[\mathcal{J} (\hat x)+\|B\hat x\|_{l^1}-\mathcal{J} (x)-\|Bx\|_{l^1}+\hat{C}_2\left(\|A\hat x-b\|_{l^2} -\|A x-b\|_{l^2}\right)+\hat{c}(x,b)\big].
\end{equation}
In practice, both assumptions \cref{assumption} and \cref{assumption2} are equivalent up to a change in ${c}(x,b)$ and $\hat{c}(x,b)$. For example, if \cref{assumption} holds, then \cref{assumption2} holds with $C_j=\hat{C}_j$ and $\hat{c}(x,b)=c(x,b)+C_2(\|Ax-b\|_{l^2}-\epsilon)$. The similarity is also reflected in the proofs in the examples we give in later sections, which typically prove \cref{assumption} via \cref{assumption2}. To analyze the problem, we proceed as in \cref{theory1}. The problem \eqref{main_problem2} is equivalent to
$$\setlength\abovedisplayskip{5pt}\setlength\belowdisplayskip{5pt}
\min_{u\in\mathbb{R}^{2N}}\lambda j(u)+\lambda\|K_2u\|_{l^1_{\mathbb{C}}} +\|K_1u-\hat b\|_{l^2},
$$
or, equivalently, the saddle point problem
$$\setlength\abovedisplayskip{5pt}\setlength\belowdisplayskip{5pt}
\min_{u\in\mathbb{R}^{2N}}\max_{y_1\in\mathbb{R}^{2m},y_2\in\mathbb{R}^{2q}}\widehat{\mathcal{L}}(u,y) := \langle K_1u-\hat b,y_1 \rangle+\langle K_2u,y_2 \rangle + \lambda j(u) - \chi_{\mathcal{B}_2}(y_1)-\chi_{\mathcal{B}_\infty}(y_2/\lambda),
$$
where $\mathcal{B}_2$ denotes the closed $l^2$ unit ball. We also remind the reader that $\mathcal{B}_\infty$ denotes the complex closed $l^\infty$ unit ball. These iterations lead to \cref{alg:RR_primal_dualB} with $\lambda=1/\hat{C}_2$. The following theorem describes the convergence. Note that no $\epsilon$ parameter is needed as input.

\begin{algorithm}[t!]
\textbf{Input:} $\hat C_1$ and $\hat C_2$ such that \eqref{assumption2} holds, $L$ (upper bound for $\sqrt{\|A\|^2+\|B\|^2}$), $\tau\in(0,1)$, $\nu\in(0,1)$, $\delta>0$, function handles for $A$, $A^*$, $B$ and $B^*$, seminorm $\mathcal{J}$ and data $b\in\mathbb{C}^m$. \\
\vspace{-4mm}
\begin{algorithmic}[1]
\STATE Set $\epsilon_0=\hat C_2\|b\|_{l^2}$. For $j=1,...,n-1$ compute $\epsilon_j=\upsilon\left(\delta+\epsilon_{j-1}\right)$.
\STATE Set $
k=\left\lceil\frac{4L\hat C_1\sqrt{\hat C_2^2+q}}{\upsilon\tau}\right\rceil,\quad \beta_j=\frac{\hat C_1(\delta+\epsilon_{j-1})}{2\sqrt{1+q\hat{C}_2^{-2}}}\quad \text{for $j=1,...,n$ (recall that $B\in\mathbb{C}^{q\times N}$).}$
\STATE Set $\phi_0(b)=0\in\mathbb{C}^N$, $\hat \phi_0(b)=0\in\mathbb{C}^m$ and for $j=1,...,n$, compute
\begin{equation}\setlength\abovedisplayskip{5pt}\setlength\belowdisplayskip{5pt}
\label{RRPDB}
\left(\!\frac{\phi_j(b)}{\beta_j},\hat \phi_j(b)\!\right)\!\!:=\!\texttt{InnerItSR}\!\left(\!\frac{b}{\beta_{j}},\!\frac{\phi_{j-1}(b)}{\beta_{j}},\hat \phi_{j-1}(b),A,B,k,\tau L^{-1}, \tau L^{-1}\!,{\hat C_2}^{-1}\!,\mathcal{J}\!\!\right)\!\!.\!\!\!\!
\end{equation}
\end{algorithmic} \textbf{Output:} $\phi_n(b)\in\mathbb{C}^N$.
\caption{\textbf{WARPdSR:} Accelerated algorithm for the solution of \eqref{main_problem2} and recovery of desired vector. The updates in \eqref{RRPDB} correspond to restarted and reweighted primal-dual iterations (performed by the routine $\texttt{InnerItSR}$).}\label{alg:RR_primal_dualB}
\end{algorithm}

\begin{theorem}
\label{main_theoremb}
Suppose that \eqref{assumption2} holds. Let $L$ be an upper bound for $\sqrt{\|A\|^2+\|B\|^2}$, $\tau\in(0,1)$, $\nu\in(0,1)$ and $\delta>0$. Then for any $n\in\mathbb{N}$ and any pair $(\varkappa,b)\in\mathbb{C}^{N}\times\mathbb{C}^m$ such that $\hat c(\varkappa,b)\leq\delta$, the following uniform recovery bound holds:
\begin{equation}\setlength\abovedisplayskip{5pt}\setlength\belowdisplayskip{5pt}
\label{key_theorem_boundb}
\|\phi_{n}(b)- \varkappa\|_{l^2}\leq \hat C_1\left(\frac{\delta}{1-\upsilon}+\upsilon^n\hat C_2\|b\|_{l^2}\right),
\end{equation}
where $\phi_n(b)$ denotes the output of WARPdSR in \cref{alg:RR_primal_dualB}.
\end{theorem}

\begin{proof}
See \cref{sec:appendix1}.
\end{proof}

\subsection{Simple extensions}
\label{sec:extensions}
We end this section with some extensions of our setting \eqref{main_problem}.

\vspace{1mm}

\textbf{Optimization over a convex set:} As well as minimization over the whole of $\mathbb{C}^N$ in \eqref{main_problem} (under the constraint $\|Ax-b\|_{l^2}\leq \epsilon$), we can consider minimization over a convex set $S\subset \mathbb{C}^N$. This is useful for extra constraints such as positivity or, if $x$ represents a matrix, Hermitian matrices. The results of this paper carry through and the only change needed in the algorithm is replacing $\mathrm{prox}_{\tau_1 \mathcal{J}}$ by $\mathrm{prox}_{\tau_1 \mathcal{J}+\chi_{1/\beta_jS}}$ in \cref{alg:inner_iterations}. The assumption \eqref{assumption} can then be weakened to only holding for $\hat x,x\in S$. In particular, \cref{approx_PD_alg} holds (under the extra assumption that $x\in S$). The rescaling of $S$ to $\beta_j^{-1}S$ in each call to \texttt{InnerIt} is needed to allow the rescaling of the generalization of \eqref{p3_key_bd}.

\vspace{1mm}

\textbf{Banach spaces and Bregman distances:} Our results can also be extended to
$$\setlength\abovedisplayskip{5pt}\setlength\belowdisplayskip{5pt}
\min_{x\in \mathcal{X}} \mathcal{J}(x)\quad \text{s.t.}\quad\|Ax-b\|\leq \epsilon,
$$
where $\mathcal{X}$ is a reflective Banach space (possibly infinite-dimensional) over the reals and $A:\mathcal{X}\rightarrow \mathcal{Y}^*$ is bounded (where $\mathcal{Y}$ is also a reflective Banach space). Using the usual notation $\langle \cdot , \cdot\rangle$ for the bilinear form on $\mathcal{Y}^*\times \mathcal{Y}$, the saddle point problem then becomes
$$\setlength\abovedisplayskip{5pt}\setlength\belowdisplayskip{5pt}
\min_{x\in\mathcal{X}}\max_{y\in\mathcal{Y}} \langle Ax-b,y \rangle + \mathcal{J}(x) - \epsilon\|y\|.
$$
In particular, this allows norms that are not induced from inner products. The bounds \eqref{assumption} and \eqref{key_theorem_bound} can also be generalized to suitable bounds in terms of Bregman distances. See \cite{chambolle2016ergodic} for primal-dual iterations in terms of Bregman distances and for reflective Banach spaces.

\vspace{1mm}

\textbf{More general norms:} Finally, the duality in argument in the proof of \cref{approx_PD_alg} can be extended to other norms $\|Bx\|$ instead of $\|Bx\|_{l^1}$. Computationally, all this requires is the proximal map of the indicator function of the unit ball of the dual norm of $\|\cdot\|$.

\section{Sparse recovery}
\label{application1}

We consider sparse recovery via the (weighted) $l^1$-norm
\begin{equation}\setlength\abovedisplayskip{4pt}\setlength\belowdisplayskip{4pt}
\label{nbibca}
\mathcal{J}(x)=\|x\|_{l^1_w}:=\sum_{j=1}^N w_j |x_j|, \quad w_j\geq 0, \quad \text{(and take $B=0$)},
\end{equation}
for which \eqref{main_problem} becomes the famous basis pursuit denoising problem. This is a ubiquitous problem in many fields, including machine learning, compressed sensing, and image processing \cite{candes2006stable,candes2006robust,donoho2006compressed,candes2006compressive}. The assumption \eqref{assumption} holds for matrices $A$ that have a (weighted) robust null space property (in levels) defined in \cref{def_rNSPL}, allowing the recovery of vectors $\varkappa$ that are approximately sparse (in levels).\footnote{This is a weaker assumption than the restricted isometry property \cite[Theorem 6.13]{foucart2013invitation}.} Our result is presented explicitly in \cref{main_theorem1}. This setting is very general, for example, encompassing both classical and structured compressed sensing. Examples in imaging for Fourier and Walsh measurements are given in \cref{haar_sparsity}.

\subsection{A general result}
\label{CS_setup}

We consider sparsity in levels \cite{adcock2017breaking}, which has been shown to play a key role in the quality of image recovery in compressed sensing via the so-called ``flip test'' \cite{adcock2017breaking,bastounis2017absence}. For many imaging modalities, sparsity in levels is crucial in demonstrating that sparse regularization is near-optimal for image recovery \cite{adcock2019log,bastounis2017absence, jones2016continuous}. It is needed to account for the good recovery often found in practice for problems such as the Fourier-wavelet problem.\footnote{{The main problem for sparsity in one level in this example is that the Fourier-wavelet matrix is coherent} \cite{adcock2017breaking}.} For example, \cite{lustig2007sparse} observed both poor recovery from uniform random sampling and the improvement offered by variable density sampling for Magnetic Resonance Imaging (MRI). For further works on structured compressed sensing, see \cite{Boyer_2016, BOYER_ACHA_2019,Kutyniok_Lim_2018,Felix_2014,traonmilin2018stable,li2019compressed,eldar2010block}. The following definitions also encompass classical compressed sensing.

\begin{definition}[Sparsity in levels]
\label{levels}
Let $\text{\upshape{\textbf{M}}}=(M_1,...,M_r) \in {\mathbb{N}^r}$, $1\leq M_1<...<M_r=N$, and $\text{\upshape{\textbf{s}}}=(s_1,...,s_r)\in {\mathbb{N}^{r}},$ where $s_k\leq M_k-M_{k-1}$ for $k=1,...,r$ ($M_0=0$). A vector
$x\in\mathbb{C}^N$ is $(\text{\upshape{\textbf{s}}},\text{\upshape{\textbf{M}}})$-sparse in levels if
$$\setlength\abovedisplayskip{5pt}\setlength\belowdisplayskip{5pt}
\left|\mathrm{supp}(x)\cap\{M_{k-1}+1,...,M_k\}\right|\leq s_k,\quad k=1,...,r.
$$
The total sparsity is $s=s_1+...+s_r$. We denote the set of $(\text{\upshape{\textbf{s}}},\text{\upshape{\textbf{M}}})$-sparse vectors by $\Sigma_{\text{\upshape{\textbf{s}}},\text{\upshape{\textbf{M}}}}$. We also define the following measure of distance of a vector $x$ to $\Sigma_{\text{\upshape{\textbf{s}}},\text{\upshape{\textbf{M}}}}$ by
$$\setlength\abovedisplayskip{5pt}\setlength\belowdisplayskip{5pt}
\sigma_{\text{\upshape{\textbf{s}}},\text{\upshape{\textbf{M}}}}(x)_{l^1_w}=\inf\left\{\|x-z\|_{l^1_w}:z\in\Sigma_{\text{\upshape{\textbf{s}}},\text{\upshape{\textbf{M}}}}\right\}.
$$
\end{definition}

Throughout the paper, we drop the $\text{\upshape{\textbf{M}}}$ subscript when considering a single level. For simplicity, we assume that $w_i=w_{(j)}>0$ if $M_{j-1}+1\leq i\leq M_{j}.$ For example, if an image $c$ is compressible in a wavelet basis with coefficients $x$, then $\sigma_{\textbf{s},\textbf{M}}(x)_{l^1_w}$ is expected to be small when the levels correspond to wavelet levels \cite[Ch.\ 9]{Mallat09}. In general, the weights are a prior on the anticipated approximate support of the vector \cite{friedlander2012recovering}. We also define the following quantities:
$$\setlength\abovedisplayskip{5pt}\setlength\belowdisplayskip{5pt}
\xi=\xi(\textbf{s},\textbf{M},w)\coloneqq \sum_{k=1}^rw_{(k)}^2s_k,\quad
\zeta=\zeta(\textbf{s},\textbf{M},w)\coloneqq \min_{k=1,...,r}w_{(k)}^2s_k,\quad
\kappa=\kappa(\textbf{s},\textbf{M},w)\coloneqq {\xi}/{\zeta}.
$$

\begin{definition}[weighted rNSP in levels \cite{bastounis2017absence}]\label{def_rNSPL}
Let $(\textup{\textbf{s}},\textup{\textbf{M}})$ be local sparsities and sparsity levels respectively. For weights $\{w_i\}_{i=1}^N$ $(w_i>0)$, we say that $A\in\mathbb{C}^{m\times N}$ satisfies the weighted robust null space property in levels (weighted rNSPL) of order $(\textup{\textbf{s}},\textup{\textbf{M}})$ with constants $0<\rho<1$ and $\gamma>0$ if for any $(\textup{\textbf{s}},\textup{\textbf{M}})$ support set $\Delta$,
$$\setlength\abovedisplayskip{5pt}\setlength\belowdisplayskip{5pt}
\|x_{\Delta}\|_{l^2}\leq{\rho\|x_{\Delta^c}\|_{l^1_w}}/{\sqrt{\xi}}+\gamma\|Ax\|_{l^2},
\quad\quad \text{for all $x\in\mathbb{C}^N$.}
$$
Here, $x_S$ denotes the vector with $[x_S]_j=x_j$ if $j\in S$ and $[x_S]_j=0$ otherwise.
\end{definition}

With these definitions in hand, the following provides the reconstruction guarantee.

\begin{theorem}
\label{main_theorem1}
Suppose that $A$ has the weighted rNSPL of order $(\textup{\textbf{s}},\textup{\textbf{M}})$ with constants $0<\rho<1$ and $\gamma>0$. Then the assumption \eqref{assumption} holds with
\begin{align*}\setlength\abovedisplayskip{3pt}\setlength\belowdisplayskip{4pt}
&C_1=\left(\rho+\frac{(1+\rho)\kappa^{1/4}}{2}\right)\frac{1+\rho}{\sqrt{\xi}(1-\rho)},\quad C_2=\frac{\gamma}{C_1}\cdot\frac{2+2\rho+(3+\rho)\kappa^{1/4}}{2(1-\rho)},\quad \text{and}\\
&c(x,b)=2\sigma_{\textup{\textbf{s}},\textup{\textbf{M}}}(x)_{l^1_w}+C_2\left(\|Ax-b\|_{l^2}+\epsilon\right).
\end{align*}
Let $\epsilon>0$, $L_A$ be an upper bound for $\|A\|$, $\tau\in(0,1)$, $\delta>0$. Then for any $n\in\mathbb{N}$ and any pair $(\varkappa,b)\in\mathbb{C}^{N}\times\mathbb{C}^m$ such that $\|A\varkappa-b\|\leq \epsilon$ and $c(\varkappa,b)\leq\delta$,
\begin{align*}\setlength\abovedisplayskip{3pt}\setlength\belowdisplayskip{3pt}
\|\phi_{n}(b)- \varkappa\|_{l^2}\!&\leq\! C_1\!\!\left[\!\frac{\delta}{1-\exp(-1)}\!+\!C_2\|b\|_{l^2}\cdot \exp\!\left(\!\!-T(n)\!\!\left\lceil{2eL_A\gamma \frac{2+2\rho+(3+\rho)\kappa^{1/4}}{2(1-\rho)}}\right\rceil^{-1}\!\right)\!\!\right]\!,\\
\|\phi_{n}(b)- \varkappa\|_{l^1_w}\!&\leq\! \frac{1+\rho}{1-\rho}\!\!\left[\!\frac{\delta}{1-\exp(-1)}\!+\!C_2\|b\|_{l^2}\cdot \exp\!\left(\!\!-T(n)\!\!\left\lceil{2eL_A\gamma \frac{2+2\rho+(3+\rho)\kappa^{1/4}}{2(1-\rho)}}\right\rceil^{-1}\!\right)\!\!\right]\!,
\end{align*}
where $\phi_n(b)$ denotes the output of WARPd in \cref{alg:RR_primal_dual} (with optimal choice $\upsilon=\exp(-1)$) and $T(n)=nk$ denotes the total number of inner iterations.
\end{theorem}

\begin{proof}
See \cref{proof_sec_3_SM}.
\end{proof}

In summary, if $A$ satisfies the robust null space property (in levels), then WARPd provides accelerated recovery. The condition $c(\varkappa,b)\leq \delta$ means that both the measurement error $\|A\varkappa-b\|_{l^2}+\epsilon$ and the distance of $\varkappa$ to $\Sigma_{\text{\upshape{\textbf{s}}},\text{\upshape{\textbf{M}}}}$ (measured by $\sigma_{\textup{\textbf{s}},\textup{\textbf{M}}}(\varkappa)_{l^1_w}$) are small. Moreover, the rate of convergence in \eqref{conv_rate99} is directly related to $C_1$ and $C_2$, and hence to $\rho$ and $\gamma$.

\subsection{Example in compressive imaging}
\label{haar_sparsity}


We consider the case that $A$ is a multilevel subsampled unitary matrix \cite{adcock2017breaking} with respect to $U=V\Psi^*$, where $\Psi$ denotes the db2 wavelet transform and $V$ is the discrete Fourier (Fourier sampling) or Walsh--Hadamard transform (binary sampling). $A$ and $A^*$ are implemented rapidly using the fast Fourier transform or fast Walsh--Hadamard transform, and the discrete wavelet transform. Note that $[\mathrm{prox}_{\tau_1\mathcal{J}}(x)]_i = \max\{0,1-{\tau_1w_i}/{|x_i|}\}x_i.$ Hence, the cost per inner iteration is $\mathcal{O}(N\log(N))$. Fourier sampling arises in numerous applications such as MRI, Nuclear Magnetic Resonance and radio interferometry, while binary sampling arises in optical imaging modalities such as lens-less imaging, infrared imaging holography and fluorescence microscopy. Further details on the bases used, sampling structure, and results that $A$ has the weighted rNSPL are given in \cref{proof_sec_3_SM}. \cref{fig:peppers} (left) shows the test image used in this section.

\begin{figure}[!tbp]
  \centering
	\vspace{1mm}
  \begin{minipage}[b]{0.32\textwidth}
    \begin{overpic}[height=4.3cm,trim={20mm 7mm 20mm 7mm},clip]{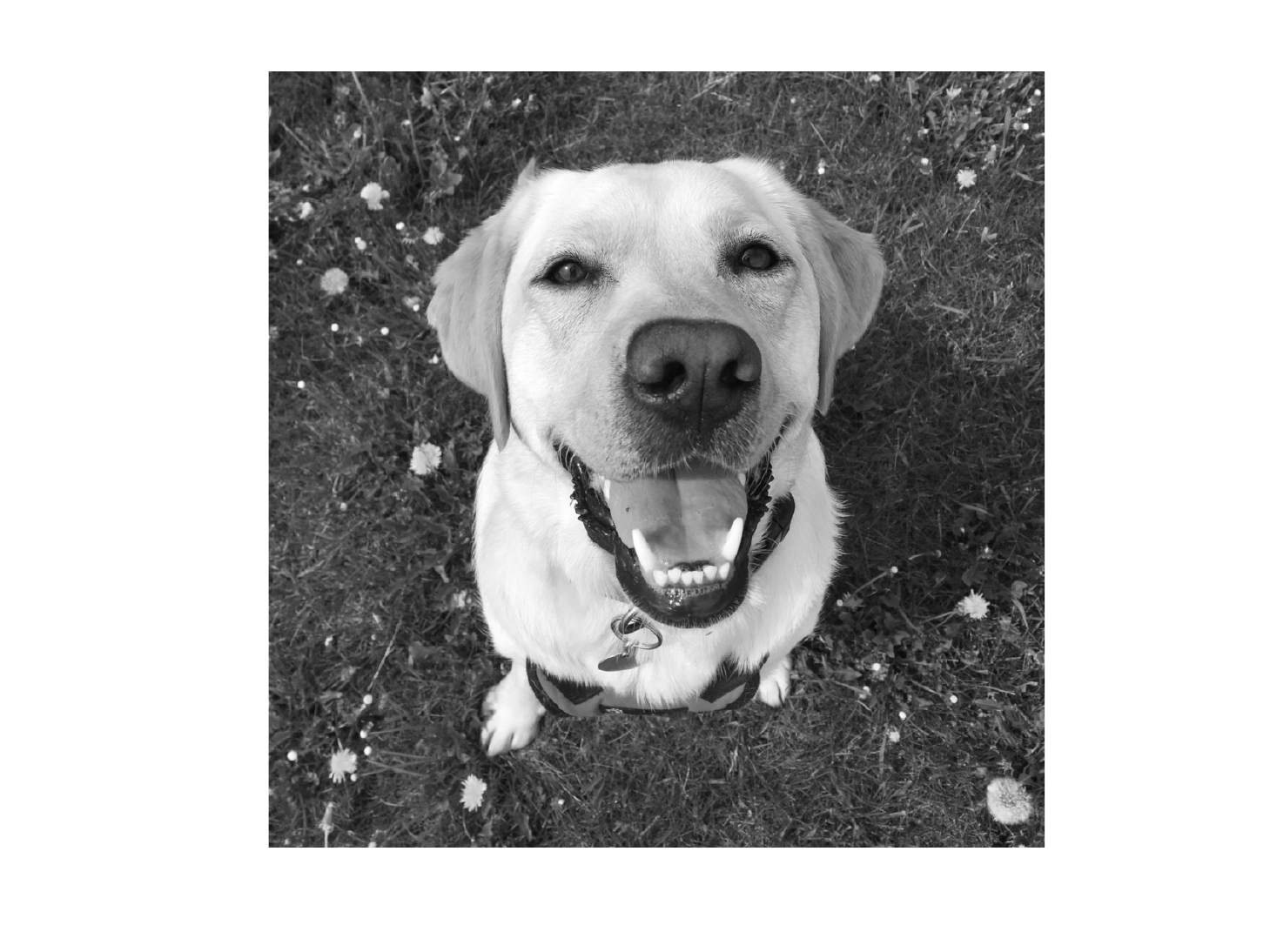}
		\put (31,92) {Test Image}
     \end{overpic}
  \end{minipage}
  \hfill
  \begin{minipage}[b]{0.32\textwidth}
		\begin{overpic}[height=4.3cm,trim={20mm 7mm 20mm 7mm},clip]{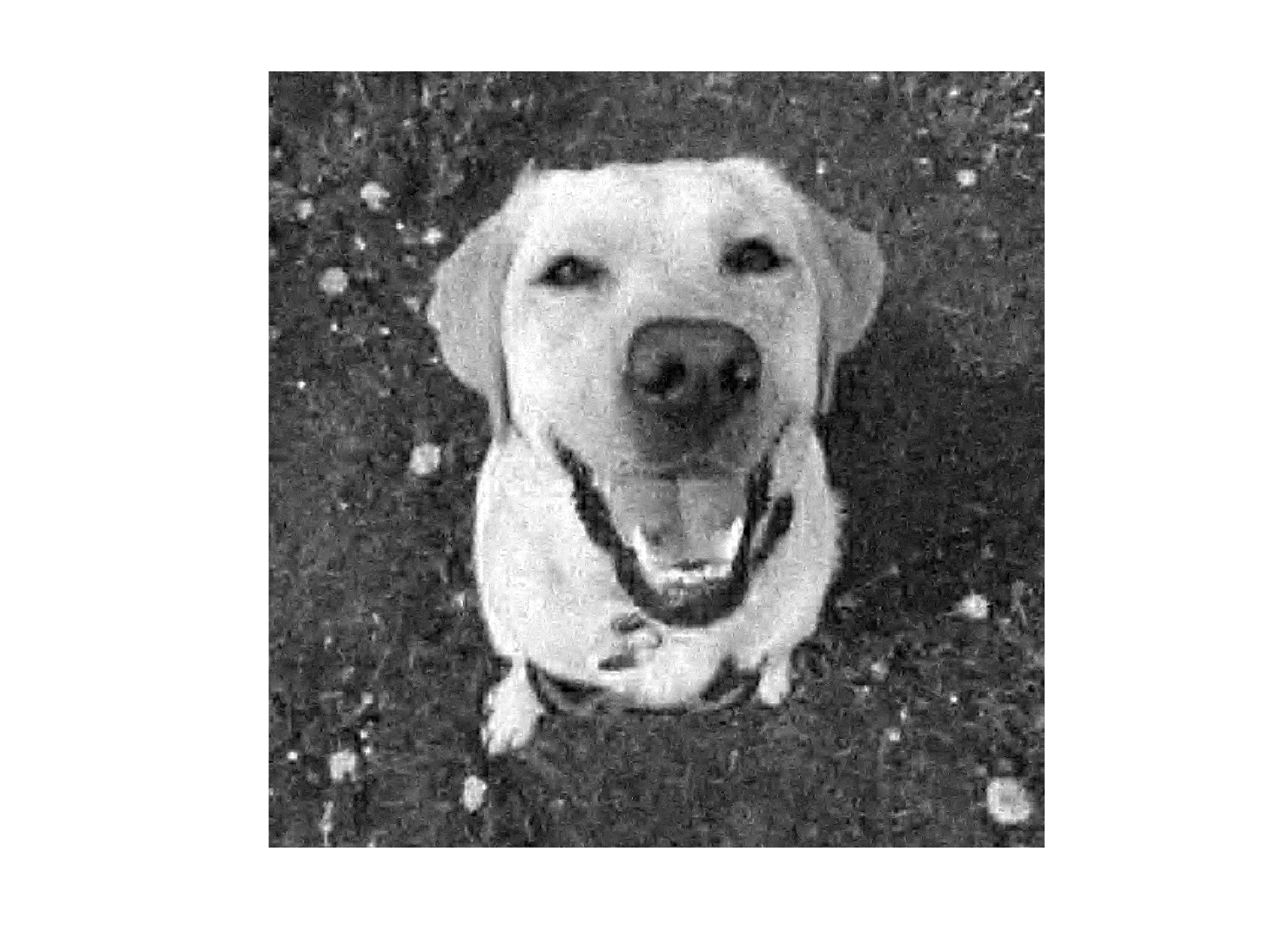}
     \put (36,92) {WARPd}
     \end{overpic}
  \end{minipage}
	\hfill
	\begin{minipage}[b]{0.32\textwidth}
\begin{overpic}[height=4.3cm,trim={20mm 7mm 20mm 7mm},clip]{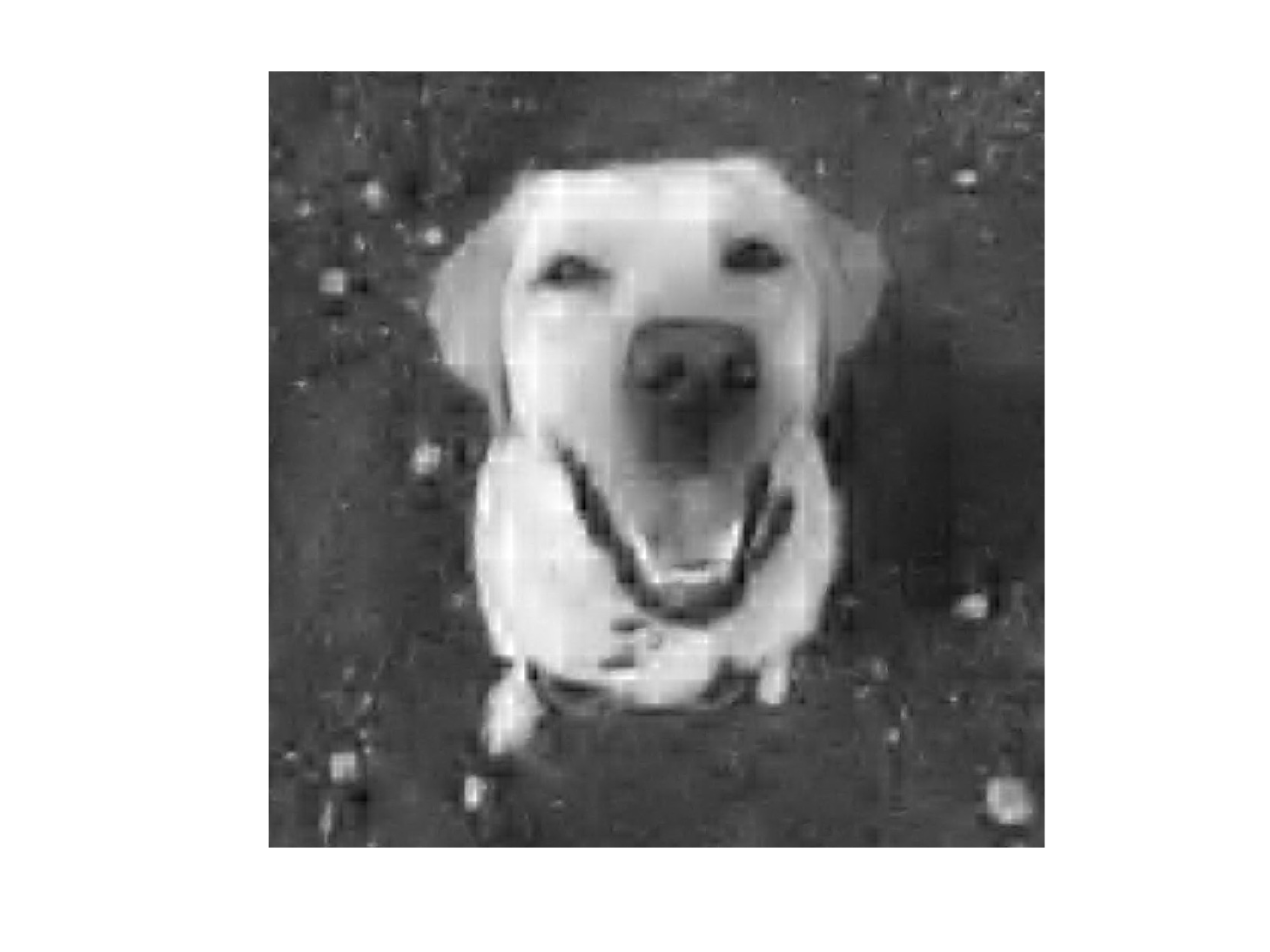}
    \put (37,92) {NESTA}
     \end{overpic}
  \end{minipage}\vspace{-2mm}
  \caption{Left: $1024\times 1024$ test image with pixel values scaled to $[0,1]$. Middle: Recovered image from 5\% binary measurements using WARPd and 30 matrix-vector products. Right: Recovered image from 5\% binary measurements using NESTA and 150 matrix-vector products.}\vspace{-3mm}
\label{fig:peppers}
\end{figure}

As a benchmark, we compare to the algorithm NESTA \cite{becker2011nesta} (available at \url{https://statweb.stanford.edu/~candes/software/nesta/}), which applies a smoothing technique and an accelerated first-order algorithm \cite{nesterov2005smooth}. NESTA is widely regarded as a state-of-the-art method for basis pursuit, is widely used for solving large-scale compressed sensing reconstruction problems, and compares favorably with other state-of-the-art methods (see, for example, the extensive numerical tests in \cite[Section 5]{becker2011nesta}). We run two versions of NESTA to solve \eqref{main_problem}, both with default parameters and acceleration through continuation. For the first version, we take a smoothing parameter $\mu=0.001$. For the second version, we perform a grid-based search for optimal smoothing parameters and for each number of iterations, we report the error for an optimal smoothing parameter. As an error metric for an iterate $x$, we take
\begin{equation}\setlength\abovedisplayskip{5pt}\setlength\belowdisplayskip{5pt}
\label{error_met1}
\mathrm{Error}(x)=\left({\left|\|x\|_{l^1_w}-\|x^*\|_{l^1_w}\right|+C_2\left|\|A x-b\|-\epsilon\right|}\right)/{\|x^*\|_{l^1_w}},
\end{equation}
where $x^*$ is an optimal solution of \eqref{main_problem} computed using several hundred thousand iterations to be sure of convergence. This error directly measures the objective function optimality gap and the feasibility gap (note also that $\mathrm{Error}(x^*)=0$). It also controls the recovery of the sought for image $\varkappa$ (see the proof of \cref{main_theorem1}). In what follows, we present this error metric as a function of the number of matrix-vector products ($A$ or $A^*$) used.

We first consider $15\%$ subsampling and corrupt the measurements with $5\%$ Gaussian noise. The constants $C_1$ and $C_2$ are taken from the discussion in \cref{proof_sec_3_SM}. The sparsities and weights are estimated by thresholding the wavelet coefficients of a Shepp--Logan phantom. In particular, we do not choose or tune any parameters based on the image we use to test the algorithm. We take $\epsilon=0.06\|b\|_{l^2}$, $\delta=C_2\epsilon$, $\nu=e^{-1}$ and $\tau=1$. \cref{fig:fourier_recovery} (left, middle) shows the convergence for our algorithm using ergodic iterates and non-ergodic iterates in both the inner iterations and restarts. We have also shown results for non-restarted primal-dual iterations (labeled PD). The benefit of acceleration is clear and our algorithm converges at a much faster rate than NESTA. The non-ergodic version of our algorithm performs better than \cref{alg:RR_primal_dual}. We do not have a theoretical explanation for this, but this kind of behavior (and its reverse, i.e., ergodic iterates performing better) has been observed for non-restarted primal-dual iterations \cite{chambolle2016ergodic}. The case of binary sampling also converges slightly faster (this is to be expected from the sampling bounds mentioned in \cref{proof_sec_3_SM}). We found similar behavior for a range of different images, subsampling rates, higher order wavelets etc.

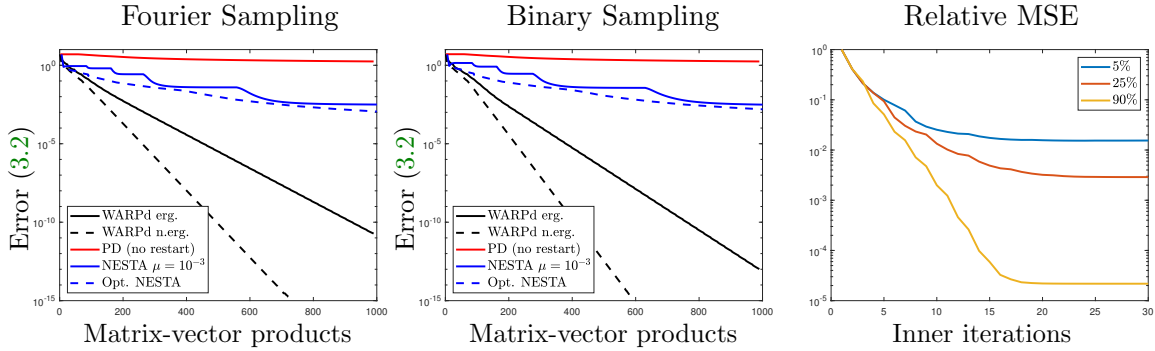
\begin{figure}[!tbp]
  \centering
	\vspace{1mm}
  \begin{minipage}[b]{0.32\textwidth}
	\centering
    \begin{overpic}[width=\textwidth,trim={2mm 0mm 10mm 0mm},clip]{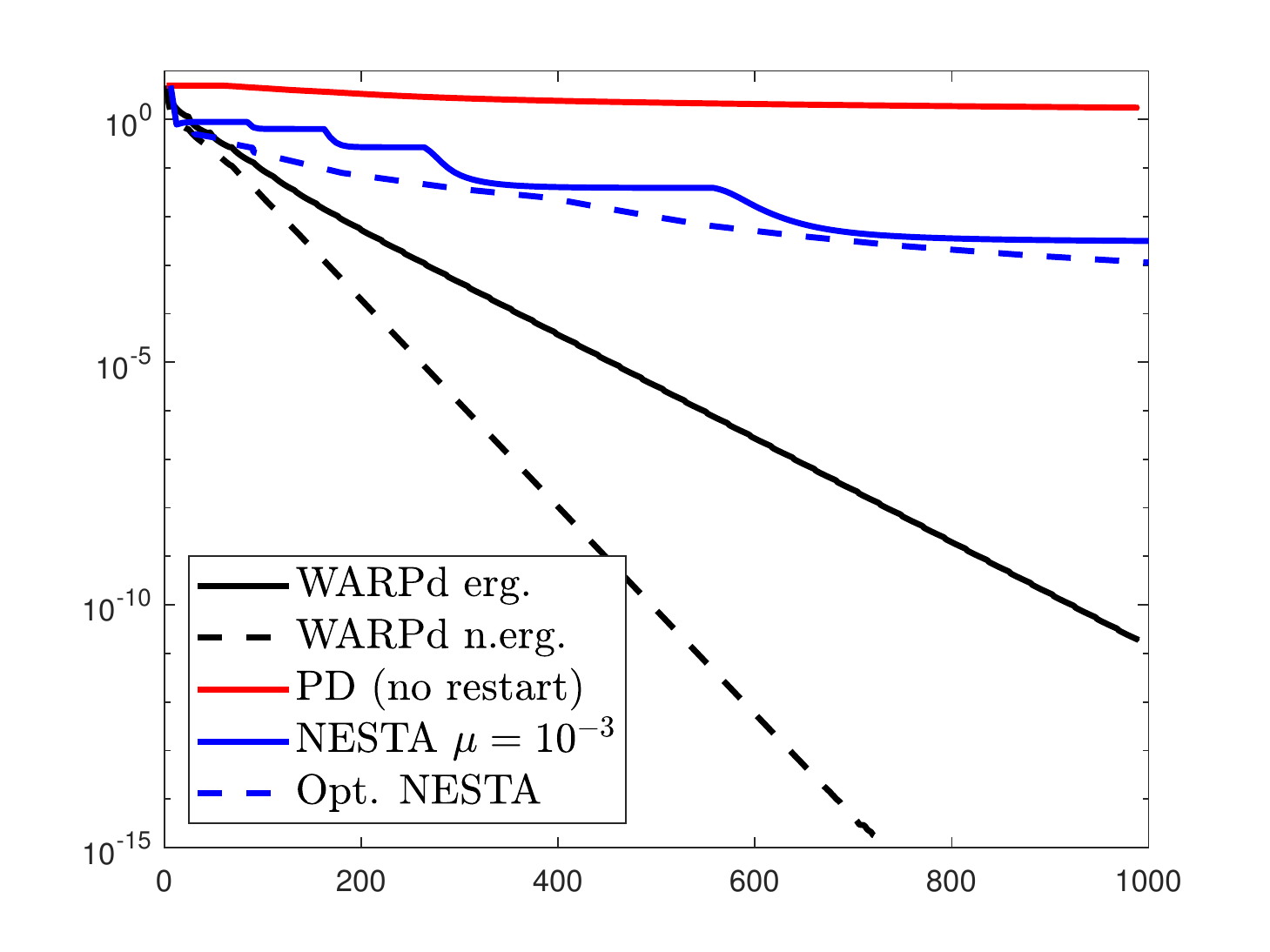}
		\put (-1,22) {\rotatebox{90}{\small{Error \eqref{error_met1}}}}
    \put (29,82) {Fourier Sampling}
     \put (19,-2) {\small{Matrix-vector products}}
     \end{overpic}
  \end{minipage}
	\begin{minipage}[b]{0.32\textwidth}
	\centering
    \begin{overpic}[width=\textwidth,trim={2mm 0mm 10mm 0mm},clip]{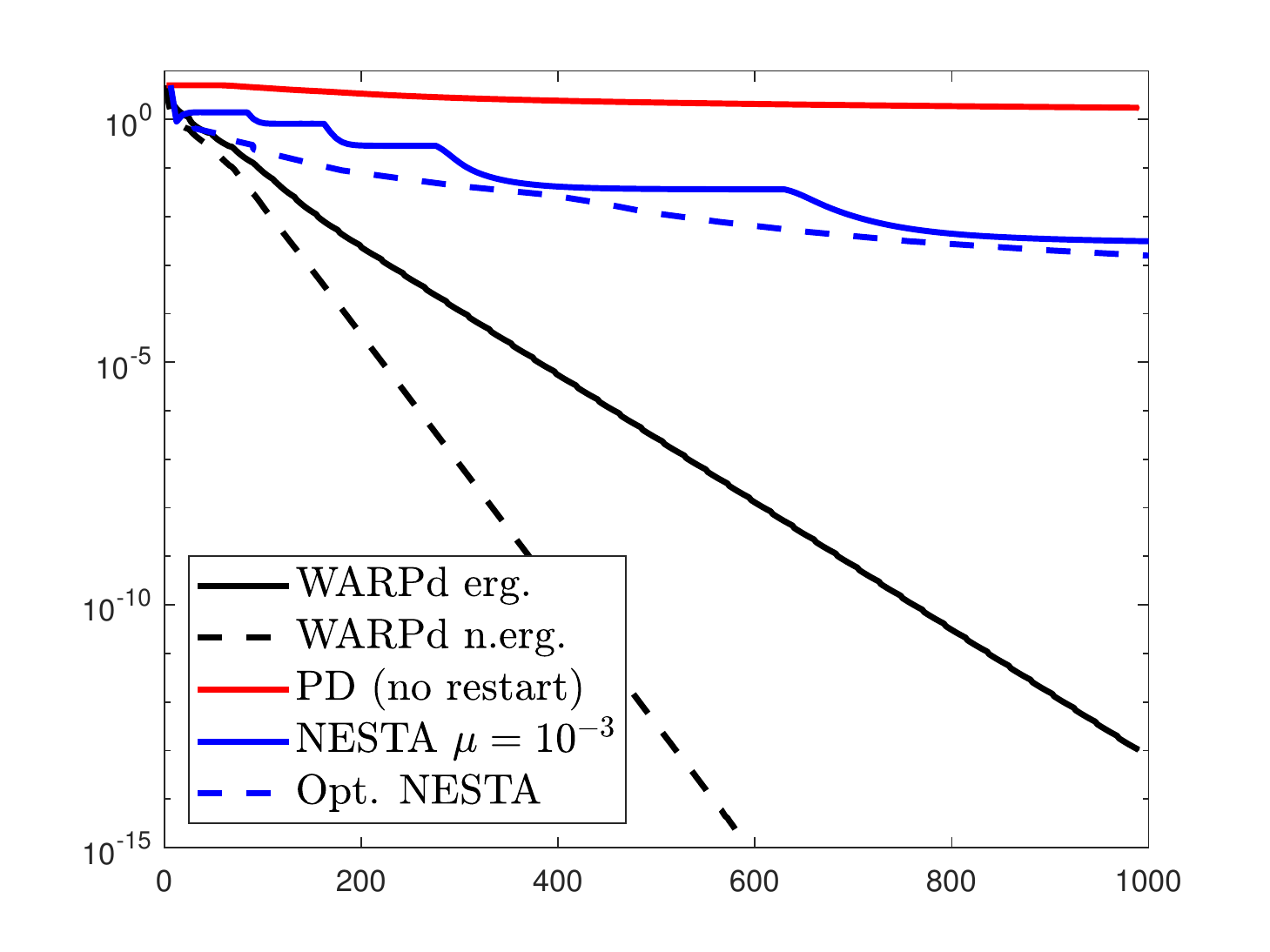}
    \put (-1,22) {\rotatebox{90}{\small{Error \eqref{error_met1}}}}
    \put (29,82) {Binary Sampling}
     \put (19,-2) {\small{Matrix-vector products}}
     \end{overpic}
  \end{minipage}
  \begin{minipage}[b]{0.32\textwidth}
    \begin{overpic}[width=\textwidth,trim={2mm 0mm 10mm 0mm},clip]{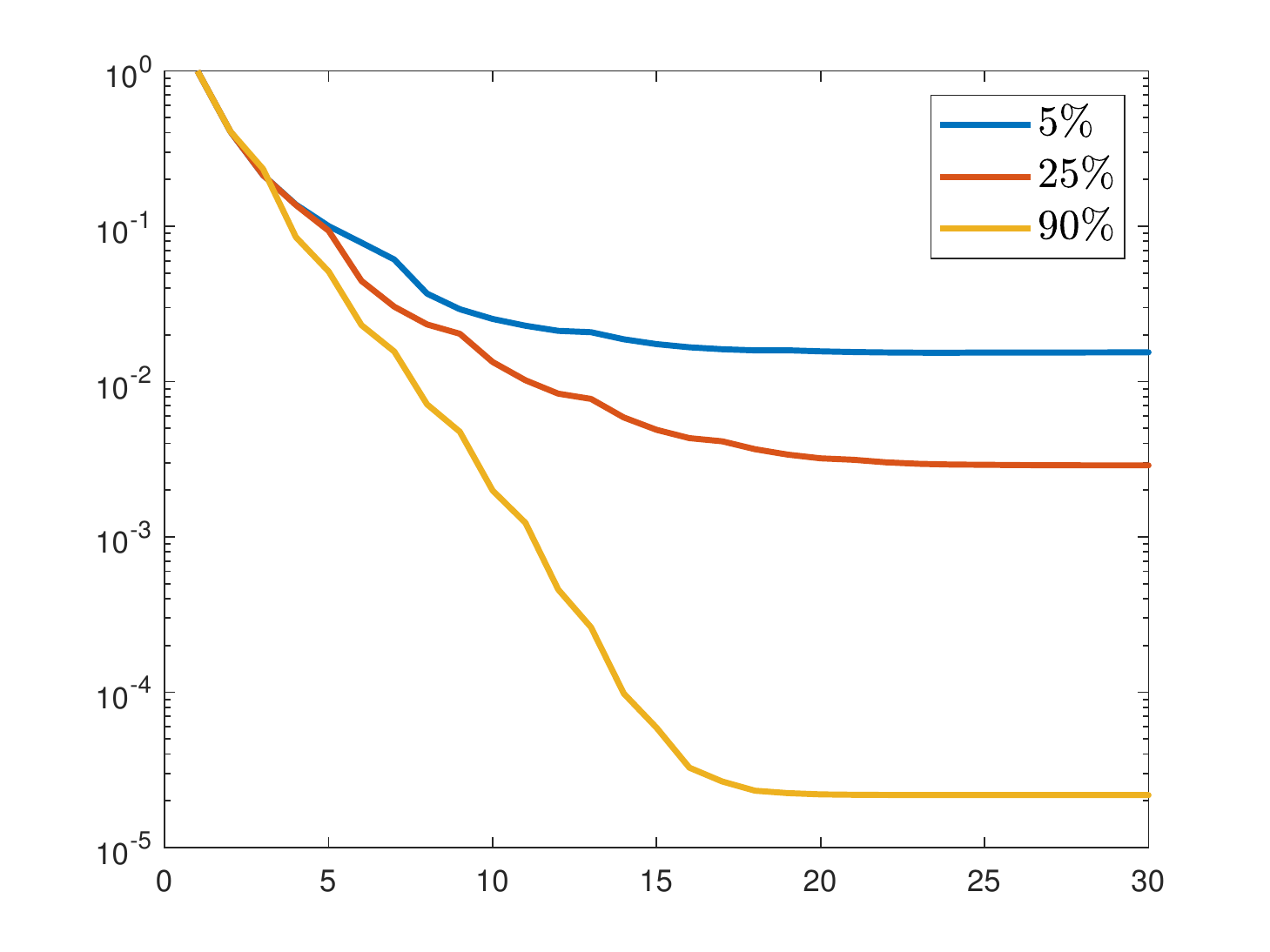}
		\put (32,82) {Relative MSE}
		\put (29,-2) {\small{Inner iterations}}
     \end{overpic}
  \end{minipage}
	  \caption{Left: Convergence for Fourier 15\% sampling. Middle: Convergence for binary 15\% sampling. Right: Convergence for Fourier sampling and different sampling rates (relative MSE).}\vspace{-3mm}
\label{fig:fourier_recovery}
\end{figure}

We now consider the difference between the reconstruction and the image itself. From \cref{main_theorem1}, we expect that this error will decrease linearly down to the intrinsic bound $\sim C_1\delta$, which corresponds to the distance from the image to the set of solutions of \eqref{main_problem}. \cref{fig:fourier_recovery} (right) shows the relative MSE between the reconstruction and the image for Fourier sampling at different sampling rates. In all cases, the level of noise was chosen so that it contributes an error comparable to solving \eqref{main_problem}. We see the expected behavior, where the final error is due to the fact that the image's wavelet coefficients are only approximately sparse (for example, the error for a standard phantom image was of the order $10^{-12}$), and, as expected, is smaller for the larger sampling rate, with a faster rate of convergence. Similar behavior occurs for binary sampling. For example, \cref{fig:fourier_recovery} (middle, right) shows the reconstruction using $5\%$ sampling and 30 matrix-vector products for WARPd, as well as 150 matrix-vector products for NESTA. Again, this demonstrates the faster convergence.

\section{Low-rank matrix recovery}
\label{application2}

In this section, we consider the problem of recovering an approximately low-rank matrix $M\in\mathbb{C}^{n_1\times n_2}$ via \eqref{main_problem} with the nuclear norm regularizer
\begin{equation}\setlength\abovedisplayskip{5pt}\setlength\belowdisplayskip{5pt}
\label{eq:nuc_norm}
\mathcal{J}(M)=\|M\|_1:=\sum_{j=1}^{\min\{n_1,n_2\}}\sigma_j(M),\quad \text{(and take $B=0$)},
\end{equation}
where $\sigma_j(M)$ denotes the singular values of $M$. Low-rank matrix is a non-commutative version of recovery of (approximately) sparse vectors. The assumption of low-rank assures that the matrix $M^*M$ is sparse in its eigenbasis. There are numerous instances where nuclear norm minimization \eqref{main_problem} with \eqref{eq:nuc_norm} (and related problems) provably recovers the desired low-rank matrix from considerably fewer than $n_1n_2$ measurements \cite{candes2009exact,gross2011recovering,liu2011universal,recht2010guaranteed,kueng2017low,krahmer2021convex}.\footnote{Similar to the relationship between $l^1$ and $l^0$ minimization, the nuclear norm is a convex relaxation of the rank operator and the rank minimization problem is NP-hard in general.}

We consider measurement maps of the form ($\mathrm{tr}$ denotes trace)
\begin{equation}\setlength\abovedisplayskip{5pt}\setlength\belowdisplayskip{5pt}
\label{mat_recov_measjkjkjk}
A(M)=\sum_{j=1}^m \mathrm{tr}(MA_j^*)e_j\in\mathbb{C}^m, \quad A^*(y)=\sum_{j=1}^m y_j A_j\in\mathbb{C}^{n_1\times n_2},
\end{equation}
where $A_j\in\mathbb{C}^{n_1\times n_2}$ are measurement matrices and the $\{e_j\}_{j=1}^m$ are the canonical basis vectors of $\mathbb{C}^m$. We apply the method of \cref{THEALG} by taking the vector $x$ to be the vectorized matrix $M$. The assumption \eqref{assumption} holds for measurement maps $A$ that satisfy the Frobenius-robust rank null space property in \cref{Frobrnsp}, which is analogous to \cref{def_rNSPL}. This is a weaker assumption than the rank restricted isometry property \cite[Theorem 6.13]{foucart2013invitation}\footnote{The cited theorem is for the analogous properties of sparse recovery of vectors. The adaptation of the proof for the case of recovery of low-rank matrices is straight-forward using the relevant Schatton $p$-norms.} (the rank restricted isometry property is another common property used to prove recovery results \cite{candes2011tight,recht2010guaranteed}), and allows the recovery of matrices $M$ that are approximately low-rank. \cref{main_theorem2} gives our result and, as an example, we consider Pauli measurements in quantum state tomography.

\subsection{A general result}
\label{mat_setup}

The following definition is analogous to \cref{def_rNSPL} for a single level and unweighted $l^1$-norm\footnote{It is possible to consider a weighted version of the nuclear norm. However, the associated optimization problem is very difficult and non-convex \cite{gu2014weighted}.}, but now the relevant norms are replaced by their Schatten $p$-norm counterparts. We use $\|M_c\|_{1}$ to denote $\sum_{j>r}\sigma_j(M)$ for a given $r$.

\begin{definition}[Frobenius-robust rank null space property \cite{kabanava2016stable}]
\label{Frobrnsp}
We say that $A:\mathbb{C}^{n_1\times n_2}\rightarrow\mathbb{C}^m$ satisfies the Frobenius-robust rank null space
property of order $r$ with constants $\rho\in(0,1)$ and $\gamma>0$ if for all $M\in\mathbb{C}^{n_1\times n_2}$, the singular values of $M$ satisfy
$$\setlength\abovedisplayskip{5pt}\setlength\belowdisplayskip{5pt}
\|M_r\|_{2}:=\sqrt{\sigma_1(M)^2+\ldots+\sigma_r(M)^2}\leq {\rho}\|M_c\|_{1}/{\sqrt{r}}+\gamma\|A(M)\|_{l^2}.
$$
\end{definition}

The following provides the reconstruction guarantee.

\begin{theorem}
\label{main_theorem2}
Suppose that $A:\mathbb{C}^{n_1\times n_2}\rightarrow\mathbb{C}^m$ satisfies the Frobenius-robust rank null space property of order $r$ with constants $\rho\in(0,1)$ and $\gamma>0$. Then \eqref{assumption} holds with
\begin{align}\setlength\abovedisplayskip{5pt}\setlength\belowdisplayskip{5pt}
\label{consts_mat_com}
C_1\!=\!\frac{(1+\rho)^2}{(1-\rho)r^{\frac{1}{2}}}, \quad\!\!\! C_2\!=\!\frac{\gamma(3+\rho)r^{\frac{1}{2}}}{(1+\rho)^2}, \quad\!\!\! c(M,b)\!=\!2\|M_c\|_1\!+\!\frac{\gamma(3+\rho)r^{\frac{1}{2}}}{(1+\rho)^2}\left(\|A(M)-b\|_{l^2}\!+\!\epsilon\right).
\end{align}
Let $\epsilon>0$, $L_A$ be an upper bound for $\|A\|$, $\tau\in(0,1)$, $\delta>0$, and $C_1,C_2$ and $c(\cdot,\cdot)$ be given by \eqref{consts_mat_com}. Then for any $n\in\mathbb{N}$, and $p\in[1,2]$, and any pair $(M,b)\in\mathbb{C}^{n_1\times n_2}\times\mathbb{C}^m$ such that $\|A(M)-b\|\leq \epsilon$ and $c(M,b)\leq\delta$, the following uniform recovery bounds hold:
$$\setlength\abovedisplayskip{5pt}\setlength\belowdisplayskip{5pt}
\|\phi_{n}(b)- M\|_{p}\!\leq \!\frac{(1+\rho)^2}{(1-\rho)}\!\!\left[\!\frac{\delta r^{\frac{1-p}{p}}}{1-\exp(-1)}\!+\!\frac{\gamma(3+\rho)r^{\frac{1}{p}-\frac{1}{2}}}{(1+\rho)^2}\|b\|_{l^2}\cdot \exp\!\left(\!-T(n)\!\left\lceil{2eL_A\gamma \frac{(3+\rho)}{(1-\rho)}}\right\rceil^{-1}\!\right)\right]\!\!,
$$
where $\phi_n(b)$ denotes the output of WARPd in \cref{alg:RR_primal_dual} (with optimal choice $\upsilon=\exp(-1)$) and $T(n)=nk$ denotes the total number of inner iterations.
\end{theorem}

\begin{proof}
See \cref{lowrkmatrecthm}.
\end{proof}

In summary, if $A$ satisfies the Frobenius-robust rank null space property, then WARPd provides accelerated recovery. The condition $c(M,b)\leq \delta$ means that both the measurement error $\|AM-b\|_{l^2}+\epsilon$ and the distance of $M$ to low-rank matrices ($\|M_c\|_1$) are small. Moreover, the convergence rate in \eqref{conv_rate99} is directly related to $C_1$ and $C_2$, and hence to $\rho$ and $\gamma$.

\subsection{Example: Pauli measurements and quantum state tomography}
\label{sec_low_rk_num}

An important application of matrix recovery in physics, known as quantum state tomography (QST), is reconstructing a finite $n$-dimensional quantum mechanical system. Such a system is fully characterized by its density operator $\rho$, an $n\times n$ positive-semidefinite matrix with trace 1. For example, QST is now a routine task for designing, testing and tuning qubits in the quest of building quantum information processing devices \cite{lvovsky2009continuous}. A key structural property, for which the quantum system is called ``almost pure'', is that $\rho$ be well-approximated by a low-rank matrix. Under this assumption, QST becomes a low-rank matrix recovery problem \cite{gross2011recovering,gross2010quantum,liu2011universal}. QST requires a measurement process that is experimentally realizable and efficient.

In this example, we consider Pauli measurements, where the $A_j$ are constructed from randomly sampling tensor products of the usual Pauli matrices. Pauli measurements lead to efficient recovery of low-rank density operators \cite{gross2011recovering,gross2010quantum} and are especially easy to carry out experimentally \cite{schwemmer2014experimental,riofrio2017experimental}. It was shown in \cite{liu2011universal} that sets of $\mathcal{O}(rn\cdot \mathrm{poly}(\log(n)))$ Pauli measurements satisfy the rank restricted isometry property, and hence satisfy the Frobenius-robust rank null space property in \cref{Frobrnsp}. We can thus apply \cref{main_theorem2}.

In the general case of \eqref{eq:nuc_norm}, the proximal map of $\mathcal{J}$ is computed using the singular value decomposition (SVD). Namely, if $M=U\mathrm{diag}(\sigma(M))V^*\in\mathbb{C}^{n_1\times n_2}$, then \cite[Theorem 2.1]{cai2010singular}
\begin{equation}\setlength\abovedisplayskip{5pt}\setlength\belowdisplayskip{5pt}
\label{mat_prox_form}
\mathrm{prox}_{\tau_1\mathcal{J}}(M) =U\phi_{\tau_1}(\mathrm{diag}(\sigma(M)))V^*,\quad \phi_{\alpha}(z)=\max\left\{0,1-{\alpha}/{|z|}\right\}z,
\end{equation}
where $\phi_{\tau_1}$ is applied element-wise to the diagonal matrix $\mathrm{diag}(\sigma(M))$. Naively, the cost of applying $\mathrm{prox}_{\tau_1\mathcal{J}}$ is dominated by the $\mathcal{O}(n_1n_2\min\{n_1,n_2\})$ cost of computing the SVD \cite[Chapter 31]{trefethen1997numerical}. In this example, since the measurement matrix is sparse and, due to the thresholding, we only need the dominant eigenvalues (the matrices are Hermitian so the SVD reduces to an eigenvalue decomposition), we found it beneficial to use methods for computing eigenvalue decompositions based on matrix-vector products (see \cref{sec:mat_alg_comp}). In general, reducing the number of iterations through accelerated methods such as WARPd is particularly important in low-rank matrix recovery since the cost of applying $A$ may be large for large $n_1$ and $n_2$ (e.g., for Gaussian measurements used in phase retrieval \cite{shechtman2015phase}).\footnote{For Gaussian measurements and general measurement matrices $A_j$ in \eqref{mat_recov_measjkjkjk}, $C_A=\mathcal{O}(n_1n_2m)$ with $m \gtrsim n_1+n_2$ so there is little benefit gained by using an approximate SVD.}



As a benchmark, we compare to TFOCS \cite{becker2011templates} (available at \url{http://cvxr.com/tfocs/}), which has become a defacto method for matrix retrieval problems such as PhaseLift \cite{candes2013phaselift,fannjiang2020numerics,candes2015phase} and other related techniques. TFOCS applies an optimal first-order method \cite{auslender2006interior} to a smoothed version of the dual problem. We use the default parameters (apart from the tolerance, which we decrease to achieve higher accuracy), accelerated continuation, and a smoothing parameter $\mu=1$ (relative to $\|A\|$). In this case, the smoothing term is $\frac{\mu}{2}\|\cdot-M_0\|_2^2$, with $M_0$ updated at each restart. Hence, although we found $\mu=1$ close to optimal, tuning the value of $\mu$ is of little practical significance. As an error metric for an iterate $\tilde M$, we take the relative error
\begin{equation}\setlength\abovedisplayskip{5pt}\setlength\belowdisplayskip{5pt}
\label{error_met2}
\mathrm{Error}(\tilde M)=\big({\big|\|\tilde M\|_1-\|M^*\|_1\big|+C_2\big|\|A(\tilde M)-b\|-\epsilon\big|}\big)/{\|M^*\|_1},
\end{equation}
where $M^*$ is an optimal solution of \eqref{main_problem}, computed using a much larger number of iterations. 

For our example, we set $r=10$ and $n=2^{10}$ (corresponding to 10 qubits). We generate two independent complex standard Gaussian matrices $M_L,M_R\in\mathbb{C}^{n\times r}$ and set $\tilde M = M_LM_R^*M_RM_L^*$, $M={\tilde M}/{\mathrm{tr}(\tilde M)}.$ We then use 10\% subsampling and corrupt the measurements with 2\% Gaussian noise. We take $\epsilon=0.03\|b\|_{l^2}$, $\delta=C_2\epsilon$ ($C_1$ and $C_2$ are selected based on the theorem in \cite{liu2011universal}), $\nu=e^{-1}$ and $\tau=1$. \cref{fig:matrix_recovery} shows the results. We see the clear benefit of acceleration and that WARPd converges at a much faster rate than TFOCS.


\begin{figure}[!tbp]
  \centering
  \begin{minipage}[b]{0.48\textwidth}
    \begin{overpic}[width=\textwidth,trim={0mm 0mm 0mm 0mm},clip]{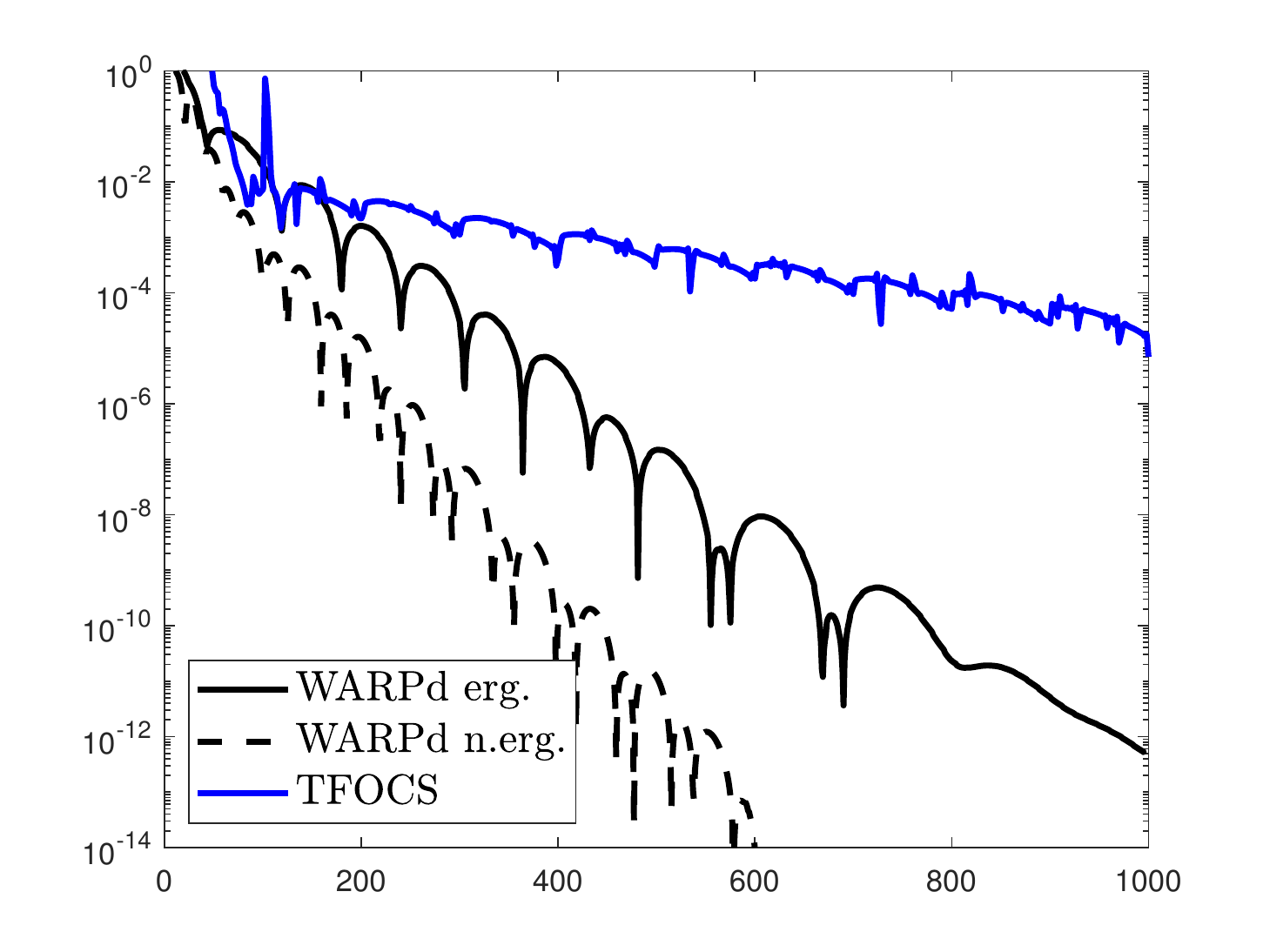}
		\put (-1,23) {\rotatebox{90}{{Error \eqref{error_met2}}}}
		\put (37,-1) {Calls to $\mathrm{prox}_{\mathcal{J}}$}
     \end{overpic}
  \end{minipage}
		  \caption{Errors for Pauli measurements example.}\vspace{-3mm}
\label{fig:matrix_recovery}
\end{figure}

\section{Matrix completion and non-uniform recovery guarantees}
\label{sec:mat_com_exm}

In this section, we consider the problem of matrix completion. Given an approximately low-rank matrix $M\in\mathbb{C}^{n_1\times n_2}$ and an index set $\Omega\subset\{1,...,n_1\}\times\{1,...,n_2\}$, we wish to recover $M$ from measurements $b$ with
$$\setlength\abovedisplayskip{5pt}\setlength\belowdisplayskip{5pt}
\left[P_{\Omega}(M)\right]_{i,j}=\begin{dcases}
M_{i,j},\quad&\text{if }(i,j)\in\Omega,\\
0,\quad &\text{otherwise},
\end{dcases}\quad\text{and} \quad b-e= A(M)=\mathrm{vect}(P_{\Omega}(M))\in\mathbb{C}^{|\Omega|}.
$$
This can be viewed as a special case of the problem considered in \cref{application2} and we consider
\begin{equation}\setlength\abovedisplayskip{5pt}\setlength\belowdisplayskip{5pt}
\label{mat_comp_problems}
\min_{\tilde M\in \mathbb{C}^{n_1\times n_2}} \|\tilde M\|_1\quad \text{s.t.}\quad\|P_{\Omega}(\tilde M)-b\|_{l^2}\leq \epsilon.
\end{equation}
However, we treat the problem in a separate section for at least three reasons. First, there are obvious rank-one matrices in the kernel of the measurement operator, and hence the Frobenius-robust rank null space property we made use of in \cref{mat_setup} cannot hold. The lack of such a global property renders matrix completion a more challenging problem. However, if certain conditions on the left and right singular vectors of the underlying low-rank matrix are imposed, essentially requiring that such vectors are uncorrelated with the canonical basis, then the matrix can be recovered with sufficiently many measurements \cite{candes2009exact,candes2010power,recht2011simpler,gross2011recovering}. Such conditions lead to \textit{non-uniform recovery guarantees}. We show how such results fall within our framework via a local version of \eqref{assumption}. Similar arguments also hold for non-uniform recovery of sparse vectors. Second, this problem has distinct algorithmic challenges when dealing with large-scale problems, discussed in \cref{sec:mat_alg_comp}. Third, (approximately) low-rank matrices pervade data science \cite{udell2019big} and matrix completion has received much attention with applications ranging from recommender systems \cite{rennie2005fast,koren2009matrix}, inferring camera motion \cite{chen2004recovering,tomasi1992shape}, multiclass learning \cite{amit2007uncovering,evgeniou2007multi} and many more in statistics, machine learning, and signal processing. 

\subsection{A general result}
\label{mat_comp_theory}

We show how a \textit{local} version of \eqref{assumption} holds for the problem of matrix completion under the existence of an approximate dual certificate. By local we mean that $\hat x =\widehat M$ is allowed to vary but $x=M$ is kept fixed. This locality does not alter \cref{main_theorem} or its proof. The existence of an approximate dual certificate is a predominant method of proving that solutions of optimization problems such as \eqref{mat_comp_problems} approximate $M$. Let $M\in\mathbb{C}^{n_1\times n_2}$ and let $M=U\Sigma V^*$ denote its singular value decomposition. Here $\Sigma\in\mathbb{R}^{r\times r}$ is diagonal (with $r\leq \min\{n_1,n_2\}$), and $U\in\mathbb{C}^{n_1\times r}, V\in\mathbb{C}^{n_2\times r}$ are partial isometries (so that $U^*U=V^*V=I_{r}$).

To state the conditions for accelerated recovery, we first introduce a few objects associated with $M$. The tangent space of the variety of rank $r$ matrices at the point $M$ is given by
$$\setlength\abovedisplayskip{5pt}\setlength\belowdisplayskip{5pt}
T_M=\{UB_1^*+B_2V^*:B_1\in\mathbb{C}^{n_2\times r},B_2\in\mathbb{C}^{n_1\times r}\}.
$$
We denote by $P_{T_M}$, the (Hilbert--Schmidt) orthogonal projection onto the tangent space and set $P_{T^\perp_{M}}=I-P_{T_M}$. Let $P=UU^*$ and $Q=VV^*$. One can easily check that
$$\setlength\abovedisplayskip{5pt}\setlength\belowdisplayskip{5pt}
P_{T^\perp_{M}}:\widetilde M\rightarrow P^{\perp}\widetilde M Q^{\perp},\quad  P_{T_M}:\widetilde M\rightarrow P\widetilde M +\widetilde M Q -P\widetilde M Q.
$$
With these in hand, we provide the following two definitions. These are slightly weaker than those usually used in the literature (for example, it is common to assume a restricted isometry property instead of \cref{RIP_mat_comp2}), but suffice to prove \cref{mat_comp_key_bound100}.

\begin{definition}
Given a measurement operator $A:\mathbb{C}^{n_1\times n_2}\rightarrow \mathbb{C}^m$, a vector $z\in\mathbb{C}^m$, with matrix $Y=A^*z\in\mathbb{C}^{n_1\times n_2}$, is an approximate dual certificate at $M$ if upon defining
\begin{equation}\setlength\abovedisplayskip{5pt}\setlength\belowdisplayskip{5pt}
\label{dual_certificate_approx}
\alpha_1=\|UV^*-P_{T_M} Y\|_2\text{ and } \alpha_2 =\|P_{T_M^\perp}Y\|,\quad \text{it holds that}\quad\alpha_2<1.
\end{equation}
\end{definition}

\begin{definition}
\label{RIP_mat_comp2}
We say that $A$ is bounded below on $T_M$ with constant $\gamma>0$ if
\begin{equation}\setlength\abovedisplayskip{5pt}\setlength\belowdisplayskip{5pt}
\label{RIP_mat_comp3}
\gamma\|Z\|_2\leq \|A(Z)\|_{l^2},\quad \forall Z\in T_M.
\end{equation}
\end{definition}

\begin{theorem}
\label{mat_comp_key_bound100}
Let $M\in\mathbb{C}^{n_1\times n_2}$  and suppose that $A$ is bounded below on $T_M$ with constant $\gamma>0$ and $z\in\mathbb{C}^m$ is an approximate dual certificate at $M$ (so that \eqref{dual_certificate_approx} holds). If $\alpha_1\|A\|<(1-\alpha_2)\gamma$, then for any $\widehat{M}\in\mathbb{C}^{n_1\times n_2}$, 
\begin{equation}\setlength\abovedisplayskip{5pt}\setlength\belowdisplayskip{5pt}
\label{inequal_mat_comp_golf}
\|\widehat{M}-M\|_2\leq\frac{\gamma+\|A\|}{(1-\alpha_2)\gamma-\alpha_1\|A\|}\left[\|\widehat{M}\|_1\!-\!\|M\|_1\!+\!\left(\!\frac{\alpha_1+1-\alpha_2}{\gamma+\|A\|}+\|z\|_{l^2}\!\!\right)\!\left\|A\left(\widehat{M}-M\right)\right\|_{l^2}\right].
\end{equation}
It follows that \eqref{assumption} is locally satisfied for the problem \eqref{eq:nuc_norm} with
$$\setlength\abovedisplayskip{5pt}\setlength\belowdisplayskip{5pt}
C_1=\frac{\gamma+\|A\|}{(1-\alpha_2)\gamma-\alpha_1\|A\|}, C_2=\left(\!\frac{\alpha_1+1-\alpha_2}{\gamma+\|A\|}+\|z\|_{l^2}\!\!\right),c(M,b)=C_2(\epsilon+\|A(M)-b\|_{l^2}),
$$
where locally refers to allowing $\hat x =\widehat M$ to vary but keeping $x=M$ fixed.
\end{theorem}

\begin{proof}
See \cref{lowrkmatrecthm}.
\end{proof}

The existence of (approximate) dual certificates for matrix completion has been studied extensively \cite{candes2009exact,candes2010power,gross2011recovering,recht2011simpler}. We follow \cite{ding2020leave}, which gives the current state-of-the-art sample complexity. The observation indices $\Omega$ are chosen randomly such that $\mathbb{P}((i,j)\in\Omega)=p\in[0,1)$ for all $(i,j)$ independently. Using the standard basis $\{e_je_k^*\}_{j=1,k=1}^{n_1,n_2}$, the coherence of $M$ is
$$\setlength\abovedisplayskip{5pt}\setlength\belowdisplayskip{5pt}
\mu(M)=\max\left\{\frac{n_1}{r}\max_{i\in\{1,...,n_1\}}\|U^*e_i\|_{l^2}^2,\frac{n_2}{r}\max_{i\in\{1,...,n_2\}}\|V^*e_i\|_{l^2}^2\right\}\in\left[1,\frac{\max\{n_1,n_2\}}{r}\right].
$$
It was shown in \cite{ding2020leave}\footnote{\cite{ding2020leave} considers real matrices but the result can be easily extended to complex matrices.} that if
$$\setlength\abovedisplayskip{5pt}\setlength\belowdisplayskip{5pt}
p\gtrsim \mu(M)r\log(\mu(M)r){\log(\max\{n_1,n_2\})}/{\min\{n_1,n_2\}},
$$
then with high probability\footnote{Meaning with probability at least $1 - c_1(n_1+n_2)^{-c_2}$ for constants $c_1, c_2 > 0$.}, there is an approximate dual certificate at $M$ with $\alpha_1\leq p/4$ and $\alpha_2\leq 1/2$, and $\|P_{T_M}p^{-1}P_{\Omega}P_{T_M}-P_{T_M}\|\leq 1/2.$ Let $Z\in T_M$, then
\begin{equation}\setlength\abovedisplayskip{5pt}\setlength\belowdisplayskip{5pt}
\label{fin_param_select}
\|P_{\Omega}Z\|_2^2=p\langle Z,P_{T_M}p^{-1}P_{\Omega}P_{T_M}Z\rangle\geq p\|Z\|_2^2(1-\|P_{T_M}p^{-1}P_{\Omega}P_{T_M}-P_{T_M}\|)\geq\frac{p}{2}\|Z\|_2^2.
\end{equation}
Hence, we take $\gamma=\sqrt{p/2}$ in \cref{RIP_mat_comp3} with $A=P_{\Omega}$ (treating outputs as vectors $P_\Omega(M)\in\mathbb{C}^{|\Omega|}$).

\begin{corollary}
If $p\gtrsim \mu(M)r\log(\mu(M)r)\frac{\log(\max\{n_1,n_2\})}{\min\{n_1,n_2\}},$ then with high probability the conditions of \cref{mat_comp_key_bound100} hold with $C_1\lesssim p^{-1/2}$ and $C_2$ bounded independently of all parameters. It follows that the conclusions of \cref{main_theorem,main_theorem_NN} hold.
\end{corollary}

Ignoring logarithmic factors, the above has a dimension scaling $C_1C_2\sim \sqrt{\min\{n_1,n_2\}}$. In general, it is impossible to eliminate this dimensional scaling \cite[Theorem 3.5]{krahmer2021convex}.

\subsection{Algorithmic considerations}
\label{sec:mat_alg_comp}

For matrix completion, the main computational burden of our algorithm is the step
$$\setlength\abovedisplayskip{5pt}\setlength\belowdisplayskip{5pt}
M^{(j+1)}=\mathrm{prox}_{\tau_1 \mathcal{J}}\big(\underbrace{M^{(j)}}_{\text{low-rank}}-\underbrace{\tau_1 A^* z_1^{(j)}}_{\text{sparse}}\big),\quad \text{where }\mathcal{J}(\cdot)=\|\cdot\|_1,
$$
which requires the application of the singular value thresholding operator in \eqref{mat_prox_form}. To reduce memory consumption, we store the iterates in low-rank factored SVD form $M^{(j)}=U^{(j)}\Sigma^{(j)}V^{(j)}$. The chosen rank of this factored form will be close to the approximate rank of $M$ when using our update rule below. The matrix $A^* z_1^{(j)}$ is sparse and its non-zero entries correspond to the indices in $\Omega$. It follows that $M^{(j)}-\tau_1 A^* z_1^{(j)}$ is a sum of a low-rank factorized matrix and a sparse matrix. Hence both it and its adjoint can be applied rapidly to vectors. We, therefore, make use of the PROPACK package \cite{larsen2004propack}, which uses iterative methods based on Lanczos bidiagonalization with partial re-orthogonalization for computing the first $r'$ singular vectors/values.\footnote{There are very efficient direct matrix factorization methods for calculating the SVD of matrices of moderate size (at most a few thousand). When the matrix is sparse, larger problems can be solved, however, the computational cost depends heavily upon the sparsity structure of the matrix. In general, for large matrices one has to resort to indirect iterative methods for calculating the leading singular vectors/values.} 
PROPACK only uses matrix-vector products, and has been found to be an efficient and stable package for computing the dominant singular values and singular vectors of large matrices. To use PROPACK in this scenario, we must supply a prediction of the dimension of the principal singular space whose singular values are larger than the given threshold. We provide an initial starting guess $r'$ ($5$ in our experiments), and at each iteration, we increase $r'$ by one for the following iteration if the dimension of the principal singular space is too small, or decrease by one if it is too large.

Following the arguments in \cref{mat_comp_theory}, we have used the parameters $C_1=\sqrt{n_1n_2/|\Omega|}$ and $C_2=1$ (as well as $\upsilon=e^{-1}$ and $\tau=1$). The value $C_2=1$ is based on empirical testing and has not been tuned. Smaller values of these constants will undoubtedly yield faster convergence for certain problems. We can use $L=1$ as a bound for $\|A\|$, but following \eqref{fin_param_select} under an incoherence assumption, we expect a local bound to scale as $\sqrt{|\Omega|/(n_1n_2)}$. We therefore took $L=\min\{1.6\sqrt{|\Omega|/(n_1n_2)},1\}$. Finally, we found that the non-ergodic version of WARPd performed slightly better than the ergodic version and so report computational results for the non-ergodic version.

\subsection{Current state-of-the-art methods}

Below we provide a brief summary of three state-of-the-art methods for matrix completion based on nuclear norm minimization, for which we compare our algorithm to in \cref{sec:mat_num_comp}. We do not claim that this is a complete list. Rather, we selected these methods for comparison based on their effectiveness, the variation of approaches, their popularity, and the availability of well-documented code.\footnote{The listed methods are all first-order methods. While nuclear norm minimization can be reformulated as a semidefinite program and solved by off-the-shelf interior point solvers, typically such methods have difficulty treating matrices larger than $n\sim 100$ because the complexity of computing each step grows quickly with $n$ (due to reliance on second-order information of the objective function). To overcome this scalability issue, the literature has focused on first-order methods.}

\subsubsection{Singular value thresholding (SVT)}\label{SVT_alg} SVT \cite{cai2010singular} performs shrinkage iterations to solve a smoothed problem (addition of an $\|M\|_2^2$ term), taking advantage of the sparsity and low rankness of the matrix iterates for approximation of the singular value thresholding operator. The algorithm uses low-rank SVD factorizations to reduce memory consumption and PROPACK. The code can be found at \url{https://statweb.stanford.edu/~candes/software/svt/code.html} and we use the default parameters suggested by \cite{cai2010singular} throughout. These parameters are based on empirical testing in \cite{cai2010singular} - we found the alternative parameters with guaranteed convergence (related to a smaller step size) to perform much worse than the results we report.

\subsubsection{Fixed point continuation with approximate SVD (FPCA)} FPCA \cite{ma2011fixed} has some similarities with SVT in that it makes use of shrinkage operations. However, the Lagrangian form of the problem (or nuclear norm regularized least-squares), $\min \mu \|M\|_1+\frac{1}{2}\|A(M)-b\|_{2}^2$, is solved with continuation for a sequence of parameters $\mu$. For the shrinkage operator, an approximate SVD is computed using a fast Monte Carlo algorithm \cite{drineas2006fast}. The code can be found at \url{https://www.math.ucdavis.edu/~sqma/FPCA.html}, and we use the given routine that selects parameters throughout.

\subsubsection{Augmented Lagrange multiplier method (ALM)} ALM \cite{lin2010augmented} is based on the augmented Lagrangian function $\|M\|_{1}+\langle Y,P_{\Omega}(M)-M-E\rangle+\mu_k\|P_{\Omega}(M)-M-E\|^2_2/2$ ($E$ is the difference between $M$ and $P_{\Omega}(M)$ and $Y$ is a dual variable). The general method of augmented Lagrange multipliers \cite{bertsekas2014constrained} applies simple updates rules for $M$, $Y$ and $E$ for a sequence of increasing $\mu_k$'s. In the case of matrix completion, a numerical difficulty is that for large $\mu_k$, the thresholding procedure (computed via an SVD) becomes numerically expensive. An inexact version of ALM was developed in \cite{lin2010augmented} to overcome this issue and shown to converge (the inexactness precludes a convergence rate analysis). The code can be found at \url{https://zhouchenlin.github.io/}, and we use the default parameters throughout. The code uses PROPACK and a simple update rule for the number of desired singular values.

\subsection{Numerical examples}
\label{sec:mat_num_comp} As our first experiment, we perform the following benchmark test often used in the literature \cite{lin2010augmented,ma2011fixed,cai2010singular}. We generate two independent standard Gaussian matrices $M_L\in\mathbb{R}^{n\times r}$, $M_R\in\mathbb{R}^{(n+20)\times r}$ and set $M = M_LM_R^*\in\mathbb{R}^{n\times(n+20)}$. Given $p\in(0,1)$, we then sample as described in \cref{mat_comp_theory}. We measure the time taken by each algorithm to achieve a relative error below $\texttt{tol}$, measured in the Frobenius norm. \cref{tab:comp_times} shows the results, where we have taken the average time over five runs for each parameter selection and we report NaN (highlighted in red) if convergence was not obtained after $5,000$ iterations or $100,000$s. For each parameter selection, we have highlighted the best average in green. Experiments were run on a modest desktop computer with a 3.4 GHz CPU. We have chosen a high accuracy tolerance $\texttt{tol}=10^{-6}$, as well as a moderate accuracy tolerance $\texttt{tol}=10^{-4}$.

In every case but one, WARPd is the fastest method, sometimes by an order of magnitude. Out of the other algorithms, ALM was the most reliable with only one NaN, but was often the slowest. A possible reason for the NaNs is the chosen value of $p$ - larger $p$ generally gives an easier problem with better convergence properties, though sometimes larger computational times due to the larger number of non-zero entries in the sparse matrices. We have deliberately shown results for varied $p$ to probe the robustness of algorithms for more challenging problems. In summary, \cref{tab:comp_times} shows clear benefits of the acceleration and demonstrates the speed and robustness of WARPd across a broad range of matrix sizes, ranks and sampling ratios.

\begin{table}
\small
\begin{center}
\begin{tabular}{ |r|c|l||r|r|r|r||r|r|r|r| }
\hline
 \multicolumn{1}{|c|}{\multirow{2}{*}{$n$}}&\multirow{2}{*}{$ r$}&\multicolumn{1}{c||}{\multirow{2}{*}{$p$}}& \multicolumn{4}{c||}{Time (s), $\texttt{tol}=10^{-4}$} &\multicolumn{4}{c|}{Time (s), $\texttt{tol}=10^{-6}$}\\
 \cline{4-11}
 & & & \multicolumn{1}{c|}{{}WARPd{}}& \multicolumn{1}{c|}{{}SVT{}} & \multicolumn{1}{c|}{{}FPCA{}} & \multicolumn{1}{c||}{{}ALM{}} & \multicolumn{1}{c|}{{WARPd}} & \multicolumn{1}{c|}{SVT} & \multicolumn{1}{c|}{{}FPCA{}} & \multicolumn{1}{c|}{{ALM}}\\
 \hline
\hline
\multirow{3}{*}{$1000$} & 10 & 0.14 &\cellcolor[HTML]{98FB98} 1.1 & 2.6\ & 1.9 & 4.5 &\cellcolor[HTML]{98FB98}1.9&3.8&2.8& 9.0\\\cdashline{2-11}
 & 30 & 0.40 & \cellcolor[HTML]{98FB98}3.4 & 7.2 & 4.6 & 7.0& \cellcolor[HTML]{98FB98}5.3 & 10.9& 6.4& 8.1\\\cdashline{2-11}
 & 60 & 0.57 & \cellcolor[HTML]{98FB98}6.2 & 14.3 & 8.4 & 8.5 & \cellcolor[HTML]{98FB98}10.2& 24.2& 12.1& 11.5\\
\hline
\multirow{3}{*}{$5000$} & 10 & {}0.02{}& \cellcolor[HTML]{98FB98}{}7.1{} & {}335.0{} & {}1093.5{} & {}203.7{} &\cellcolor[HTML]{98FB98}{}14.3{}&\multicolumn{1}{c|}{\cellcolor[HTML]{FA8072}{}\textbf{NaN}{}}&\multicolumn{1}{c|}{\cellcolor[HTML]{FA8072}{}\textbf{NaN}{}}&{}{}465.7{}\\\cdashline{2-11}
 & 30 & {}0.08{} &\cellcolor[HTML]{98FB98} {}39.0{}  & {}50.4{} & {}69.2{} & {}165.7{} &\cellcolor[HTML]{98FB98}{}57.9{}&{}83.7
&{}129.5{}&{}345.9{}\\\cdashline{2-11}
 & 60 & {}0.19{} & {}97.3{} &  {}156.0{} & \cellcolor[HTML]{98FB98}{}81.3{} & {}194.5{} &\cellcolor[HTML]{98FB98}{}160.5{}&{}257.3
&{}189.0{}&{}443.0{}\\
\hline
\multirow{3}{*}{$10000$} & 10 & {}0.01{}& \cellcolor[HTML]{98FB98}{}13.9{} & {}356.7{} & \multicolumn{1}{c|}{\cellcolor[HTML]{FA8072}{}\textbf{NaN}{}} & {}1335.7{} &\cellcolor[HTML]{98FB98}{}28.9{}&\multicolumn{1}{c|}{\cellcolor[HTML]{FA8072}{}\textbf{NaN}{}}&\multicolumn{1}{c|}{\cellcolor[HTML]{FA8072}{}\textbf{NaN}{}}&{}1787.1{}\\\cdashline{2-11}
 & 30 & {}0.04{} & \cellcolor[HTML]{98FB98}{}97.1{} &  {}1132.9{} & {}7312.0{} & {}1237.9{} &\cellcolor[HTML]{98FB98}{}164.5{}&{}1810.2
&\multicolumn{1}{c|}{\cellcolor[HTML]{FA8072}{}\textbf{NaN}{}}&{}1639.4{}\\\cdashline{2-11}
 & 60 & {}0.10{} & \cellcolor[HTML]{98FB98}{}289.5{}  & {}496.5{} & {}432.8{} & {}1160.7{} &\cellcolor[HTML]{98FB98}{}476.0{}&{}836.2
&{}507.9{}&{}1614.1{}\\
\hline
\multirow{3}{*}{$20000$} & 10 & {}0.005{}& \cellcolor[HTML]{98FB98}{}30.1{} & {}9114.3{}&\multicolumn{1}{c|}{\cellcolor[HTML]{FA8072}{}\textbf{NaN}{}}&{}4085.8{} &\cellcolor[HTML]{98FB98}{}64.7{}&\multicolumn{1}{c|}{\cellcolor[HTML]{FA8072}{}\textbf{NaN}{}}&\multicolumn{1}{c|}{\cellcolor[HTML]{FA8072}{}\textbf{NaN}{}}&\multicolumn{1}{c|}{\cellcolor[HTML]{FA8072}\textbf{NaN}{}}\\\cdashline{2-11}
 & 30 & {}0.020{} & \cellcolor[HTML]{98FB98}{}268.2{} &{}384.0{}&\multicolumn{1}{c|}{\cellcolor[HTML]{FA8072}{}\textbf{NaN}{}}&{}3732.3{} &{}\cellcolor[HTML]{98FB98}{}495.1{}{}&1283.7&\multicolumn{1}{c|}{\cellcolor[HTML]{FA8072}{}\textbf{NaN}{}}&{}9349.2{}\\\cdashline{2-11}
 & 60 & {}0.049{} & \cellcolor[HTML]{98FB98}{}1200.9{} & {}1296.8{}&\multicolumn{1}{c|}{\cellcolor[HTML]{FA8072}{}\textbf{NaN}{}}&{}6704.4{} &\cellcolor[HTML]{98FB98}{}2032.8{}&{}4461.5{}&\multicolumn{1}{c|}{\cellcolor[HTML]{FA8072}{}\textbf{NaN}{}}&{}9597.1{}\\
\hline
\end{tabular}
\end{center}
\normalsize
\vspace{2mm}
\caption{Computational times for a wide variety of parameter values for the low-rank random matrix recovery problem. All times are averaged over five runs and the best average for each experiment is shown in green. We report a `NaN' if convergence was not obtained after $5,000$ iterations or after $100,000$s.}\vspace{-2mm}
\label{tab:comp_times}
\end{table}

We now consider a real data example and an \textit{approximately} low-rank matrix. We took the data set \url{https://dataportal.orr.gov.uk/statistics/usage/estimates-of-station-usage/} of the locations of all 2569 railway stations in Great Britain. We considered two matrices, $M^{(1)}\in\mathbb{R}^{2569\times 2569}$ corresponding to the geodesic distance between all pairs of stations (rounded to the nearest 10m) and $M^{(2)}$ corresponding to the distance squared, both with $p=0.07$ (so that only approximately 7\% of the entries are sampled). \cref{fig:matrix_recovery_train} (left) shows the convergence for WARPd with $\epsilon=10^{-10}$ and $\delta=C_2\epsilon$. The accuracy of solutions of \cref{main_problem} is achieved in around $100$ and $60$ iterations respectively, with linear convergence down to this bound. \cref{fig:matrix_recovery_train} (right) shows the singular values of both matrices and explains why recovering $M^{(2)}$ is easier that $M^{(1)}$. For example,
the best rank six approximations of each matrix satisfy
$$\setlength\abovedisplayskip{5pt}\setlength\belowdisplayskip{5pt}
{\|M^{(1)}_6-M^{(1)}\|_2}/{\|M^{(1)}\|_2}\approx 0.0359,\quad {\|M^{(2)}_6-M^{(2)}\|_2}/{\|M^{(2)}\|_2}\approx 1.16\times 10^{-5}.
$$

\begin{figure}[!tbp]
  \centering
  \begin{minipage}[b]{0.48\textwidth}
    \begin{overpic}[width=\textwidth,trim={0mm 0mm 0mm 0mm},clip]{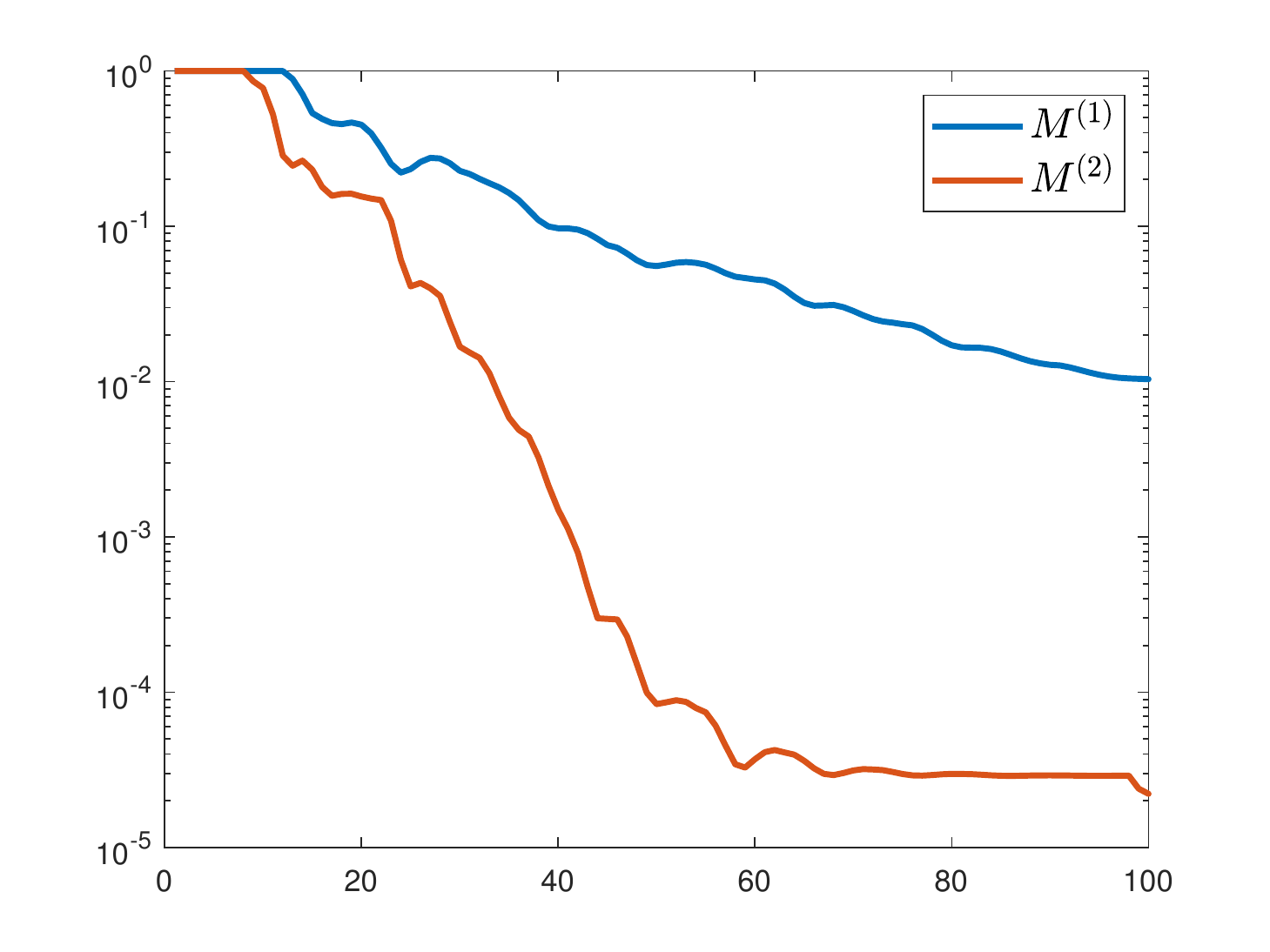}
		\put (42,0) {Iterations}
		\put (35,73) {Relative Error}
     \end{overpic}
  \end{minipage}
  \begin{minipage}[b]{0.48\textwidth}
    \begin{overpic}[width=\textwidth,trim={0mm 0mm 0mm 0mm},clip]{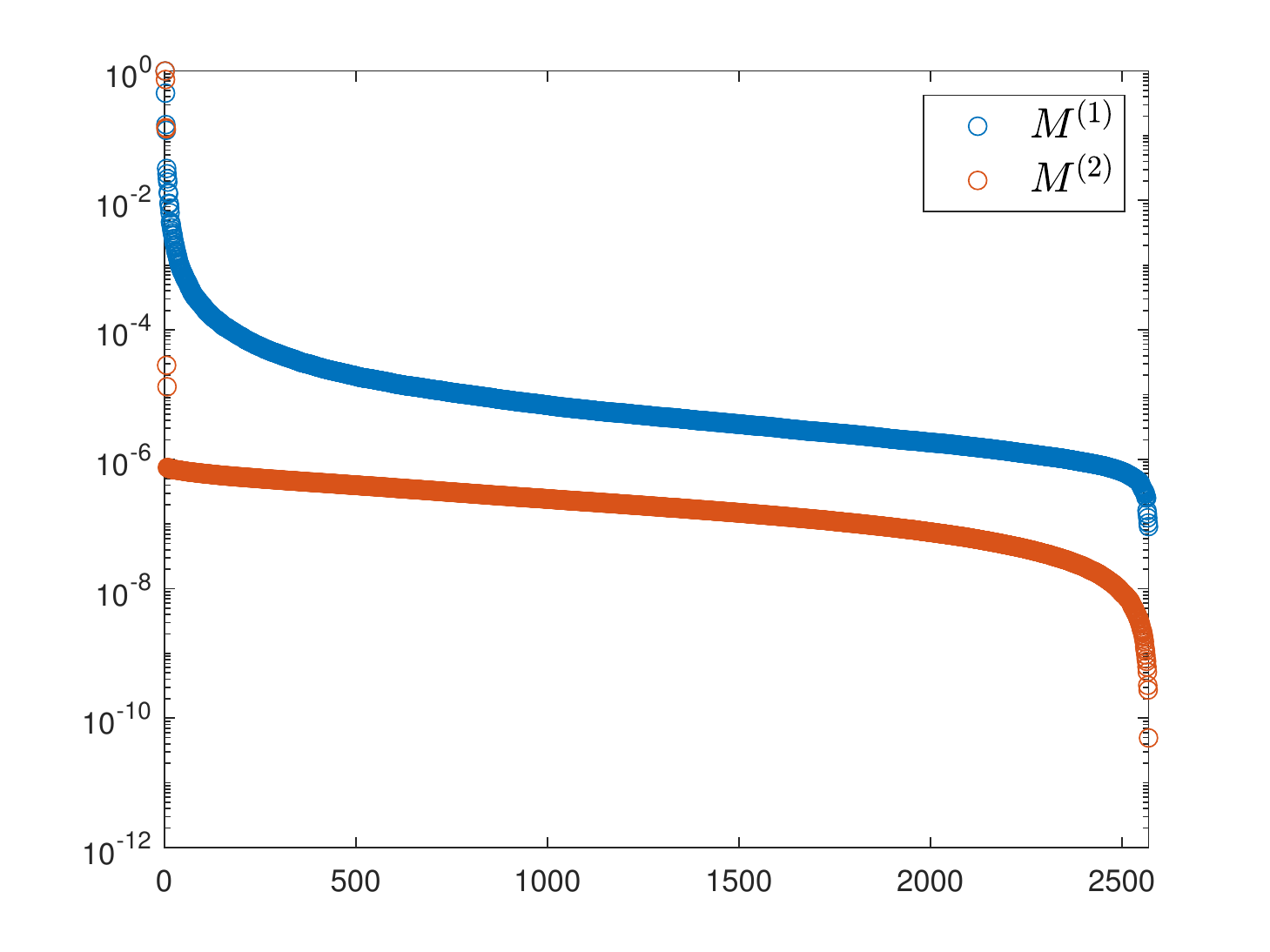}
		\put (25,73) {Relative Singular Values}
		\put (25,0) {Singular Value Number}
     \end{overpic}
  \end{minipage}
		  \caption{Results for the experiment with real data. Left: Relative error in the Frobenius norm. Right: Relative singular values (singular values normalized by the largest singular value) for each matrix.}
\label{fig:matrix_recovery_train}
\end{figure}

\section{Examples with non-trivial matrix $B$}
\label{application3}

Our final section of examples considers the case of non-trivial matrix $B$. We provide theorems for two common cases: $l^1$-analysis and TV regularization. We end with a numerical example involving shearlets and Total Generalized Variation (TGV), as well as iterative reweighting.

\subsection{Two example theorems}

\subsubsection{$l^1$-analysis with tight frames} We consider the problem (with $\mathcal{J}=0$)
\begin{equation}\setlength\abovedisplayskip{5pt}\setlength\belowdisplayskip{5pt}
\label{main_problem_last1}
\min_{x\in \mathbb{C}^N} \|D^*x\|_{l^1}\quad \text{s.t.}\quad\|Ax-b\|_{l^2}\leq \epsilon,
\end{equation}
where the columns of $D$ provide a tight frame.\footnote{Our results can be extended to frames that are not tight, but the analysis is more complicated.} Common examples of $D$ include oversampled DFT, Gabor frames, curvelets, shearlets, concatenations of orthonormal bases etc. Without loss of generality, we assume that $DD^*$ is the identity. See \cite{nam2013cosparse,duarte2013spectral,selesnick2009signal,elad2007analysis} for examples where an analysis approach \eqref{main_problem_last1} has advantages over a synthesis approach such as \eqref{nbibca}.

The following definition (which imposes no incoherence restriction on the dictionary) is a natural generalization of the well-known restricted isometry property.

\begin{definition}[\cite{candes2011compressed}]
\label{DRIP}
Let $s\in\mathbb{N}$ and let $\Sigma_s$ denote the union of all subspaces spanned by all subsets of $s$ columns of $D$. We say that the measurement matrix $A$ obeys the restricted isometry property adapted to $D$ (D-RIP) with constant $\delta_s=\delta_s(A,D)$ if
\begin{equation}\setlength\abovedisplayskip{5pt}\setlength\belowdisplayskip{5pt}
\label{DRIP_eq}
(1-\delta_s)\|v\|_{l^2}^2\leq \|Av\|_{l^2}^2\leq (1+\delta_s)\|v\|_{l^2}^2,\quad \forall v\in\Sigma_s.
\end{equation}
\end{definition}

For explicit examples where \cref{DRIP} holds, see \cite{candes2011compressed}. This definition yields the following theorem, whose proof is partly based on the arguments of \cite{candes2011compressed}. 

\begin{theorem}
\label{frame_result}
Let $t>s$ and set $\rho=s/t<1$. Suppose that
$$\setlength\abovedisplayskip{5pt}\setlength\belowdisplayskip{5pt}
\omega(A,D):=1-\rho-{\sqrt{\rho(1+\delta_t(A,D))}}/{\sqrt{1-\delta_{s+t}(A,D)}}>0, \text{ then \eqref{assumption} holds for \eqref{main_problem_last1} with}
$$
\begin{equation}\setlength\abovedisplayskip{5pt}\setlength\belowdisplayskip{5pt}
\label{dict_bounds}
\begin{split}
&C_1\!=\!{\sqrt{\rho^2+\rho}+1-\omega(A,D)}/({\omega(A,D)\sqrt{s}}), c(x,b)\!=\!2\sigma_{\text{\upshape{\textbf{s}}}}(D^*x)_{l^1}+C_2(\|Ax-b\|_{l^2}+\epsilon),\\&C_2\!=\!{\sqrt{s}\left(\sqrt{\rho^2+\rho}+1-\omega(A,D)\right)^{-1}}/{\sqrt{1-\delta_{s+t}(A,D)}}.
\end{split}
\end{equation}
It follows that the conclusions of \cref{main_theorem,main_theorem_NN} hold.
\end{theorem}

\begin{proof}
See \cref{frame_proof}.
\end{proof}

In summary, if $A$ satisfies the D-RIP, then WARPd provides accelerated recovery via \eqref{main_problem_last1}. Using $\delta_{t}<\delta_{s+t}$, the condition $\omega(A,D)>0$ is satisfied if $
\delta_{s+t}(A,D)<\frac{1+\rho^2-3\rho}{1+\rho^2-\rho}.$

\subsubsection{Total variation minimization}
\label{sec:TVTVTVT}

\!{}TV minimization \cite{rudin1992nonlinear} is widely used for image restoration tasks such as denoising, deblurring and inpainting \cite{chambolle2010introduction,chambolle2004algorithm,parisotto2020higher,bredies2010total}, as well as compressed sensing \cite{candes2006robust,lustig2007sparse}.  We consider a 2D signal $X\in\mathbb{C}^{\hat N\times \hat N}$. For vectorized $x=\mathrm{vect}(X)\in\mathbb{C}^N,N=\hat N^2$,  $\nabla\in\mathbb{C}^{2N\times N}$ is given by $\nabla=(\nabla_1\hspace{2mm} \nabla_2)^\top$ with
$$\setlength\abovedisplayskip{5pt}\setlength\belowdisplayskip{5pt}
[\nabla_1 X]_{i_1,i_2}=X_{i_1+1,i_2}-X_{i_1,i_2},\quad [\nabla_2 X]_{i_1,i_2}=X_{i_1,i_2+1}-X_{i_1,i_2},
$$
where $X_{\hat N+1,i_2}=X_{1,i_2},X_{i_1,\hat N+1}=X_{i_1,1}$. The periodic anisotropic TV-seminorm is given by
$$\setlength\abovedisplayskip{3pt}\setlength\belowdisplayskip{5pt}
\|X\|_{\mathrm{TV}}=\|x\|_{\mathrm{TV}}=\|\nabla x\|_{l^1}=\! \sum_{i_1,i_2=1}^{\hat N}\! {|X_{i_1+1,i_2}-X_{i_1,i_2}|+|X_{i_1,i_2+1}-X_{i_1,i_2}|}.
$$ We therefore consider the problem (with $\mathcal{J}=0$ and $B=\nabla$)
\begin{equation}\setlength\abovedisplayskip{5pt}\setlength\belowdisplayskip{5pt}
\label{main_problem_last2}
\min_{x\in \mathbb{C}^N} \|x\|_{\mathrm{TV}}\quad \text{s.t.}\quad\|Ax-b\|_{l^2}\leq \epsilon.
\end{equation}
For accurate and stable recovery guarantees for this problem, see \cite{needell2013stable,needell2013near}, which exploit the connection between the TV-seminorm and Haar wavelet coefficients. 
For sampling strategies for Fourier and binary measurements, see \cite{krahmer2013stable,poon2015role,adcock2021improved}. It is beyond the scope of this paper to discuss how all of these results fit into our framework so we consider the following general setting. Recall that a matrix $A\in\mathbb{C}^{m\times N}$ satisfies the restricted isometry property (RIP) of order $s$ if there exists $\delta_s(A)\in(0,1)$ such that for any $s$-sparse vector $z\in\mathbb{C}^N$,
$$\setlength\abovedisplayskip{5pt}\setlength\belowdisplayskip{5pt}
(1-\delta_s(A))\|z\|_{l^2}^2\leq \|Az\|_{l^2}^2\leq (1+\delta_s(A))\|z\|_{l^2}^2.
$$
The following theorem \cite[Theorem 17.17]{adcock2021compressive}\footnote{The result of \cite{adcock2021compressive} considered the isotropic version of the TV-seminorm. Both versions are equivalent up to a factor of $\sqrt{2}$ and hence the theoretical result is the same. We have considered the anisotropic version to fit into \cref{main_problem}. It is also straightforward to adapt WARPd to the isotropic TV-seminorm by adapting the proximal maps in \cref{alg:inner_iterations}.} provides a version of \eqref{assumption} (it is possible to chase down the explicit constants by studying the proof), and, to facilitate \cref{biiuvbakv}, we have stated the conclusion slightly differently to \cite{adcock2021compressive}. 

\begin{theorem}[\cite{adcock2021compressive}]
\label{TV_theorem}
Let $\hat N\geq s \geq 2$, $\Phi\in\mathbb{R}^{\hat N^2\times \hat N^2}$ be the matrix of the two-dimensional
discrete Haar wavelet sparsifying transform and $A\in\mathbb{C}^{m\times \hat N^2}$. Suppose that $A\Phi$ has the
RIP of order $t\gtrsim  s\log(\hat N)\log^2(2\hat N^2/s)$ with constant $\delta_t(A\Phi)\leq 1/2$. Then for any $x,\hat x\in \mathbb{C}^{\hat N^2}$,
\begin{equation*}\setlength\abovedisplayskip{5pt}\setlength\belowdisplayskip{5pt}
\|\hat x-x\|_{l^2}\lesssim \left(\|\hat x\|_{\mathrm{TV}}-\|x\|_{\mathrm{TV}}+\sigma_{\text{\upshape{\textbf{s}}}}(\nabla x)_{l^1}\right)/\sqrt{\smash[b]{{s\log(\hat N)}}}+(\|A\hat x-b\|_{l^2}-\epsilon)+(\|A x-b\|_{l^2}+\epsilon).
\end{equation*}
\end{theorem}

The following shows WARPd allows accelerated recovery via \eqref{main_problem_last2} if $A\Phi$ satisfies the RIP.

\begin{corollary}
\label{biiuvbakv}
Suppose that the conditions of \cref{TV_theorem} hold. Then \eqref{assumption} holds, with
$$\setlength\abovedisplayskip{5pt}\setlength\belowdisplayskip{5pt}
C_1\lesssim 1/\sqrt{\smash[b]s\log(\hat N)},\quad C_2\lesssim \sqrt{\smash[b]s\log(\hat N)},\quad c(x,b)=\sigma_{\text{\upshape{\textbf{s}}}}(\nabla x)_{l^1}+C_2(\|A x-b\|_{l^2}+\epsilon),
$$
for the problem \eqref{main_problem_last2}. It follows that the conclusions of \cref{main_theorem,main_theorem_NN} hold.
\end{corollary}

\subsection{A numerical example involving shearlets and TGV}
\label{sec:num_shear_TGV}

\begin{figure}[!tbp]
  \centering
	\vspace{1mm}
  \begin{minipage}[b]{0.32\textwidth}
	\centering
    \begin{overpic}[width=\textwidth,trim={25mm 0mm 25mm 0mm},clip]{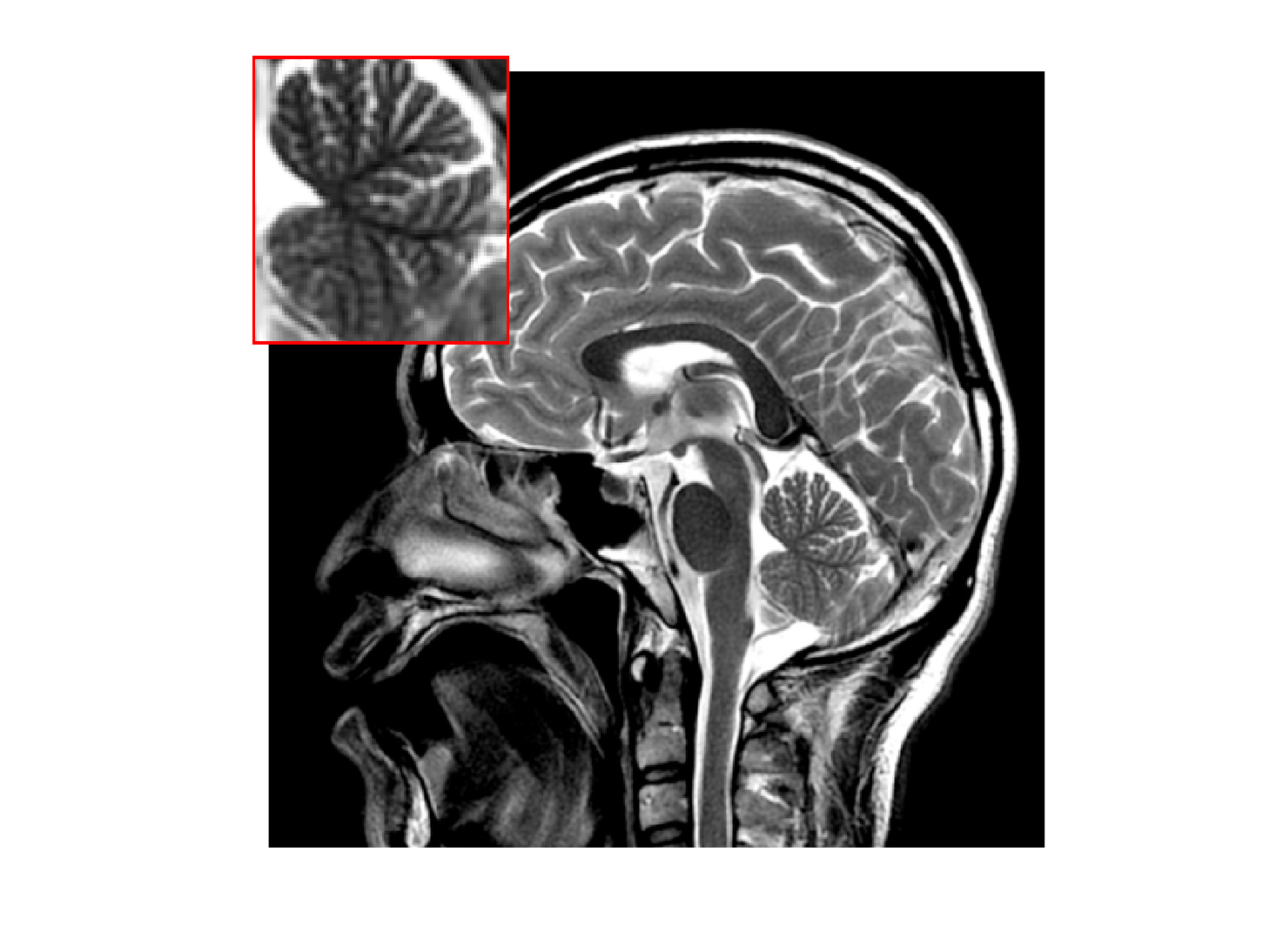}
    \put (30,100) {Test Image}
     \end{overpic}
  \end{minipage}
	\begin{minipage}[b]{0.32\textwidth}
	\centering
    \begin{overpic}[width=\textwidth,trim={25mm 0mm 25mm 0mm},clip]{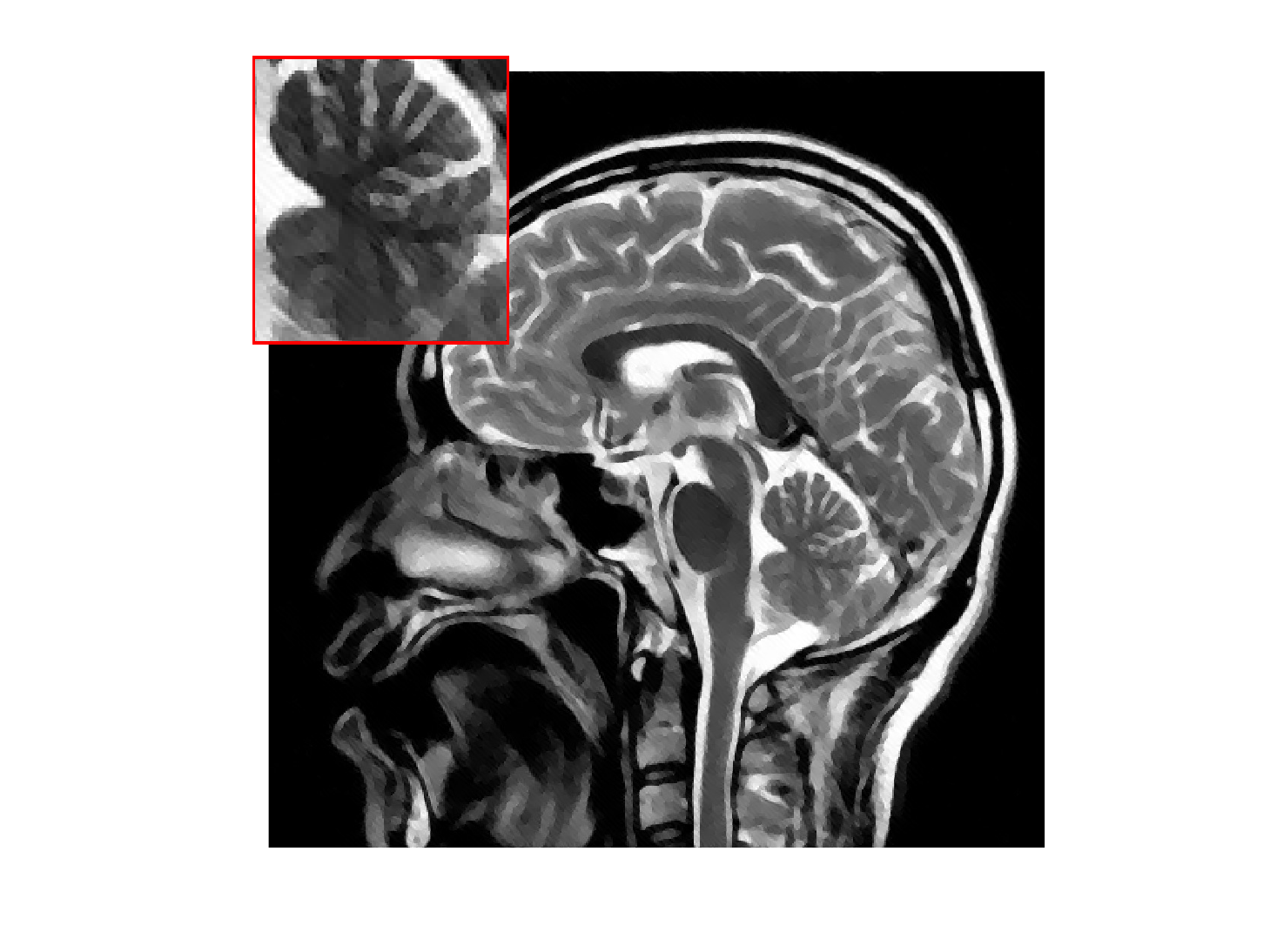}
    \put (10,100) {TV Recon., PSNR=27.9}
     \end{overpic}
  \end{minipage}
  \begin{minipage}[b]{0.32\textwidth}
    \begin{overpic}[width=\textwidth,trim={25mm 0mm 25mm 0mm},clip]{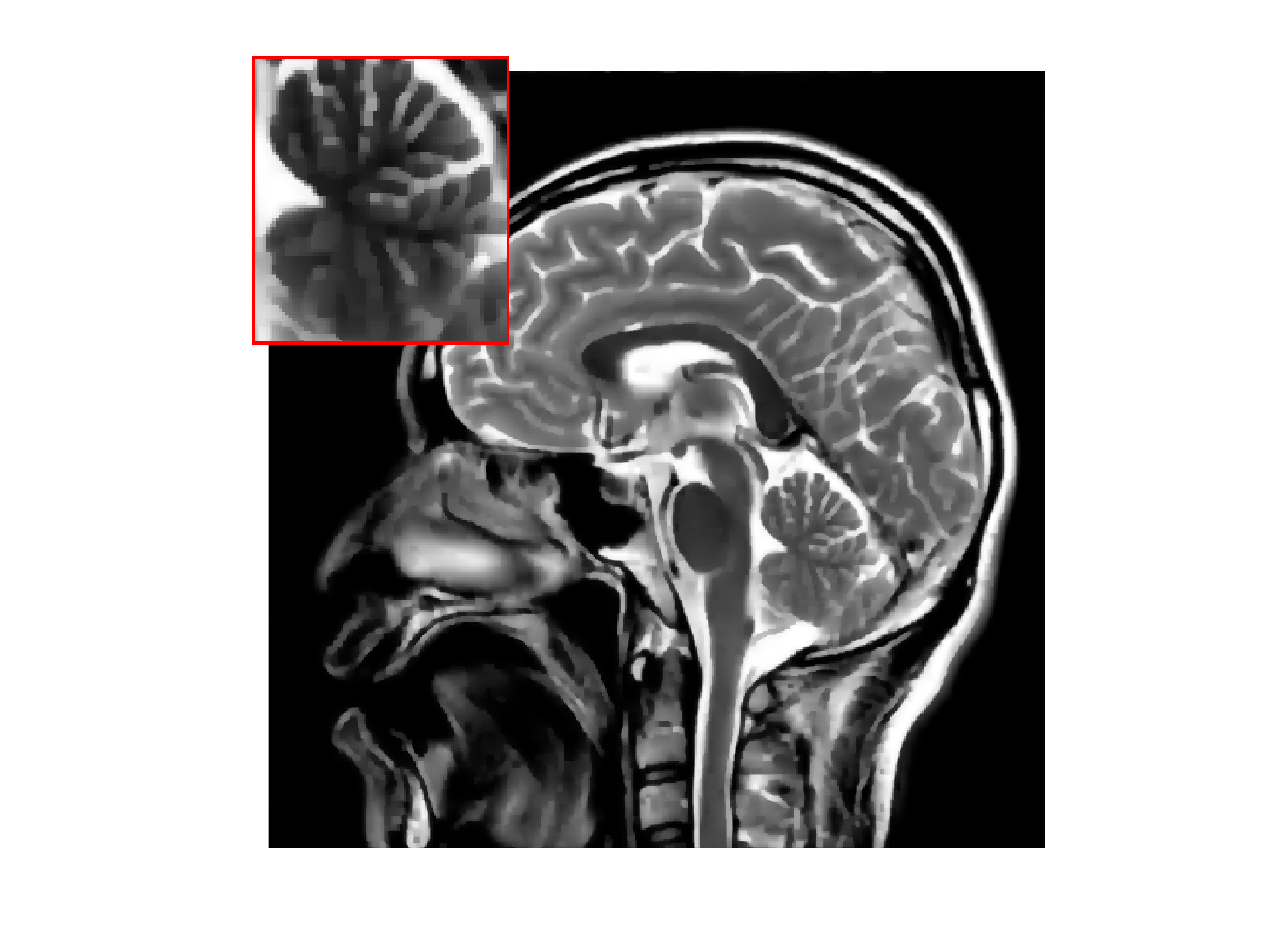}
		\put (8,100) {TGV Recon., PSNR=29.8}
     \end{overpic}
  \end{minipage}
	\vspace{-7mm}
	  \caption{Left: $512\times 512$ test image with pixel values scaled to $[0,1]$, the red box shows a zoomed in section. Middle: Converged reconstruction using TV. Right: Converged reconstruction using TGV (using $\alpha_0=0.4$ and $\alpha_1=0.2$, see \eqref{TGV_def} for meaning of parameters). Both reconstructions were computed using WARPd.}
\label{fig:analysis}
\end{figure}

The goal of this final numerical example is to demonstrate the flexibility of our algorithm, rather than promote a particular transform or regularizer. \cref{fig:analysis} (left) shows the used test image. We let $A$ be a DFT, $15\%$ subsampled according to an inverse square law density \cite{Felix_2014}. This sampling pattern has recently been shown to be optimal for TV reconstruction \cite{adcock2021improved}. The measurements are corrupted with $5\%$ Gaussian noise. We first use WARPd to reconstruct the image via \eqref{main_problem_last2}, the results are shown in \cref{fig:analysis} (middle). Whilst convergence to a solution of \eqref{main_problem_last2} was rapid, the reconstruction shows the typical artifacts of TV regularization such as staircasing. Next, we replace the TV regularizer with the (discrete) TGV regularizer \cite{bredies2010total}
\begin{equation}\setlength\abovedisplayskip{5pt}\setlength\belowdisplayskip{5pt}
\label{TGV_def}
\mathrm{TGV}_{\alpha}^2(x)=\min_{v\in\mathbb{C}^{2N}}\alpha_1\|\nabla x -v\|+\alpha_0\left\|\begin{pmatrix} \nabla_1 v_x & \frac{1}{2}\left(\nabla_2 v_x+\nabla_1 v_y\right)\\
\frac{1}{2}\left(\nabla_2 v_x+\nabla_1 v_y\right) & \nabla_2 v_y
\end{pmatrix}\right\|_{1},
\end{equation}
which has been proposed to improve on these issues by involving higher order derivatives. The improved results are shown in \cref{fig:analysis} (right). Again, convergence to a solution of the optimization problem was rapid.

\begin{figure}[!tbp]
  \centering
	\vspace{1mm}
  \begin{minipage}[b]{0.32\textwidth}
	\centering
    \begin{overpic}[width=\textwidth,trim={25mm 0mm 25mm 0mm},clip]{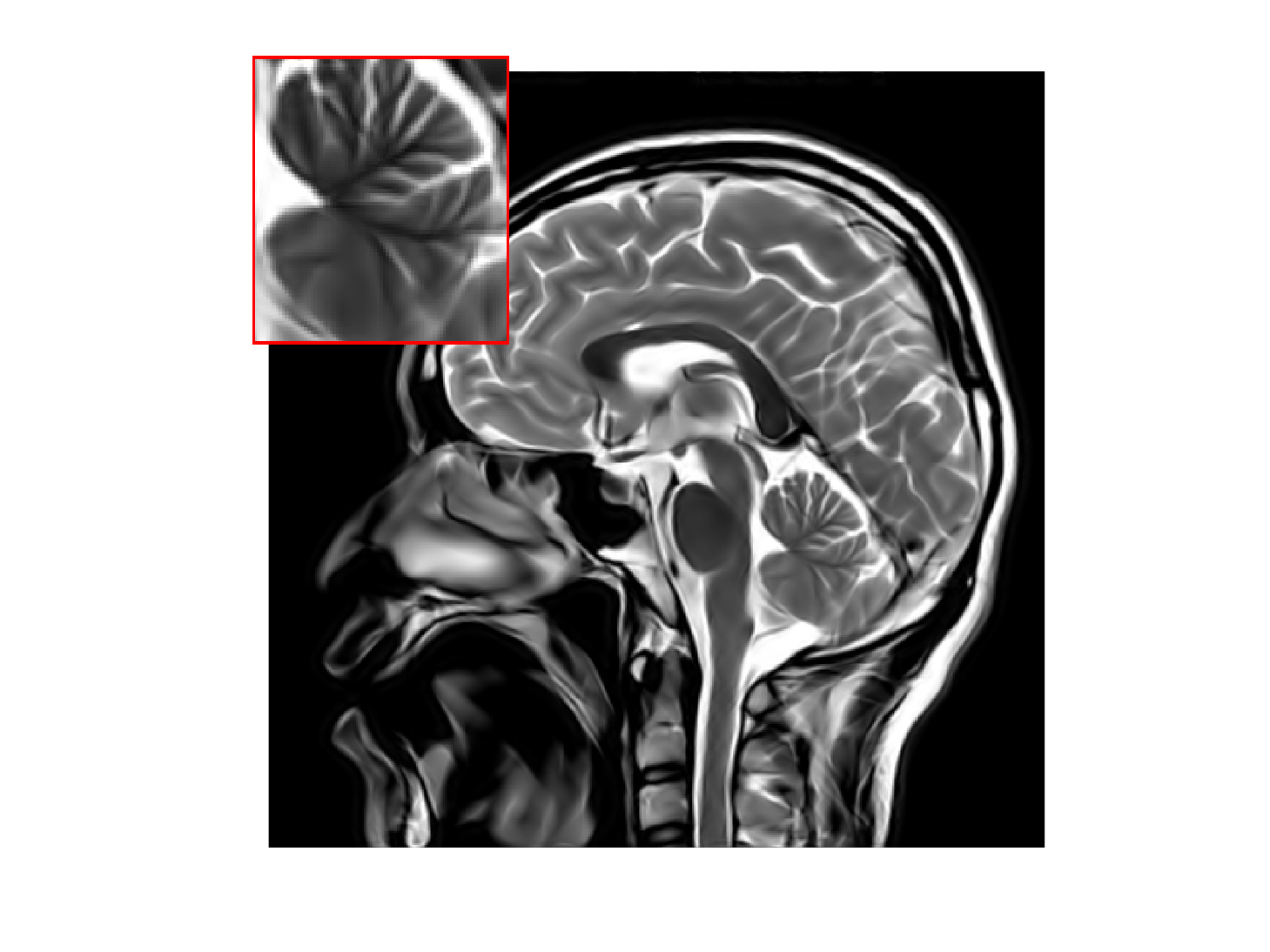}
    \put (12,100) {WARPd, PSNR=31.4}
     \end{overpic}
  \end{minipage}
	\begin{minipage}[b]{0.32\textwidth}
	\centering
    \begin{overpic}[width=\textwidth,trim={25mm 0mm 25mm 0mm},clip]{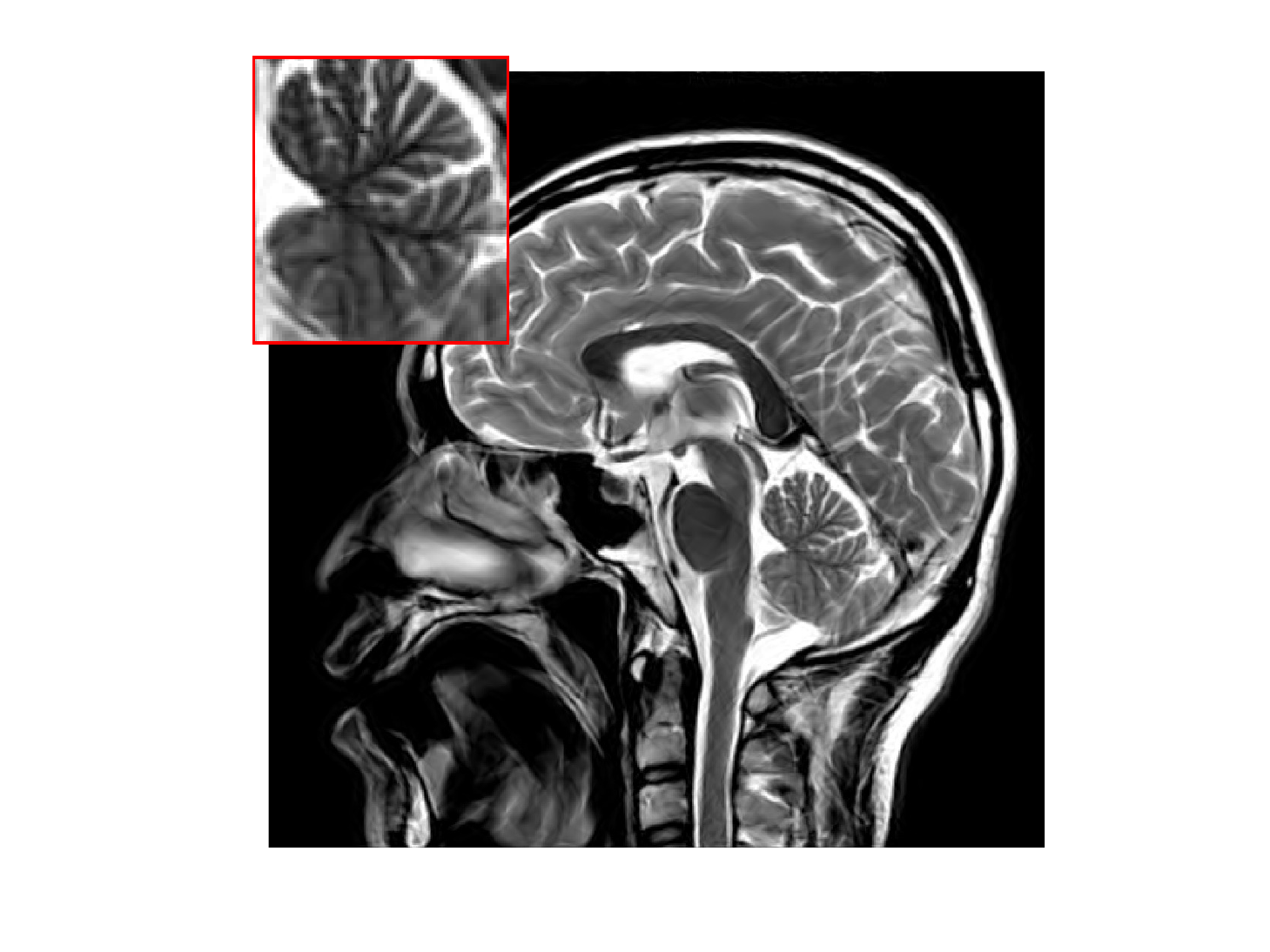}
    \put (10,100) {WARPdSR, PSNR=33.1}
     \end{overpic}
  \end{minipage}
  \begin{minipage}[b]{0.32\textwidth}
    \begin{overpic}[width=\textwidth,trim={0mm -30mm 0mm 0mm},clip]{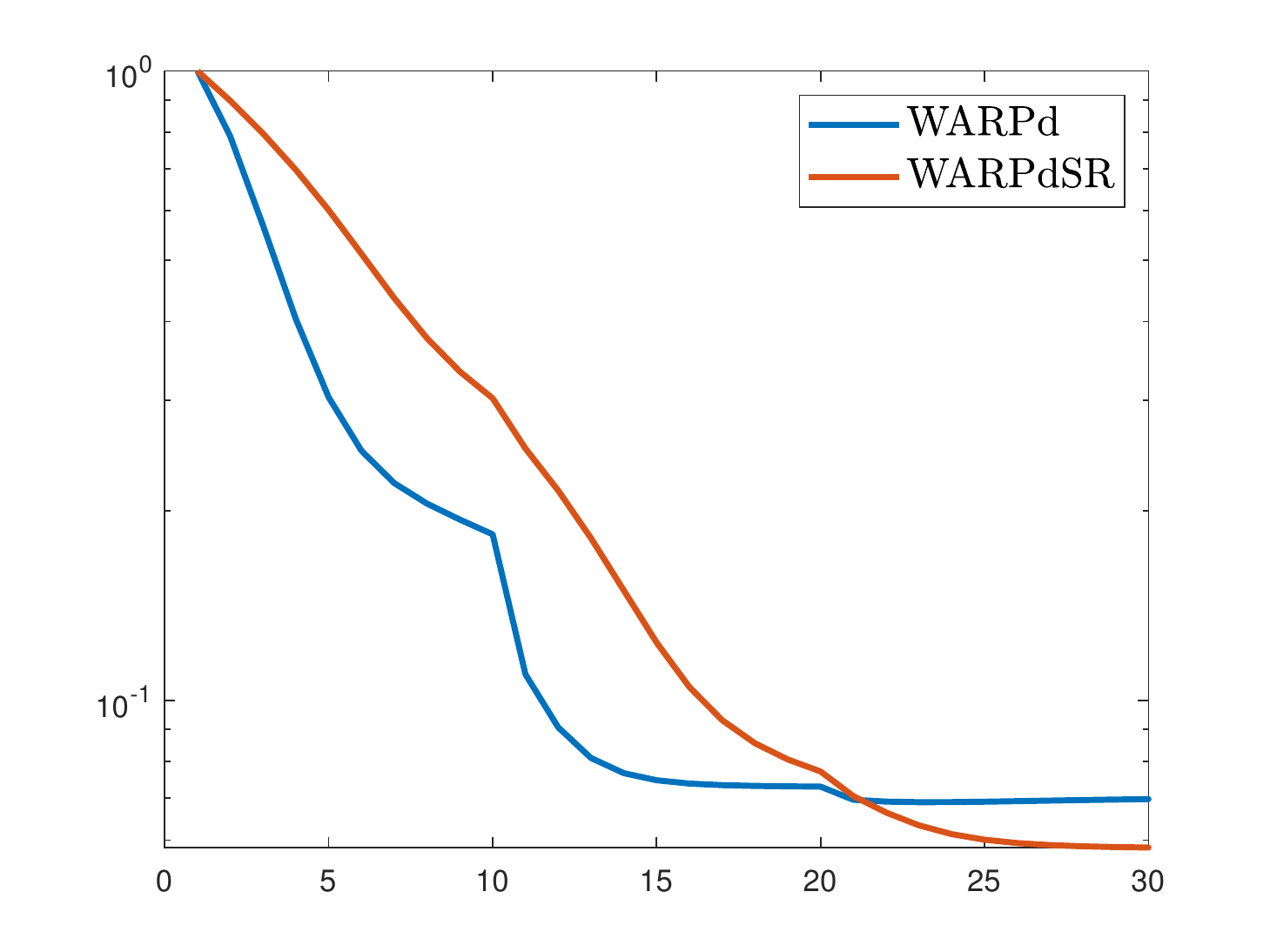}
		\put (0,35) {\rotatebox{90}{\small{Relative $l^2$ Error}}}
		\put (29,15) {\small{Inner iterations}}
     \end{overpic}
  \end{minipage}
	\vspace{-7mm}
	  \caption{Left: Reconstruction using WARPd and \eqref{complicated_problem}. Middle: Reconstruction using WARPdSR. Right: The relative $l^2$ error as a function of the number of inner iterations.}
\label{fig:analysis2}
\end{figure}

To improve the reconstruction further, we consider
\begin{equation}\setlength\abovedisplayskip{5pt}\setlength\belowdisplayskip{5pt}
\label{complicated_problem}
\min_{x\in \mathbb{C}^N} \|WD^*x\|_{l^1}+\mathrm{TGV}_{\alpha}^2(x)\quad \text{s.t.}\quad\|Ax-b\|_{l^2}\leq \epsilon,
\end{equation}
where $W$ denotes a diagonal scaling matrix and $D$ corresponds to a shearlet frame. We used the MATLAB shearlab package in this example, which can be found at \url{https://shearlab.math.lmu.de/}. Throughout this paper, we have so far only discussed numerical examples for WARPd, since the results of WARPdSR are similar (if not better). For completeness, in this example we also consider WARPdSR to demonstrate that it sometimes leads to better reconstructions. The weight matrix $W$ is updated after each call to \texttt{InnerIt} in \cref{alg:RR_primal_dual} (or \texttt{InnerItSR} in \cref{alg:RR_primal_dualB}) according to
\begin{equation}\setlength\abovedisplayskip{5pt}\setlength\belowdisplayskip{5pt}
\label{weight_update}
W_{jj}=\frac{1}{\max\{[D^*x]_{jj},10^{-5}\}}\times\frac{\sum_{k\in I(j)}\max\{[D^*x]_{kk},10^{-5}\}}{|I(j)|},
\end{equation}
where $I(j)$ denotes the set of indices corresponding to the shearlet scale containing the index $j$, and $x$ is the current reconstruction. We initialized the weights according to \eqref{weight_update} with $x=A^*b$. The update rule takes into account the difference in magnitudes of the shearlet coefficients of an image at different scales - see \cite{ahmad2015iteratively,ma2016multilevel} and \cite[Section 4.6]{adcock2021compressive} for the motivation of similar update rules. \cref{fig:analysis2} shows the reconstruction using WARPd (left) and WARPdSR (middle), which show a marked improvement on the results of \cref{fig:analysis}. Moreover, WARPdSR shows a better reconstruction of the fine details of the image. \cref{fig:analysis2} (right) plots the relative $l^2$ error between the reconstruction and the image against the number of inner iterations. Convergence is obtained in under $30$ iterations. This example demonstrates that WARPd and WARPdSR can easily handle more complicated mixed regularization problems such as \eqref{complicated_problem}.




\linespread{0.94}\selectfont{}
\bibliography{NNBib2}

\begin{thebibliography}{100}

\bibitem{rudin1992nonlinear}
{\sc L.~I. Rudin, S.~Osher, and E.~Fatemi}, {\em Nonlinear total variation
  based noise removal algorithms}, Phys. D, 60 (1992), pp.~259--268.

\bibitem{jin17}
{\sc K.~H. Jin, M.~T. McCann, E.~Froustey, and M.~Unser}, {\em Deep
  convolutional neural network for inverse problems in imaging}, IEEE Trans.
  Image Process., 26 (2017), pp.~4509--4522.

\bibitem{arridge2019solving}
{\sc S.~Arridge, P.~Maass, O.~{\"O}ktem, and C.-B. Sch{\"o}nlieb}, {\em Solving
  inverse problems using data-driven models}, Acta Numer., 28 (2019),
  pp.~1--174.

\bibitem{hastie2015statistical}
{\sc T.~Hastie, R.~Tibshirani, and M.~Wainwright}, {\em Statistical learning
  with sparsity: the {LASSO} and generalizations}, CRC press, 2015.

\bibitem{beck2009fast}
{\sc A.~Beck and M.~Teboulle}, {\em A fast iterative shrinkage-thresholding
  algorithm for linear inverse problems}, SIAM J. Imaging Sci., 2 (2009),
  pp.~183--202.

\bibitem{daubechies2004iterative}
{\sc I.~Daubechies, M.~Defrise, and C.~De~Mol}, {\em An iterative thresholding
  algorithm for linear inverse problems with a sparsity constraint}, Comm. Pure
  Appl. Math., 57 (2004), pp.~1413--1457.

\bibitem{chambolle2016introduction}
{\sc A.~Chambolle and T.~Pock}, {\em An introduction to continuous optimization
  for imaging}, Acta Numer., 25 (2016), pp.~161--319.

\bibitem{Mallat09}
{\sc S.~Mallat}, {\em A wavelet tour of signal processing: {T}he sparse way},
  Academic Press, third~ed., 2008.

\bibitem{candes2009exact}
{\sc E.~J. Cand{\`e}s and B.~Recht}, {\em Exact matrix completion via convex
  optimization}, Found. Comput. Math., 9 (2009), pp.~717--772.

\bibitem{chandrasekaran2012convex}
{\sc V.~Chandrasekaran, B.~Recht, P.~A. Parrilo, and A.~S. Willsky}, {\em The
  convex geometry of linear inverse problems}, Found. Comput. Math., 12 (2012),
  pp.~805--849.

\bibitem{recht2010guaranteed}
{\sc B.~Recht, M.~Fazel, and P.~A. Parrilo}, {\em Guaranteed minimum-rank
  solutions of linear matrix equations via nuclear norm minimization}, SIAM
  Rev., 52 (2010), pp.~471--501.

\bibitem{fannjiang2020numerics}
{\sc A.~Fannjiang and T.~Strohmer}, {\em The numerics of phase retrieval}, Acta
  Numer., 29 (2020), pp.~125--228.

\bibitem{adcock2021compressive}
{\sc B.~Adcock and A.~Hansen}, {\em Compressive Imaging: Structure, Sampling,
  Learning}, CUP, 2021.

\bibitem{candes2006robust}
{\sc E.~J. Cand{\`e}s, J.~Romberg, and T.~Tao}, {\em Robust uncertainty
  principles: {E}xact signal reconstruction from highly incomplete frequency
  information}, IEEE Trans. Inf. Theory, 52 (2006), pp.~489--509.

\bibitem{donoho2006compressed}
{\sc D.~L. Donoho}, {\em Compressed sensing}, IEEE Trans. Inf. Theory, 52
  (2006), pp.~1289--1306.

\bibitem{candes2006stable}
{\sc E.~J. Candes, J.~K. Romberg, and T.~Tao}, {\em Stable signal recovery from
  incomplete and inaccurate measurements}, Commun. Pure Appl. Math., 59 (2006),
  pp.~1207--1223.

\bibitem{chambolle2010introduction}
{\sc A.~Chambolle, V.~Caselles, D.~Cremers, M.~Novaga, and T.~Pock}, {\em An
  introduction to total variation for image analysis}, in Theoretical
  foundations and numerical methods for sparse recovery, de Gruyter, 2010,
  pp.~263--340.

\bibitem{becker2011nesta}
{\sc S.~Becker, J.~Bobin, and E.~J. Cand{\`e}s}, {\em {NESTA}: {A} fast and
  accurate first-order method for sparse recovery}, SIAM J. Imaging Sci., 4
  (2011), pp.~1--39.

\bibitem{nesterov2003introductory}
{\sc Y.~Nesterov}, {\em Introductory lectures on convex optimization: {A} basic
  course}, vol.~87, Springer Science \& Business Media, 2003.

\bibitem{donoho2008fast}
{\sc D.~L. Donoho and Y.~Tsaig}, {\em Fast solution of {$\ell_1$}-norm
  minimization problems when the solution may be sparse}, IEEE Trans Inf
  Theory, 54 (2008), pp.~4789--4812.

\bibitem{roulet2020computational}
{\sc V.~Roulet, N.~Boumal, and A.~d’Aspremont}, {\em Computational complexity
  versus statistical performance on sparse recovery problems}, Inf. Inference,
  9 (2020), pp.~1--32.

\bibitem{roulet2020sharpness}
{\sc V.~Roulet and A.~d'Aspremont}, {\em Sharpness, restart, and acceleration},
  SIAM J. Optim., 30 (2020), pp.~262--289.

\bibitem{wang2018image}
{\sc G.~Wang, J.~C. Ye, K.~Mueller, and J.~A. Fessler}, {\em Image
  reconstruction is a new frontier of machine learning}, IEEE Trans Med
  Imaging, 37 (2018), pp.~1289--1296.

\bibitem{hammernik2018learning}
{\sc K.~Hammernik, T.~Klatzer, E.~Kobler, M.~P. Recht, D.~K. Sodickson,
  T.~Pock, and F.~Knoll}, {\em Learning a variational network for
  reconstruction of accelerated {MRI} data}, Magn. Reson. Med., 79 (2018),
  pp.~3055--3071.

\bibitem{mccann2017convolutional}
{\sc M.~T. McCann, K.~H. Jin, and M.~Unser}, {\em Convolutional neural networks
  for inverse problems in imaging: {A} review}, IEEE Signal Process Mag., 34
  (2017), pp.~85--95.

\bibitem{bubba2019learning}
{\sc T.~A. Bubba, G.~Kutyniok, M.~Lassas, M.~M{\"a}rz, W.~Samek, S.~Siltanen,
  and V.~Srinivasan}, {\em Learning the invisible: {A} hybrid deep
  learning-shearlet framework for limited angle computed tomography}, Inverse
  Probl, 35 (2019), p.~064002.

\bibitem{kobler2020total}
{\sc E.~Kobler, A.~Effland, K.~Kunisch, and T.~Pock}, {\em Total deep
  variation: {A} stable regularizer for inverse problems}, arXiv:2006.08789,
  (2020).

\bibitem{huang2018some}
{\sc Y.~Huang et~al.}, {\em Some investigations on robustness of deep learning
  in limited angle tomography}, in MICCAI, Springer, 2018, pp.~145--153.

\bibitem{antun2020instabilities}
{\sc V.~Antun, F.~Renna, C.~Poon, B.~Adcock, and A.~C. Hansen}, {\em On
  instabilities of deep learning in image reconstruction and the potential
  costs of {AI}}, PNAS,  (2020).

\bibitem{finlayson2019adversarial}
{\sc S.~G. Finlayson, J.~D. Bowers, J.~Ito, J.~L. Zittrain, A.~L. Beam, and
  I.~S. Kohane}, {\em Adversarial attacks on medical machine learning},
  Science, 363 (2019), pp.~1287--1289.

\bibitem{knoll2020advancing}
{\sc F.~Knoll et~al.}, {\em Advancing machine learning for {MR} image
  reconstruction with an open competition: {Overview of the 2019 fastMRI
  challenge}}, Magn. Reson. Med.,  (2020).

\bibitem{muckley2020state}
{\sc M.~J. Muckley et~al.}, {\em {State-of-the-art Machine Learning {MRI}
  Reconstruction in 2020: {R}esults of the Second fast{MRI} Challenge}},
  arXiv:2012.06318,  (2020).

\bibitem{devore2020neural}
{\sc R.~DeVore, B.~Hanin, and G.~Petrova}, {\em Neural network approximation},
  Acta Numer., 30 (2021), pp.~327--444.

\bibitem{adcock2020gap}
{\sc B.~Adcock and N.~Dexter}, {\em The gap between theory and practice in
  function approximation with deep neural networks}, SIAM J. Math. Data Sci., 3
  (2021), pp.~624--655.

\bibitem{colbrookNN2021can}
{\sc M.~J. Colbrook, V.~Antun, and A.~C. Hansen}, {\em Can stable and accurate
  neural networks be computed? - {O}n the barriers of deep learning and
  {S}male's 18th problem}, arXiv:2101.08286,  (2021).

\bibitem{monga2019algorithm}
{\sc V.~Monga, Y.~Li, and Y.~C. Eldar}, {\em Algorithm unrolling:
  {I}nterpretable, efficient deep learning for signal and image processing},
  IEEE Signal Process Mag, 38 (2021), pp.~18--44.

\bibitem{chen2018theoretical}
{\sc X.~Chen, J.~Liu, Z.~Wang, and W.~Yin}, {\em Theoretical linear convergence
  of unfolded {ISTA} and its practical weights and thresholds}, in Adv. Neural
  Inf. Process. Syst, 2018, pp.~9061--9071.

\bibitem{liu2018alista}
{\sc J.~Liu, X.~Chen, Z.~Wang, and W.~Yin}, {\em {{ALISTA}: {A}nalytic weights
  are as good as learned weights in {LISTA}}}, in ICLR, 2018.

\bibitem{figueiredo2007gradient}
{\sc M.~A. Figueiredo, R.~D. Nowak, and S.~J. Wright}, {\em Gradient projection
  for sparse reconstruction: {A}pplication to compressed sensing and other
  inverse problems}, IEEE J Sel Top Signal Process, 1 (2007), pp.~586--597.

\bibitem{van2009probing}
{\sc E.~Van Den~Berg and M.~P. Friedlander}, {\em Probing the {P}areto frontier
  for basis pursuit solutions}, SIAM J. Sci. Comput., 31 (2009), pp.~890--912.

\bibitem{becker2011templates}
{\sc S.~R. Becker, E.~J. Cand{\`e}s, and M.~C. Grant}, {\em Templates for
  convex cone problems with applications to sparse signal recovery}, Math.
  Program. Comput., 3 (2011), p.~165.

\bibitem{nesterov2005smooth}
{\sc Y.~Nesterov}, {\em Smooth minimization of non-smooth functions}, Math.
  Program., 103 (2005), pp.~127--152.

\bibitem{beck2017first}
{\sc A.~Beck}, {\em First-Order Methods in Optimization}, SIAM, 2017.

\bibitem{agarwal2012fast}
{\sc A.~Agarwal, S.~Negahban, and M.~J. Wainwright}, {\em Fast global
  convergence of gradient methods for high-dimensional statistical recovery},
  Ann. Statist.,  (2012), pp.~2452--2482.

\bibitem{zhou2015ell_1}
{\sc Z.~Zhou, Q.~Zhang, and A.~M.-C. So}, {\em {$\ell_1$}, {$p$}-norm
  regularization: {E}rror bounds and convergence rate analysis of first-order
  methods}, in ICML, PMLR, 2015, pp.~1501--1510.

\bibitem{liang2014local}
{\sc J.~Liang, J.~M. Fadili, and G.~Peyr{\'e}}, {\em Local linear convergence
  of forward--backward under partial smoothness}, in NIPS, 2014.

\bibitem{necoara2019linear}
{\sc I.~Necoara, Y.~Nesterov, and F.~Glineur}, {\em Linear convergence of first
  order methods for non-strongly convex optimization}, Math. Program., 175
  (2019), pp.~69--107.

\bibitem{zhou2017unified}
{\sc Z.~Zhou and A.~M.-C. So}, {\em A unified approach to error bounds for
  structured convex optimization problems}, Math. Program., 165 (2017),
  pp.~689--728.

\bibitem{freund2018new}
{\sc R.~M. Freund and H.~Lu}, {\em New computational guarantees for solving
  convex optimization problems with first order methods, via a function growth
  condition measure}, Math. Program., 170 (2018), pp.~445--477.

\bibitem{nemirovskii1985optimal}
{\sc A.~S. Nemirovskii and Y.~E. Nesterov}, {\em Optimal methods of smooth
  convex minimization}, USSR Comput. Math. Math. Phys., 25 (1985), pp.~21--30.

\bibitem{hoffman1952approximate}
{\sc A.~J. Hoffman}, {\em On approximate solutions of systems of linear
  inequalities}, J. Research Nat. Bur. Standards, 49 (1952).

\bibitem{robinson1975application}
{\sc S.~M. Robinson}, {\em An application of error bounds for convex
  programming in a linear space}, SIAM J. Control, 13 (1975), pp.~271--273.

\bibitem{mangasarian1985condition}
{\sc O.~L. Mangasarian}, {\em A condition number for differentiable convex
  inequalities}, Math. Oper. Res., 10 (1985), pp.~175--179.

\bibitem{auslender1988global}
{\sc A.~Auslender and J.-P. Crouzeix}, {\em Global regularity theorems}, Math.
  Oper. Res., 13 (1988), pp.~243--253.

\bibitem{burke1993weak}
{\sc J.~V. Burke and M.~C. Ferris}, {\em Weak sharp minima in mathematical
  programming}, SIAM J. Control Optim., 31 (1993), pp.~1340--1359.

\bibitem{burke2002weak}
{\sc J.~Burke and S.~Deng}, {\em Weak sharp minima revisited {Part I}: basic
  theory}, Control Cybernet., 31 (2002), pp.~439--469.

\bibitem{lojasiewicz1963propriete}
{\sc S.~Lojasiewicz}, {\em Une propri{\'e}t{\'e} topologique des sous-ensembles
  analytiques r{\'e}els}, Les {\'e}quations aux d{\'e}riv{\'e}es partielles,
  117 (1963), pp.~87--89.

\bibitem{bolte2007lojasiewicz}
{\sc J.~Bolte, A.~Daniilidis, and A.~Lewis}, {\em The {\l{}}ojasiewicz
  inequality for nonsmooth subanalytic functions with applications to
  subgradient dynamical systems}, SIAM J. Optim., 17 (2007), pp.~1205--1223.

\bibitem{bolte2014proximal}
{\sc J.~Bolte, S.~Sabach, and M.~Teboulle}, {\em Proximal alternating
  linearized minimization for nonconvex and nonsmooth problems}, Math.
  Program., 146 (2014), pp.~459--494.

\bibitem{bolte2017error}
{\sc J.~Bolte, T.~P. Nguyen, J.~Peypouquet, and B.~W. Suter}, {\em From error
  bounds to the complexity of first-order descent methods for convex
  functions}, Math. Program., 165 (2017), pp.~471--507.

\bibitem{attouch2010proximal}
{\sc H.~Attouch, J.~Bolte, P.~Redont, and A.~Soubeyran}, {\em Proximal
  alternating minimization and projection methods for nonconvex problems: {A}n
  approach based on the {K}urdyka-{\l{}}ojasiewicz inequality}, Math. Oper.
  Res., 35 (2010), pp.~438--457.

\bibitem{frankel2015splitting}
{\sc P.~Frankel, G.~Garrigos, and J.~Peypouquet}, {\em Splitting methods with
  variable metric for {K}urdyka--{\l{}}ojasiewicz functions and general
  convergence rates}, J. Optim. Theory Appl., 165 (2015), pp.~874--900.

\bibitem{fercoq2016restarting}
{\sc O.~Fercoq and Z.~Qu}, {\em Restarting accelerated gradient methods with a
  rough strong convexity estimate}, arXiv:1609.07358,  (2016).

\bibitem{o2015adaptive}
{\sc B.~O’donoghue and E.~Candes}, {\em Adaptive restart for accelerated
  gradient schemes}, Found. Comput. Math., 15 (2015), pp.~715--732.

\bibitem{benlectures}
{\sc A.~Ben-Tal and A.~Nemirovski}, {\em Lectures on modern convex
  optimization},  (2020/2021), \url{https://www2.isye.gatech.edu/~nemirovs/}.

\bibitem{chambolle2016ergodic}
{\sc A.~Chambolle and T.~Pock}, {\em On the ergodic convergence rates of a
  first-order primal--dual algorithm}, Math. Program., 159 (2016),
  pp.~253--287.

\bibitem{chambolle2011first}
{\sc A.~Chambolle and T.~Pock}, {\em A first-order primal-dual algorithm for
  convex problems with applications to imaging}, J. Math. Imaging Vision, 40
  (2011), pp.~120--145.

\bibitem{esser2010general}
{\sc E.~Esser, X.~Zhang, and T.~F. Chan}, {\em A general framework for a class
  of first order primal-dual algorithms for convex optimization in imaging
  science}, SIAM J. Imaging Sci., 3 (2010).

\bibitem{pock2009algorithm}
{\sc T.~Pock, D.~Cremers, H.~Bischof, and A.~Chambolle}, {\em An algorithm for
  minimizing the {M}umford--{S}hah functional}, in IEEE Int Conf Comput Vis,
  IEEE, 2009, pp.~1133--1140.

\bibitem{chambolle2018stochastic}
{\sc A.~Chambolle, M.~J. Ehrhardt, P.~Richt{\'a}rik, and C.~Schonlieb}, {\em
  Stochastic primal-dual hybrid gradient algorithm with arbitrary sampling and
  imaging applications}, SIAM J. Optim., 28 (2018).

\bibitem{daskalakis2017training}
{\sc C.~Daskalakis, A.~Ilyas, V.~Syrgkanis, and H.~Zeng}, {\em Training {GAN}s
  with optimism}, arXiv:1711.00141,  (2017).

\bibitem{valkonen2017acceleration}
{\sc T.~Valkonen and T.~Pock}, {\em Acceleration of the {PDHGM} on partially
  strongly convex functions}, J. Math. Imaging Vision, 59 (2017), pp.~394--414.

\bibitem{applegate2021faster}
{\sc D.~Applegate, O.~Hinder, H.~Lu, and M.~Lubin}, {\em Faster first-order
  primal-dual methods for linear programming using restarts and sharpness},
  arXiv:2105.12715,  (2021).

\bibitem{rockafellar1976monotone}
{\sc R.~T. Rockafellar}, {\em Monotone operators and the proximal point
  algorithm}, SIAM J. Control Optim., 14 (1976), pp.~877--898.

\bibitem{opt_big}
{\sc A.~Bastounis, A.~C. Hansen, and V.~{Vla\v{c}i\'{c}}}, {\em The extended
  {S}male's 9th problem -- {O}n computational barriers and paradoxes in
  estimation, regularisation, computer-assisted proofs and learning}, 2021.

\bibitem{shalev2014understanding}
{\sc S.~Shalev-Shwartz and S.~Ben-David}, {\em Understanding machine learning},
  CUP, 2014.

\bibitem{belloni2011square}
{\sc A.~Belloni, V.~Chernozhukov, and L.~Wang}, {\em Square-root {LASSO}:
  pivotal recovery of sparse signals via conic programming}, Biometrika, 98
  (2011), pp.~791--806.

\bibitem{belloni2014pivotal}
{\sc A.~Belloni, V.~Chernozhukov, and L.~Wang}, {\em Pivotal estimation via
  square-root {LASSO} in nonparametric regression}, Ann. Statist., 42 (2014),
  pp.~757--788.

\bibitem{adcock2019correcting}
{\sc B.~Adcock, A.~Bao, and S.~Brugiapaglia}, {\em Correcting for unknown
  errors in sparse high-dimensional function approximation}, Numer. Math., 142
  (2019), pp.~667--711.

\bibitem{candes2006compressive}
{\sc E.~J. Cand{\`e}s et~al.}, {\em Compressive sampling}, in ICM, vol.~3,
  2006, pp.~1433--1452.

\bibitem{foucart2013invitation}
{\sc S.~Foucart and H.~Rauhut}, {\em A mathematical introduction to compressive
  sensing}, Springer, 2013.

\bibitem{adcock2017breaking}
{\sc B.~Adcock, A.~C. Hansen, C.~Poon, and B.~Roman}, {\em Breaking the
  coherence barrier: {A} new theory for compressed sensing}, in Forum Math.
  Sigma, vol.~5, CUP, 2017.

\bibitem{bastounis2017absence}
{\sc A.~Bastounis and A.~C. Hansen}, {\em On the absence of uniform recovery in
  many real-world applications of compressed sensing and the restricted
  isometry property and nullspace property in levels}, SIAM J. Imaging Sci., 10
  (2017), pp.~335--371.

\bibitem{adcock2019log}
{\sc B.~Adcock, S.~Brugiapaglia, and M.~King-Roskamp}, {\em Do log factors
  matter? {O}n optimal wavelet approximation and the foundations of compressed
  sensing}, Found. Comput. Math.,  (2021), pp.~1--61.

\bibitem{jones2016continuous}
{\sc A.~Jones, A.~Tamt{\"o}gl, I.~Calvo-Almaz{\'a}n, and A.~Hansen}, {\em
  Continuous compressed sensing for surface dynamical processes with helium
  atom scattering}, Sci. Rep., 6 (2016), p.~27776.

\bibitem{lustig2007sparse}
{\sc M.~Lustig, D.~Donoho, and J.~M. Pauly}, {\em Sparse {MRI}: The application
  of compressed sensing for rapid {MR} imaging}, Magn. Reson. Med., 58 (2007),
  pp.~1182--1195.

\bibitem{Boyer_2016}
{\sc J.~{Bigot}, C.~{Boyer}, and P.~{Weiss}}, {\em An analysis of block
  sampling strategies in compressed sensing}, IEEE Trans. Inf. Theory, 62
  (2016), pp.~2125--2139.

\bibitem{BOYER_ACHA_2019}
{\sc C.~Boyer, J.~Bigot, and P.~Weiss}, {\em Compressed sensing with structured
  sparsity and structured acquisition}, Appl. Comput. Harmon. Anal., 46 (2019),
  pp.~312 -- 350.

\bibitem{Kutyniok_Lim_2018}
{\sc G.~Kutyniok and W.-Q. Lim}, {\em Optimal compressive imaging of fourier
  data}, SIAM J. Imag. Sci., 11 (2018), pp.~507--546.

\bibitem{Felix_2014}
{\sc F.~{Krahmer} and R.~{Ward}}, {\em Stable and robust sampling strategies
  for compressive imaging}, IEEE Trans. Image Process., 23 (2014),
  pp.~612--622.

\bibitem{traonmilin2018stable}
{\sc Y.~Traonmilin and R.~Gribonval}, {\em Stable recovery of low-dimensional
  cones in {H}ilbert spaces: One {RIP} to rule them all}, Appl. Comput. Harmon.
  Anal., 45 (2018), pp.~170--205.

\bibitem{li2019compressed}
{\sc C.~Li and B.~Adcock}, {\em Compressed sensing with local structure:
  uniform recovery guarantees for the sparsity in levels class}, Appl. Comput.
  Harmon. Anal., 46 (2019), pp.~453--477.

\bibitem{eldar2010block}
{\sc Y.~C. Eldar, P.~Kuppinger, and H.~Bolcskei}, {\em Block-sparse signals:
  Uncertainty relations and efficient recovery}, IEEE Transactions on Signal
  Processing, 58 (2010), pp.~3042--3054.

\bibitem{friedlander2012recovering}
{\sc M.~P. Friedlander, H.~Mansour, R.~Saab, and {\"O}.~Yilmaz}, {\em
  Recovering compressively sampled signals using partial support information},
  IEEE Trans. Inf. Theory, 58 (2012), pp.~1122--1134.

\bibitem{gross2011recovering}
{\sc D.~Gross}, {\em Recovering low-rank matrices from few coefficients in any
  basis}, IEEE Trans Inf Theory, 57 (2011), pp.~1548--1566.

\bibitem{liu2011universal}
{\sc Y.-K. Liu}, {\em Universal low-rank matrix recovery from {P}auli
  measurements}, Adv. Neural Inf. Process. Syst, 24 (2011), pp.~1638--1646.

\bibitem{kueng2017low}
{\sc R.~Kueng, H.~Rauhut, and U.~Terstiege}, {\em Low rank matrix recovery from
  rank one measurements}, Appl. Comput. Harmon. Anal., 42 (2017), pp.~88--116.

\bibitem{krahmer2021convex}
{\sc F.~Krahmer and D.~St{\"o}ger}, {\em On the convex geometry of blind
  deconvolution and matrix completion}, Comm. Pure Appl. Math., 74 (2021),
  pp.~790--832.

\bibitem{candes2011tight}
{\sc E.~J. Candes and Y.~Plan}, {\em Tight oracle inequalities for low-rank
  matrix recovery from a minimal number of noisy random measurements}, IEEE
  Trans Inf Theory, 57 (2011), pp.~2342--2359.

\bibitem{gu2014weighted}
{\sc S.~Gu, L.~Zhang, W.~Zuo, and X.~Feng}, {\em Weighted nuclear norm
  minimization with application to image denoising}, in IEEE Conf Comput Vis
  Pattern Recognit, 2014, pp.~2862--2869.

\bibitem{kabanava2016stable}
{\sc M.~Kabanava, R.~Kueng, H.~Rauhut, and U.~Terstiege}, {\em Stable low-rank
  matrix recovery via null space properties}, Inf. Inference, 5 (2016),
  pp.~405--441.

\bibitem{lvovsky2009continuous}
{\sc A.~I. Lvovsky and M.~G. Raymer}, {\em Continuous-variable optical
  quantum-state tomography}, Rev. Modern Phys., 81 (2009), p.~299.

\bibitem{gross2010quantum}
{\sc D.~Gross, Y.-K. Liu, S.~T. Flammia, S.~Becker, and J.~Eisert}, {\em
  Quantum state tomography via compressed sensing}, Phys. Rev. Lett., 105
  (2010), p.~150401.

\bibitem{schwemmer2014experimental}
{\sc C.~Schwemmer, G.~T{\'o}th, A.~Niggebaum, T.~Moroder, D.~Gross,
  O.~G{\"u}hne, and H.~Weinfurter}, {\em Experimental comparison of efficient
  tomography schemes for a six-qubit state}, Phys. Rev. Lett., 113 (2014),
  p.~040503.

\bibitem{riofrio2017experimental}
{\sc C.~A. Riofrio, D.~Gross, S.~T. Flammia, T.~Monz, D.~Nigg, R.~Blatt, and
  J.~Eisert}, {\em Experimental quantum compressed sensing for a seven-qubit
  system}, Nature comm., 8 (2017), pp.~1--8.

\bibitem{cai2010singular}
{\sc J.-F. Cai, E.~J. Cand{\`e}s, and Z.~Shen}, {\em A singular value
  thresholding algorithm for matrix completion}, SIAM J. Optim., 20 (2010),
  pp.~1956--1982.

\bibitem{trefethen1997numerical}
{\sc L.~N. Trefethen and D.~Bau~III}, {\em Numerical linear algebra}, vol.~50,
  Siam, 1997.

\bibitem{shechtman2015phase}
{\sc Y.~Shechtman, Y.~C. Eldar, O.~Cohen, H.~N. Chapman, J.~Miao, and
  M.~Segev}, {\em Phase retrieval with application to optical imaging: a
  contemporary overview}, IEEE Signal Process Mag, 32 (2015).

\bibitem{candes2013phaselift}
{\sc E.~J. Candes, T.~Strohmer, and V.~Voroninski}, {\em Phaselift: {E}xact and
  stable signal recovery from magnitude measurements via convex programming},
  Comm. Pure Appl. Math., 66 (2013).

\bibitem{candes2015phase}
{\sc E.~J. Candes, Y.~C. Eldar, T.~Strohmer, and V.~Voroninski}, {\em Phase
  retrieval via matrix completion}, SIAM Rev., 57 (2015), pp.~225--251.

\bibitem{auslender2006interior}
{\sc A.~Auslender and M.~Teboulle}, {\em Interior gradient and proximal methods
  for convex and conic optimization}, SIAM J. Optim., 16 (2006), pp.~697--725.

\bibitem{candes2010power}
{\sc E.~J. Cand{\`e}s and T.~Tao}, {\em The power of convex relaxation:
  Near-optimal matrix completion}, IEEE Trans Inf Theory, 56 (2010),
  pp.~2053--2080.

\bibitem{recht2011simpler}
{\sc B.~Recht}, {\em A simpler approach to matrix completion}, J. Mach. Learn.
  Res., 12 (2011).

\bibitem{udell2019big}
{\sc M.~Udell and A.~Townsend}, {\em Why are big data matrices approximately
  low rank?}, SIAM J. Math. Data Sci., 1 (2019), pp.~144--160.

\bibitem{rennie2005fast}
{\sc J.~D. Rennie and N.~Srebro}, {\em Fast maximum margin matrix factorization
  for collaborative prediction}, in ICML, 2005, pp.~713--719.

\bibitem{koren2009matrix}
{\sc Y.~Koren, R.~Bell, and C.~Volinsky}, {\em Matrix factorization techniques
  for recommender systems}, Computer, 42 (2009), pp.~30--37.

\bibitem{chen2004recovering}
{\sc P.~Chen and D.~Suter}, {\em Recovering the missing components in a large
  noisy low-rank matrix: Application to {SFM}}, IEEE Trans Pattern Anal Mach
  Intell, 26 (2004), pp.~1051--1063.

\bibitem{tomasi1992shape}
{\sc C.~Tomasi and T.~Kanade}, {\em Shape and motion from image streams under
  orthography: a factorization method}, Int. J. Comput. Vis., 9 (1992),
  pp.~137--154.

\bibitem{amit2007uncovering}
{\sc Y.~Amit, M.~Fink, N.~Srebro, and S.~Ullman}, {\em Uncovering shared
  structures in multiclass classification}, in ICML, 2007, pp.~17--24.

\bibitem{evgeniou2007multi}
{\sc A.~Evgeniou and M.~Pontil}, {\em Multi-task feature learning}, Adv. Neural
  Inf. Process. Syst, 19 (2007).

\bibitem{ding2020leave}
{\sc L.~Ding and Y.~Chen}, {\em Leave-one-out approach for matrix completion:
  {P}rimal and dual analysis}, IEEE Trans Inf Theory, 66 (2020),
  pp.~7274--7301.

\bibitem{larsen2004propack}
{\sc R.~M. Larsen}, {\em Propack-software for large and sparse {SVD}
  calculations}, Available online. URL http://sun. stanford. edu/rmunk/PROPACK,
   (2004), pp.~2008--2009.

\bibitem{ma2011fixed}
{\sc S.~Ma, D.~Goldfarb, and L.~Chen}, {\em Fixed point and {B}regman iterative
  methods for matrix rank minimization}, Math. Program., 128 (2011),
  pp.~321--353.

\bibitem{drineas2006fast}
{\sc P.~Drineas, R.~Kannan, and M.~W. Mahoney}, {\em Fast {M}onte {C}arlo
  algorithms for matrices {II}: {C}omputing a low-rank approximation to a
  matrix}, SIAM J. Comput., 36 (2006), pp.~158--183.

\bibitem{lin2010augmented}
{\sc Z.~Lin, M.~Chen, and Y.~Ma}, {\em The augmented {L}agrange multiplier
  method for exact recovery of corrupted low-rank matrices}, arXiv:1009.5055,
  (2010).

\bibitem{bertsekas2014constrained}
{\sc D.~P. Bertsekas}, {\em Constrained optimization and Lagrange multiplier
  methods}, Academic press, 2014.

\bibitem{nam2013cosparse}
{\sc S.~Nam, M.~E. Davies, M.~Elad, and R.~Gribonval}, {\em The cosparse
  analysis model and algorithms}, Appl. Comput. Harmon. Anal., 34 (2013),
  pp.~30--56.

\bibitem{duarte2013spectral}
{\sc M.~F. Duarte and R.~G. Baraniuk}, {\em Spectral compressive sensing},
  Appl. Comput. Harmon. Anal., 35 (2013), pp.~111--129.

\bibitem{selesnick2009signal}
{\sc I.~W. Selesnick and M.~A. Figueiredo}, {\em Signal restoration with
  overcomplete wavelet transforms: Comparison of analysis and synthesis
  priors}, in Wavelets XIII, vol.~7446, International Society for Optics and
  Photonics, 2009, p.~74460D.

\bibitem{elad2007analysis}
{\sc M.~Elad, P.~Milanfar, and R.~Rubinstein}, {\em Analysis versus synthesis
  in signal priors}, Inverse problems, 23 (2007), p.~947.

\bibitem{candes2011compressed}
{\sc E.~J. Candes, Y.~C. Eldar, D.~Needell, and P.~Randall}, {\em Compressed
  sensing with coherent and redundant dictionaries}, Appl. Comput. Harmon.
  Anal., 31 (2011), pp.~59--73.

\bibitem{chambolle2004algorithm}
{\sc A.~Chambolle}, {\em An algorithm for total variation minimization and
  applications}, J. Math. Imaging Vision, 20 (2004), pp.~89--97.

\bibitem{parisotto2020higher}
{\sc S.~Parisotto, J.~Lellmann, S.~Masnou, and C.~Schonlieb}, {\em Higher-order
  total directional variation: Imaging applications}, SIAM J. Imaging Sci., 13
  (2020), pp.~2063--2104.

\bibitem{bredies2010total}
{\sc K.~Bredies, K.~Kunisch, and T.~Pock}, {\em Total generalized variation},
  SIAM J. Imaging Sci., 3 (2010), pp.~492--526.

\bibitem{needell2013stable}
{\sc D.~Needell and R.~Ward}, {\em Stable image reconstruction using total
  variation minimization}, SIAM J. Imaging Sci., 6 (2013), pp.~1035--1058.

\bibitem{needell2013near}
{\sc D.~Needell and R.~Ward}, {\em Near-optimal compressed sensing guarantees
  for total variation minimization}, IEEE Trans. Image Process., 22 (2013),
  pp.~3941--3949.

\bibitem{krahmer2013stable}
{\sc F.~Krahmer and R.~Ward}, {\em Stable and robust sampling strategies for
  compressive imaging}, IEEE Trans. Image Process., 23 (2013), pp.~612--622.

\bibitem{poon2015role}
{\sc C.~Poon}, {\em On the role of total variation in compressed sensing}, SIAM
  J. Imaging Sci., 8 (2015).

\bibitem{adcock2021improved}
{\sc B.~Adcock, N.~Dexter, and Q.~Xu}, {\em Improved recovery guarantees and
  sampling strategies for {TV} minimization in compressive imaging}, SIAM J.
  Imaging Sci., 14 (2021), pp.~1149--1183.

\bibitem{ahmad2015iteratively}
{\sc R.~Ahmad and P.~Schniter}, {\em Iteratively reweighted {$\ell_1$}
  approaches to sparse composite regularization}, IEEE transactions on
  computational imaging, 1 (2015), pp.~220--235.

\bibitem{ma2016multilevel}
{\sc J.~Ma and M.~M{\"a}rz}, {\em A multilevel based reweighting algorithm with
  joint regularizers for sparse recovery}, arXiv:1604.06941,  (2016).

\end{thebibliography}
\bibliographystyle{siamplain}



\section*{Supplementary material}

# SUPPLEMENTARY MATERIALS: WARPd: A linearly convergent first-order method for inverse problems with approximate sharpness conditions[*]

Matthew J. Colbrook[†]

**SM1. Proof of Theorem 2.2.** For $\eta > 0$, define

(SM1.1) $\quad \widehat{G}_\eta(\hat{x}, x, b) := \underbrace{\mathcal{J}(\hat{x}) + \|B\hat{x}\|_{l^1} + \eta\|A\hat{x} - b\|_{l^2} - \mathcal{J}(x) - \|Bx\|_{l^1} - \eta\|Ax - b\|_{l^2}}_{\text{objective function difference with } \lambda = \eta^{-1}}.$

Recall that (2.21) is equivalent to

$$\min_{u \in \mathbb{R}^{2N}} \max_{y_1 \in \mathbb{R}^{2m}, y_2 \in \mathbb{R}^{2q}} \widehat{\mathcal{L}}(u, y) := \langle K_1 u - \hat{b}, y_1 \rangle + \langle K_2 u, y_2 \rangle + \lambda j(u) - \chi_{\mathcal{B}_2}(y_1) - \chi_{\mathcal{B}_\infty}(y_2/\lambda).$$

We use $y = (y_1, y_2)^\top$ to denote the dual variables and start by setting $u^{(0)} = u_0$ and $y^{(0)} = y_0 = (y_{01}, y_{02})$ with $\|y_{01}\|_{l^2} \leq 1$ (inputs to the iterations). We iterate a primal-dual algorithm as before, but now the updates become

$$u^{(k+1)} = \operatorname{argmin}_{u \in \mathbb{R}^{2N}} \lambda j(u) + \frac{1}{2\tau_1} \|u - (u^{(k)} - \tau_1 K^* y^{(k)})\|_{l^2}^2 = \operatorname{prox}_{\lambda \tau_1 j}\left(u^{(k)} - \tau_1 K^* y^{(k)}\right)$$

$$y^{(k+1)} = \operatorname{argmin}_{y \in \mathbb{R}^{2m+2q}} \chi_{\mathcal{B}_2}(y_1) + \langle \hat{b}, y_1 \rangle + \chi_{\mathcal{B}_\infty}(y_2/\lambda) + \frac{1}{2\tau_2} \|y - y^{(k)} - \tau_2 K(2u^{(k+1)} - u^{(k)})\|_{l^2}^2$$

$$= \left(\vartheta(y_1^{(k)} + \tau_2 K_1(2u^{(k+1)} - u^{(k)}) - \tau_2 \hat{b}), \varsigma_\lambda(y_2^{(k)} + \tau_2 K_2(2u^{(k+1)} - u^{(k)}))\right)^\top,$$

where $\tau_1, \tau_2 > 0$ denote proximal step sizes and $\vartheta(y) := \min\{1, \|y\|_{l^2}^{-1}\}y$. Analogous to (2.4), we use $\widehat{\text{PD}}_{\tau_1, \tau_2}$ to denote the exact updates. The following theorem bounds the gap $\widehat{G}_{\lambda^{-1}}$ for inexact primal-dual updates.

**Theorem SM1.1** (Stable bounds on $\widehat{G}_{\lambda^{-1}}$ for inexact primal-dual updates). *Suppose that $\tau_1$ and $\tau_2$ satisfy $\tau_1 \tau_2 (\|A\|^2 + \|B\|^2) < 1$. Let $y_0 = (y_{01}, y_{02})^\top \in \mathbb{R}^{2m+2q}$ with $\|y_{01}\|_{l^2} \leq 1$ and $\|y_{02}\|_{l^2} \leq \lambda\sqrt{q}$, $x_0 \in \mathbb{C}^N$ and $u^{(0)} = [\text{real}(x_0); \text{imag}(x_0)]$. Set $\tilde{u}^{(0)} = u^{(0)}$ and $\tilde{y}^{(0)} = y^{(0)} = y_0$. Suppose that $(\tilde{u}^{(k)}, \tilde{y}^{(k)})$ are such that*

$$\|(\tilde{u}^{(k)}, \tilde{y}^{(k)}) - \widehat{\text{PD}}_{\tau_1, \tau_2}(\tilde{u}^{(k-1)}, \tilde{y}^{(k-1)})\|_{l^2} \leq \epsilon_k, \quad k \geq 1.$$

*In other words, each primal-dual iterate is approximately applied/computed to accuracy $\epsilon_k$. Define the ergodic averages*

$$U^{(n)} = \frac{1}{n} \sum_{k=1}^n u^{(k)}, \quad \tilde{U}^{(n)} = \frac{1}{n} \sum_{k=1}^n \tilde{u}^{(k)}, \quad Y^{(n)} = \frac{1}{n} \sum_{k=1}^n y^{(k)}, \quad \tilde{Y}^{(n)} = \frac{1}{n} \sum_{k=1}^n \tilde{y}^{(k)},$$

---

and let $X_n \in \mathbb{C}^N$ and $\tilde{X}_n \in \mathbb{C}^N$ denote the complexifications of $U^{(n)} \in \mathbb{R}^{2N}$ and $\tilde{U}^{(n)} \in \mathbb{R}^{2N}$ respectively. Then
(SM1.2)
$$\left\|\left(X_n - \tilde{X}_n, Y^{(n)} - \tilde{Y}^{(n)}\right)^\top\right\|_{l^2} \leq \left[\sqrt{\frac{\tau_1 + \tau_2}{1 - \tau_1\tau_2(\|A\|^2 + \|B\|^2)}}\sqrt{\tau_1^{-1} + \tau_2^{-1}}\right]\frac{1}{n}\sum_{k=1}^{n}\sum_{j=1}^{k}\epsilon_j,$$

and for any $x \in \mathbb{C}^N$,

(SM1.3) $$\lambda \cdot \widehat{G}_{\lambda^{-1}}(X_n, x, b) \leq \frac{1}{n}\left(\frac{\|x_0 - x\|_{l^2}^2}{\tau_1} + \frac{4(1 + \lambda^2 q)}{\tau_2}\right).$$

*Proof.* Theorem 1 and remark 2 of [SM1] show that for any $u \in \mathbb{R}^{2N}$ and any $y \in \mathbb{R}^{2m+2q}$,

(SM1.4) $$\widehat{\mathcal{L}}\left(U^{(n)}, y\right) - \widehat{\mathcal{L}}\left(u, Y^{(n)}\right) \leq \frac{\|u_0 - u\|_{l^2}^2}{n\tau_1} + \frac{\|y_0 - y\|_{l^2}^2}{n\tau_2},$$

since $\tau_1\tau_2\|K\|^2 \leq \tau_1\tau_2(\|A\|^2 + \|B\|^2) < 1$. Let $x = x_1 + ix_2$, and $y_1$ be parallel to $K_1U^{(n)} - \hat{b}$ with $\|y_1\|_{l^2} = 1$. Recalling the definition of $\widehat{\mathcal{L}}$, we see that (SM1.4) now yields

(SM1.5)

$$\|K_1U^{(n)} - \hat{b}\|_{l^2} + \lambda\mathcal{J}(X_n) - \langle K_1u - \hat{b}, Y_1^{(n)}\rangle - \lambda\mathcal{J}(x) + \chi_{\mathcal{B}_2}(Y_1^{(n)})$$
$$+ \langle K_2U^{(n)}, y_2\rangle - \chi_{\mathcal{B}_\infty}(y_2/\lambda) - \langle K_2u, Y_2^{(n)}\rangle + \chi_{\mathcal{B}_\infty}(Y_2^{(n)}/\lambda) \leq \frac{\|u_0 - u\|_{l^2}^2}{n\tau_1} + \frac{\|y_0 - y\|_{l^2}^2}{n\tau_2}.$$

Arguing similarly to the proof of Theorem 2.1, we now choose $y_2$ of complex $l^\infty$-norm $\lambda$ such that $\langle K_2U^{(n)}, y_2\rangle = \lambda\|BX_n\|_{l^1}$. Since $\|y_2\|_2^2 \leq \lambda^2 q$, and hence $\|y_0 - y\|_{l^2} \leq 2\sqrt{1 + \lambda^2 q}$ it follows that (SM1.5) reduces to

(SM1.6) $$\|K_1U^{(n)} - \hat{b}\|_{l^2} + \lambda\mathcal{J}(X_n) + \lambda\|BX_n\|_{l^1} - \langle K_1u - \hat{b}, Y_1^{(n)}\rangle - \lambda\mathcal{J}(x)$$
$$+ \chi_{\mathcal{B}_2}(Y_1^{(n)}) - \langle K_2u, Y_2^{(n)}\rangle + \chi_{\mathcal{B}_\infty}(Y_2^{(n)}/\lambda) \leq \frac{\|u_0 - u\|_{l^2}^2}{n\tau_1} + \frac{4(1 + \lambda^2 q)}{n\tau_2}.$$

The left-hand side must be finite. It follows that $\|Y_1^{(n)}\|_{l^2} \leq 1$ and hence that $\langle K_1u - \hat{b}, Y_1^{(n)}\rangle \leq \|Ax - b\|_{l^2}$. Similarly, since $\chi_{\mathcal{B}_\infty}(Y_2^{(n)}/\lambda) = 0$, we must have that $\langle K_2u, Y_2^{(n)}\rangle \leq \lambda\|Bx\|_{l^1}$. Hence (SM1.6) reduces to (SM1.3) upon complexification. The proof of (SM1.2) follows the proof of Theorem 2.1 (specifically of (2.6)), but using $h^*(y) = \chi_{\mathcal{B}_2}(y_1) + \langle \hat{b}, y_1\rangle + \chi_{\mathcal{B}_\infty}(y_2/\lambda)$.∎

We now prove Theorem 2.2, where we remind the reader that

$$\widehat{\mathcal{J}}(x) := \mathcal{J}(x) + \|Bx\|_{l^1}.$$

*Proof of Theorem* 2.2. Consider the exact iterations described in Theorem SM1.1 with $n = k$, $\lambda = 1/\hat{C}_2$, non-rescaled input $z_0$ (corresponding to $y_0$), rescaled input $b/\beta$ and $x_0/\beta$

for a given $k \in \mathbb{N}$, and $\beta > 0$ (both of which are explicitly defined below). We denote the maps corresponding to $X_k$ and $Y^{(k)}$ respectively as

$$\psi_k = \psi_k\left(\frac{b}{\beta}, \frac{x_0}{\beta}, z_0\right), \quad \varphi_k = \varphi_k\left(\frac{b}{\beta}, \frac{x_0}{\beta}, z_0\right),$$

where we complexify the dual variable to obtain $\varphi_k \in \mathbb{C}^m$. For simplicity, we have assumed exact computations, but a bound for $\mu$ similar to (2.16) can be derived with slight changes to the choice of $\beta$ and $k$ as in the proof of Theorem 1.1 (in practical computations we found that this never an issue). Theorem SM1.1 ensures that, for any initial dual variable $z_0 = (z_{01}, z_{02})^\top$ with $\|z_{01}\|_{l^2} \leq 1$ and $\|z_{02}\|_{l^2} \leq \hat{C}_2^{-1}\sqrt{q}$,

$$(\text{SM1.7}) \quad \lambda\left(\widehat{\mathcal{J}}(\psi_k) - \widehat{\mathcal{J}}(x/\beta)\right) + \left\|A\psi_k - \frac{b}{\beta}\right\|_{l^2} - \left\|A\frac{x}{\beta} - \frac{b}{\beta}\right\|_{l^2} \leq \frac{1}{k}\left(\frac{\|x - x_0\|_{l^2}^2}{\beta^2 \tau_1} + \frac{4(1+\lambda^2 q)}{\tau_2}\right)$$

for any $x \in \mathbb{C}^N$. Given the upper bound $L$ for $\sqrt{\|A\|^2 + \|B\|^2}$, we choose $\tau_1^{-1} = \tau_2^{-1} \leq \tau^{-1}L$ with $\tau^{-1}L \lesssim \tau_1^{-1}$. Define the map $F_k^\beta : \mathbb{C}^m \times \mathbb{C}^N \times \mathbb{C}^{m+q} \to \mathbb{C}^N \times \mathbb{C}^{m+q}$ by

$$F_k^\beta(b, x_0, z_0) = \left(\underbrace{\beta \cdot \psi_k\left(\frac{b}{\beta}, \frac{x_0}{\beta}, z_0\right)}_{=:\widehat{H}_k^\beta(b,x_0,z_0)}, \underbrace{\varphi_k\left(\frac{b}{\beta}, \frac{x_0}{\beta}, z_0\right)}_{=:I_k^\beta(b,x_0,z_0)}\right).$$

Rescaling (SM1.7) (with $\lambda = \hat{C}_2^{-1}$) and using the fact that $\widehat{\mathcal{J}}$ is positive homogenous of degree 1 yields

$$\widehat{G}_{\hat{C}_2}\left(\widehat{H}_k^\beta(b, x_0, z_0), x, b\right) = \widehat{\mathcal{J}}\left(\widehat{H}_k^\beta(b, x_0, z_0)\right) - \widehat{\mathcal{J}}(x) + \hat{C}_2\left\|A\widehat{H}_k^\beta(b, x_0, z_0) - b\right\|_{l^2} - \hat{C}_2\|Ax - b\|_{l^2}$$

$$\leq \frac{\hat{C}_2 L}{k\tau}\left(\frac{\|x - x_0\|_{l^2}^2}{\beta} + 4(1 + q/\hat{C}_2^2)\beta\right).$$

Combining this with (2.22), we obtain the key inequality

$$(\text{SM1.8}) \quad \widehat{G}_{\hat{C}_2}\left(\widehat{H}_k^\beta(b, x_0, z_0), x, b\right) \leq \frac{L\hat{C}_2\hat{C}_1^2}{\tau k \beta}\left[\hat{c}(x, b) + \widehat{G}_{\hat{C}_2}(x_0, x, b)\right]^2 + \frac{4L\beta}{k}(\hat{C}_2 + q/\hat{C}_2)\tau^{-1}.$$

For pairs $(\varkappa, b)$ with $\hat{c}(\varkappa, b) \leq \delta$ and $\widehat{G}_{\hat{C}_2}(x_0, \varkappa, b) \leq \epsilon_0$, (SM1.8) implies that

$$\widehat{G}_{\hat{C}_2}\left(\widehat{H}_k^\beta(b, x_0, z_0), \varkappa, b\right) \leq \frac{L\hat{C}_2\hat{C}_1^2}{\tau k \beta}(\delta + \epsilon_0)^2 + \frac{4L\beta}{k}(\hat{C}_2 + q/\hat{C}_2)\tau^{-1}.$$

Balancing the two terms on the right-hand side leads to $\beta = \hat{C}_1(\delta + \epsilon_0)/\left[2\sqrt{1 + q\hat{C}_2^{-2}}\right]$, with the corresponding bound

$$\widehat{G}_{\hat{C}_2}\left(\widehat{H}_k^\beta(b, x_0, z_0), \varkappa, b\right) \leq \frac{4L\hat{C}_1\sqrt{\hat{C}_2^2 + q}}{\tau k}(\delta + \epsilon_0).$$

For a given $v \in (0,1)$, we define

$$k(v,\tau) = \left\lceil \frac{4L\hat{C}_1\sqrt{\hat{C}_2^2 + q}}{v\tau} \right\rceil, \quad \beta(v,\tau,\epsilon_0) = \frac{\hat{C}_1(\delta + \epsilon_0)}{2\sqrt{1 + q\hat{C}_2^{-2}}}.$$

This ensures that

$$\widehat{G}_{\hat{C}_2}\left(\widehat{H}_k^\beta(b, x_0, z_0), \varkappa, b\right) \leq v(\delta + \epsilon_0)$$

whenever $\widehat{G}_{\hat{C}_2}(x_0, \varkappa, b) \leq \epsilon_0$.

We are now ready to describe the restart scheme. Note first that $\widehat{G}_{\hat{C}_2}(0, \varkappa, b) \leq \hat{C}_2\|b\|_{l^2}$. Given $n \in \mathbb{N}$, we set $\epsilon_0 = \hat{C}_2\|b\|_{l^2}$ and for $j = 1,...,n-1$ set $\epsilon_j = v(\delta + \epsilon_{j-1})$. By summing a geometric series, this implies $\epsilon_n \leq \frac{v\delta}{1-v} + v^n\hat{C}_2\|b\|_{l^2}$. We define $\phi_n(b)$ (and $\hat{\phi}_n(b)$) iteratively as follows. We set

$$\phi_1(b) = \widehat{H}_{k(v,\tau)}^{\beta(v,\tau,\epsilon_0)}(b,0,0), \quad \hat{\phi}_1(b) = I_{k(v,\tau)}^{\beta(v,\tau,\epsilon_0)}(b,0,0)$$

and for $j = 2,...,n$ we set

$$\phi_j(b) = \widehat{H}_{k(v,\tau)}^{\beta(v,\tau,\epsilon_{j-1})}(b, \phi_{j-1}(b), \hat{\phi}_{j-1}(b)), \quad \hat{\phi}_j(b) = I_{k(v,\tau)}^{\beta(v,\tau,\epsilon_{j-1})}(b, \phi_{j-1}(b), \hat{\phi}_{j-1}(b))$$

Note in particular that owing to the projections $\vartheta$ and $\varsigma_\lambda$, each $\hat{\phi}_j(b)$ has $\|\hat{\phi}_j(b)_1\|_{l^2} \leq 1$ and $\|\hat{\phi}_j(b)_2\|_{l^2} \leq \hat{C}_2^{-1}\sqrt{q}$. This allows us to apply Theorem SM1.1 at each inductive step. The choice of $\epsilon_j$ and the above argument inductively shows that $\widehat{G}_{\hat{C}_2}(\phi_j(b), \varkappa, b) \leq \epsilon_j$. Hence,

$$\widehat{G}_{\hat{C}_2}(\phi_n(b), \varkappa, b) \leq \epsilon_n \leq \frac{v\delta}{1-v} + v^n\hat{C}_2\|b\|_{l^2}.$$

Combining with (2.22), we see that

(SM1.9) $$\|\phi_n(b) - \varkappa\|_{l^2} \leq \hat{C}_1\left(\frac{\delta}{1-v} + v^n\hat{C}_2\|b\|_{l^2}\right). \blacksquare$$

**SM2. Proofs of results and further details for section 3.** We begin with the proof of results from section 3. The following two lemmas are taken from the compressed sensing literature [SM2].

**Lemma SM2.1 (rNSPL implies $l_w^1$ distance bound).** *Suppose that $A$ has the weighted rNSPL of order $(\mathbf{s}, \mathbf{M})$ with constants $0 < \rho < 1$ and $\gamma > 0$. Let $x, \hat{x} \in \mathbb{C}^N$, then*

(SM2.1) $$\|\hat{x} - x\|_{l_w^1} \leq \frac{1+\rho}{1-\rho}\left(2\sigma_{\mathbf{s},\mathbf{M}}(x)_{l_w^1} + \|\hat{x}\|_{l_w^1} - \|x\|_{l_w^1}\right) + \frac{2\gamma}{1-\rho}\sqrt{\xi}\|A(\hat{x} - x)\|_{l^2}.$$

**Lemma SM2.2 (rNSPL implies $l^2$ distance bound).** *Suppose that $A$ has the weighted rNSPL of order $(\mathbf{s}, \mathbf{M})$ with constants $0 < \rho < 1$ and $\gamma > 0$. Let $x, \hat{x} \in \mathbb{C}^N$, then*

(SM2.2) $$\|\hat{x} - x\|_{l^2} \leq \left(\rho + \frac{(1+\rho)\kappa^{1/4}}{2}\right)\frac{\|\hat{x} - x\|_{l_w^1}}{\sqrt{\xi}} + \left(1 + \frac{\kappa^{1/4}}{2}\right)\gamma\|A(\hat{x} - x)\|_{l^2}.$$

Combining these two lemmas, we can prove Lemma SM2.3.

Lemma SM2.3. *Suppose that $A$ has the weighted rNSPL of order $(\mathbf{s}, \mathbf{M})$ with constants $0 < \rho < 1$ and $\gamma > 0$. Then the assumption (1.3) holds with*

$$C_1 = \left(\rho + \frac{(1+\rho)\kappa^{1/4}}{2}\right)\frac{1+\rho}{\sqrt{\xi}(1-\rho)},$$

$$C_2 = \frac{\left(1+\frac{\kappa^{1/4}}{2}\right)\gamma + \left(\rho + \frac{(1+\rho)\kappa^{1/4}}{2}\right)\frac{2\gamma}{(1-\rho)}}{C_1} = \frac{\gamma}{C_1} \cdot \frac{2+2\rho+(3+\rho)\kappa^{1/4}}{2(1-\rho)},$$

$$c(x,b) = 2\sigma_{\mathbf{s},\mathbf{M}}(x)_{l_w^1} + C_2\left(\|Ax-b\|_{l^2} + \epsilon\right).$$

*Moreover,*

(SM2.3) $$\|\hat{x} - x\|_{l_w^1} \leq \frac{1+\rho}{1-\rho}\left(G_{C_2}(\hat{x}, x, b) + c(x,b)\right).$$

*Proof of Lemma SM2.3.* We first substitute (SM2.1) into the right-hand side of (SM2.2) to obtain

$$\|\hat{x} - x\|_{l^2} \leq C_1\left(\|\hat{x}\|_{l_w^1} - \|x\|_{l_w^1}\right) + 2C_1\sigma_{\mathbf{s},\mathbf{M}}(x)_{l_w^1} + C_1C_2\|A(\hat{x}-x)\|_{l^2}.$$

Using $\|A(\hat{x}-x)\|_{l^2} \leq \|A\hat{x}-b\|_{l^2} - \epsilon + \|Ax-b\|_{l^2} + \epsilon$, and rearranging, we arrive at (1.3) for the stated choice of $C_1$, $C_2$ and $c$. For the final part, note that

$$\frac{2\gamma}{1+\rho}\sqrt{\xi} = \frac{\left(\rho + \frac{(1+\rho)\kappa^{1/4}}{2}\right)\frac{2\gamma}{(1-\rho)}}{C_1} \leq C_2.$$

Combining this with (SM2.1), we see that

$$\|\hat{x} - x\|_{l_w^1} \leq \frac{1+\rho}{1-\rho}\left(2\sigma_{\mathbf{s},\mathbf{M}}(x)_{l_w^1} + \|\hat{x}\|_{l_w^1} - \|x\|_{l_w^1} + C_2\|A(\hat{x}-x)\|_{l^2}\right).$$

Again, using $\|A(\hat{x}-x)\|_{l^2} \leq \|A\hat{x}-b\|_{l^2} - \epsilon + \|Ax-b\|_{l^2} + \epsilon$, we arrive at (SM2.3). ∎

*Proof of Theorem 3.3.* The only result that does not follow directly from Theorem 1.1 and Lemma SM2.3 is the bound on $\|\phi_n(b) - \varkappa\|_{l_w^1}$. However, the proof of Theorem 1.1 shows that

$$G_{C_2}(\phi_n(b), \varkappa, b) + c(\varkappa, b) \leq \frac{\delta}{1-\exp(-1)}$$
$$+ C_2\|b\|_{l^2} \cdot \exp\left(-T(n)\left[2eL_A\gamma\frac{2+2\rho+(3+\rho)\kappa^{1/4}}{2(1-\rho)}\right]^{-1}\right).$$

Combining this with (SM2.3) gives the required result. ∎

For completeness, we now describe the sampling setup for the example in subsection 3.2. We first recall the concept of multilevel random subsampling.

Definition SM2.4 (Multilevel random subsampling [SM3]). Let $\mathbf{N} = (N_1, \ldots, N_l) \in \mathbb{N}^l$, where $1 \leq N_1 < \cdots < N_l = N$ and $\mathbf{m} = (m_1, \ldots, m_l) \in \mathbb{N}^l$ with $m_k \leq N_k - N_{k-1}$ for $k = 1, \ldots, l$, and $N_0 = 0$. For each $k = 1, \ldots, l$, let $\mathcal{I}_k = \{N_{k-1}+1, \ldots, N_k\}$ if $m_k = N_k - N_{k-1}$ and if not, let $t_{k,1}, \ldots, t_{k,m_k}$ be chosen uniformly and independently from the set $\{N_{k-1}+1, \ldots, N_k\}$ (with possible repeats), and set $\mathcal{I}_k = \{t_{k,1}, \ldots, t_{k,m_k}\}$. If $\mathcal{I} = \mathcal{I}_{\mathbf{N},\mathbf{m}} = \mathcal{I}_1 \cup \cdots \cup \mathcal{I}_l$ we refer to $\mathcal{I}$ as an $(\mathbf{N},\mathbf{m})$-multilevel subsampling scheme.

Definition SM2.5 (Multilevel subsampled unitary matrix). A matrix $A \in \mathbb{C}^{m \times N}$ is an $(\mathbf{N},\mathbf{m})$-multilevel subsampled unitary matrix if $A = P_\mathcal{I} DU$ for a unitary matrix $U \in \mathbb{C}^{N \times N}$ and $(\mathbf{N},\mathbf{m})$-multilevel subsampling scheme $\mathcal{I}$. Here $D$ is a diagonal scaling matrix with

$$D_{ii} = \sqrt{\frac{N_k - N_{k-1}}{m_k}}, \quad i = N_{k-1}+1, \ldots, N_k, \quad k = 1, \ldots, l$$

and $P_\mathcal{I}$ denotes the projection onto the span of the basis vectors indexed by $\mathcal{I}$.

Throughout this subsection, we let $Q = 2^r$ for $r \in \mathbb{N}$, and consider vectors on $\mathbb{C}^Q$ or $d$-dimensional tensors on $\mathbb{C}^{Q \times \cdots \times Q}$. To keep notation consistent with the main text, we set $N = Q^d$ so that the objective is to recover a vectorized $\varkappa \in \mathbb{C}^N$.[1] Let $V \in \mathbb{C}^{N \times N}$ be either the matrix $F^{(d)}$ or $W^{(d)}$, corresponding to the $d$-dimensional discrete Fourier or Walsh–Hadamard transform. In the Fourier case, we divide the different frequencies $\{-Q/2+1, \ldots, Q/2\}^d$ into dyadic bands. For $d = 1$, we let $B_1 = \{0,1\}$ and $B_k = \{-2^{k-1}+1, \ldots, -2^{k-2}\} \cup \{2^{k-2}+1, \ldots, 2^{k-1}\}$ for $k = 2, \ldots, r$. In the binary case, we define the frequency bands $B_1 = \{0, 1\}$ and $B_k = \{2^{k-1}, \ldots, 2^k - 1\}$ for $k = 2, \ldots, r$ in the one-dimensional case. In the general $d$-dimensional case for Fourier or binary sampling, we set $B_{\mathbf{k}}^{(d)} = B_{k_1} \times \ldots \times B_{k_d}$ for $\mathbf{k} = (k_1, \ldots, k_d) \in \mathbb{N}^d$. To recover a sparse representation, we consider the Haar wavelet coefficients for simplicity, though similar statements can be made for higher order Daubechies wavelets [SM4] and [SM5, Table 1]. We denote the discrete Haar Wavelet transform by $\Phi \in \mathbb{C}^{N \times N}$. We consider a multilevel subsampled unitary matrix (Definition SM2.5), with $U = V\Psi^*$. Given $\{m_{\mathbf{k}=(k_1,\ldots,k_d)}\}_{k_1,\ldots,k_d=1}^r$, we use multilevel random sampling with $m_{\mathbf{k}}$ measurements chosen from $B_{\mathbf{k}}^{(d)}$ according to Definition SM2.4. This corresponds to $l = r^d$ and the $N_i$'s can be chosen given a suitable ordering of the Fourier/Walsh basis. The sparsity in levels structure (Definition 3.1) is chosen to correspond to the $r$ wavelet levels. Finally, we define
(SM2.4)
$$\mathcal{M}_\mathcal{F} = \sum_{j=1}^{\|\mathbf{k}\|_{l^\infty}} s_j \prod_{i=1}^d 2^{-|k_i - j|} + \sum_{j=\|\mathbf{k}\|_{l^\infty}+1}^r s_j 2^{-2(j-\|\mathbf{k}\|_{l^\infty})} \prod_{i=1}^d 2^{-|k_i-j|}, \quad \mathcal{M}_\mathcal{W} = s_{\|\mathbf{k}\|_{l^\infty}} \prod_{i=1}^d 2^{-|k_i - \|\mathbf{k}\|_{l^\infty}|}.$$

The following result was proven in [SM6].

Theorem SM2.6. *Consider the above setup of recovering a $d$-dimensional tensor $c \in \mathbb{C}^{Q^d}$ ($N = Q^d$) from subsampled Fourier or binary measurements $Vc$, such that $A$ is a multilevel subsampled unitary matrix with respect to $U = V\Psi^*$. Let $\epsilon_\mathbb{P} \in (0,1)$ and $\mathcal{L} = d \cdot r^2 \cdot \log(2m) \cdot \log^2(s \cdot \kappa(\mathbf{s}, \mathbf{M}, w)) + \log(\epsilon_\mathbb{P}^{-1})$. Suppose that:*

---

[1] The following can also be generalized to rectangles (i.e. $\mathbb{C}^{2^{r_1} \times \cdots \times 2^{r_d}}$ with possibly different $r_1, \ldots, r_d$) or dimensions that are not powers of two.

  (a) In the Fourier case, $m_{\mathbf{k}} \gtrsim \kappa(\mathbf{s}, \mathbf{M}, w) \cdot \mathcal{M}_{\mathcal{F}}(\mathbf{s}, \mathbf{k}) \cdot \mathcal{L}$.
  (b) In the binary case, $m_{\mathbf{k}} \gtrsim \kappa(\mathbf{s}, \mathbf{M}, w) \cdot \mathcal{M}_{\mathcal{W}}(\mathbf{s}, \mathbf{k}) \cdot \mathcal{L}$.
*Then with probability at least* $1 - \epsilon_{\mathbb{P}}$, $A$ *satisfies the weighted rNSPL of order* $(\mathbf{s}, \mathbf{M})$ *with constants* $\rho = 1/16$ *and* $\gamma = \sqrt{3/2}$.

The sampling conditions are optimized by minimizing $\kappa(\mathbf{s}, \mathbf{M}, w)$. Up to a constant scale, this corresponds to the choice $w_{(j)} = \sqrt{s/s_j}$. Up to log-factors, the measurement condition then becomes equivalent to the currently best-known oracle estimator (where one assumes apriori knowledge of the support of the vector) [SM7, Prop. 3.1]. Theorem SM2.6 gives us an immediate example of being able to apply Theorem 3.3 as is done in the main text.

**SM3. Proofs of results in section 4 and section 5.** We first consider the uniform recovery guarantees in section 4, for which we make use of the following theorem that generalizes Lemma SM2.3.

**Theorem SM3.1** ( [SM8, Theorem 3.2]). *Let $p \in [1, 2]$ and suppose that $A : \mathbb{C}^{n_1 \times n_2} \to \mathbb{C}^m$ satisfies the Frobenius-robust rank null space property of order $r$ with constants $\rho \in (0, 1)$ and $\gamma > 0$. Then for any $M, \widehat{M} \in \mathbb{C}^{n_1 \times n_2}$,*
(SM3.1)
$$\|\widehat{M} - M\|_p \leq \frac{(1+\rho)^2}{(1-\rho)r^{\frac{p-1}{p}}} \left(2\|M_c\|_1 + \|\widehat{M}\|_1 - \|M\|_1\right) + \frac{\gamma(3+\rho)}{1-\rho} r^{\frac{1}{p}-\frac{1}{2}} \|A(\widehat{M} - M)\|_{l^2}.$$

In particular, taking the $p = 2$ case in (SM3.1) and using $\|A(\widehat{M} - M)\|_{l^2} \leq \|A(\widehat{M}) - b\|_{l^2} - \epsilon + \|A(M) - b\|_{l^2} + \epsilon$, we see that (1.3) is satisfied with

$$C_1 = \frac{(1+\rho)^2}{(1-\rho)r^{\frac{1}{2}}}, \quad C_2 = \frac{\gamma(3+\rho)r^{\frac{1}{2}}}{(1+\rho)^2}, \quad c(M,b) = 2\|M_c\|_1 + \frac{\gamma(3+\rho)r^{\frac{1}{2}}}{(1+\rho)^2}\left(\|A(M) - b\|_{l^2} + \epsilon\right).$$

We can now finish the proof of Theorem 4.2.

*Proof of Theorem* 4.2. The proof of Theorem 1.1 shows that

$$G_{C_2}(\phi_n(b), M, b) + c(M, b) \leq \frac{\delta}{1 - \exp(-1)} + C_2 \|b\|_{l^2} \cdot \exp\left(-T(n) \left[2eL_A \gamma \frac{(3+\rho)}{(1-\rho)}\right]^{-1}\right),$$

where we have used the remarks after the theorem and $C_1 C_2 = \frac{\gamma(3+\rho)}{1-\rho}$. The result now follows from (SM3.1) in Theorem SM3.1. ∎

We now turn to the non-uniform recovery guarantees in section 5 for matrix completion. We will need the following non-symmetric pinching lemma.

**Lemma SM3.2** ( [SM9]). *Let $P_1 \in \mathbb{C}^{n_1 \times n_1}$ and $P_2 \in \mathbb{C}^{n_2 \times n_2}$ be two orthogonal projection matrices. Then for any $M \in \mathbb{C}^{n_1 \times n_2}$,*
(SM3.2)
$$\|M\|_1 \geq \|P_1 M P_2\|_1 + \|P_1^\perp M P_2^\perp\|_1.$$

We are now ready to prove Theorem 5.3. We use the following convention for the Hilbert–Schmidt inner product:
$$\operatorname{tr}(M_2^* M_1) = \langle M_2, M_1 \rangle.$$

*Proof of Theorem* 5.3. Let $\widehat{M} \in \mathbb{C}^{n_1 \times n_2}$ and $\Delta = \widehat{M} - M$. Using Lemma SM3.2 and $P^\perp M Q^\perp = 0$,

$$\|\widehat{M}\|_1 \geq \|P\widehat{M}Q\|_1 + \|P^\perp \widehat{M} Q^\perp\|_1 = \|P\widehat{M}Q\|_1 + \|P^\perp \Delta Q^\perp\|_1.$$

Let $\Delta_T^\perp = P^\perp \Delta Q^\perp$ and $\Delta_T = \Delta - \Delta_T^\perp$. Using the fact that $P\widehat{M}Q = M + P\Delta Q$, we have

(SM3.3) $$\|\widehat{M}\|_1 \geq \|M + P\Delta Q\|_1 + \|\Delta_T^\perp\|_1.$$

Since $\|UV^*\| \leq 1$ and $\langle UV^*, M\rangle = \|M\|_1$, we have that

(SM3.4) $$\|M\|_1 - |\langle UV^*, P\Delta Q\rangle| \leq |\langle UV^*, M + P\Delta Q\rangle| \leq \|M + P\Delta Q\|_1.$$

Writing out the inner product in terms of the trace, we see that

$$\langle UV^*, P\Delta Q\rangle = \operatorname{tr}(VU^*UU^*\Delta VV^*) = \operatorname{tr}(VU^*\Delta) = \langle UV^*, \Delta\rangle,$$

where we use the cyclic property of tr and $U^*U = V^*V = I_r$. We then use the decomposition

(SM3.5) $$\langle UV^*, \Delta\rangle = \langle Y, \Delta\rangle + \langle UV^* - Y, \Delta\rangle.$$

Using the definition of the adjoint, the first inner product on the right-hand side of (SM3.5) is equal to $\langle z, A(\Delta)\rangle$. The second inner product can be written as

$$\langle UV^* - Y, \Delta\rangle = \langle UV^* - Y, \Delta_T\rangle - \langle P_{T_M^\perp} Y, \Delta_T^\perp\rangle,$$

where we have used the fact that $\Delta_T^\perp = P_{T_M^\perp} \Delta_T^\perp$ and $P_{T_M^\perp} UV^* = 0$. Note also that

$$|\langle UV^* - Y, \Delta_T\rangle| = |\langle UV^* - P_{T_M} Y, \Delta_T\rangle| \leq \alpha_1 \|\Delta_T\|_2.$$

We then combine these arguments and use (5.2) to obtain

$$|\langle UV^*, P\Delta Q\rangle| \leq \|z\|_{l^2} \|A(\Delta)\|_{l^2} + \alpha_1 \|\Delta_T\|_2 + \alpha_2 \|\Delta_T^\perp\|_2.$$

Combining this with (SM3.3) and (SM3.4) yields

$$\|\widehat{M}\|_1 \geq \|M\|_1 + \|\Delta_T^\perp\|_1 - \|z\|_{l^2}\|A(\Delta)\|_{l^2} - \alpha_1\|\Delta_T\|_2 - \alpha_2\|\Delta_T^\perp\|_2.$$

Since $\|\cdot\|_2 \leq \|\cdot\|_1$ and $\alpha_2 < 1$, we obtain the inequality

(SM3.6) $$(1-\alpha_2)\|\Delta_T^\perp\|_2 - \alpha_1\|\Delta_T\|_2 \leq \|\widehat{M}\|_1 - \|M\|_1 + \|z\|_{l^2}\|A(\Delta)\|_{l^2}.$$

We now note that

$$\|A(\Delta_T)\|_{l^2} - \|A(\Delta_T^\perp)\|_{l^2} \leq \|A(\Delta)\|_{l^2}.$$

Due to (5.3) and the fact that $\Delta_T \in T_M$, it follows that

(SM3.7) $$\gamma \|\Delta_T\|_2 - \|A\|\|\Delta_T^\perp\|_2 \leq \|A(\Delta)\|_{l^2}.$$

Combining (SM3.6) and (SM3.7), we have

$$\|\Delta_T\|_2 + \|\Delta_T^\perp\|_2 \leq \frac{\gamma + \|A\|}{(1-\alpha_2)\gamma - \alpha_1 \|A\|}\left[\|\widehat{M}\|_1 - \|M\|_1 + \left(\frac{\alpha_1 + 1 - \alpha_2}{\gamma + \|A\|} + \|z\|_{l^2}\right)\left\|A\left(\widehat{M} - M\right)\right\|_{l^2}\right].$$

The inequality (5.4) now follows. Finally, using

$$\|A(\widehat{M} - M)\|_{l^2} \leq \|A(\widehat{M}) - b\|_{l^2} - \epsilon + \epsilon + \|A(M) - b\|_{l^2},$$

(1.3) is locally satisfied with the given values of $C_1$, $C_2$ and $c(M, b)$. ∎

**SM4. Proof of Theorem 6.2.** We first need some results that follow from straightforward adaptations of the arguments laid out in [SM10, Section 2]. For completeness, we have provided the details of the necessary modifications. In what follows, we let $x, \hat{x} \in \mathbb{C}^N$ and set $h = x - \hat{x}$. Let $T_0$ denote the set of the largest $s$ coefficients of $D^*x$ in magnitude, and let $D_T$ denote the matrix $D$ restricted to the columns indexed by $T$. We divide the coordinates $T_0^c$ into sets of size $t$ (chosen later) in order of decreasing magnitude of $D_{T_0^c}^* h$. Call these sets $T_1, T_2, ...$ and set $T_{01} = T_0 \cup T_1$. We also collapse the notation $\delta_s(A, D)$ to $\delta_s$.

First, an application of the triangle inequality yields

$$\|D^*x - D^*h\|_{l^1} \leq \|D^*x\|_{l^1} + (\|D^*\hat{x}\|_{l^1} - \|D^*x\|_{l^1}),$$

which implies

$$\|D_{T_0^c}^* h\|_{l^1} \leq 2\|D_{T_0^c}^* x\|_{l^1} + \|D_{T_0}^* h\|_{l^1} + (\|D^*\hat{x}\|_{l^1} - \|D^*x\|_{l^1}). \tag{SM4.1}$$

The following lemma is a direct generalization of [SM10, Lemma 2.2], and we have omitted the proof since it simply makes use of (SM4.1) instead of [SM10, Lemma 2.1] (which assumes that $\hat{x}$ solves (6.1) so that the bracketed term on the right-hand side of (SM4.1) can be dropped).

**Lemma SM4.1.** *Setting $\rho = s/t$ and $\eta = 2\|D_{T_0^c}^* x\|_{l^1}/\sqrt{s}$, we have*

$$\sum_{j \geq 2} \|D_{T_j}^* h\|_{l^2} \leq \sqrt{\rho}(\|D_{T_0}^* h\|_{l^2} + \eta) + \frac{1}{\sqrt{t}}(\|D^*\hat{x}\|_{l^1} - \|D^*x\|_{l^1}). \tag{SM4.2}$$

The next result we tweak is [SM10, Lemma 2.4], where the following is proven via the same string of inequalities but with Lemma SM4.1 replacing [SM10, Lemma 2.2].

**Lemma SM4.2.** *As a consequence of D-RIP, the following holds:*
$$\sqrt{1 - \delta_{s+t}}\|D_{T_{01}} D_{T_{01}}^* h\|_{l^2} - \sqrt{\rho(1 + \delta_t)}(\|h\|_{l^2} + \eta) \leq \|Ah\|_{l^2} + \frac{\sqrt{1 + \delta_t}}{\sqrt{t}}(\|D^*\hat{x}\|_{l^1} - \|D^*x\|_{l^1}). \tag{SM4.3}$$

Similarly, we obtain the following by adapting the proof of [SM10, Lemma 2.5].

**Lemma SM4.3.** *The vector $h$ satisfies*

$$\|h\|_{l^2}^2 \leq \|h\|_{l^2} \|D_{T_{01}} D_{T_{01}}^* h\|_{l^2} + \left[\sqrt{\rho}(\|D_{T_0}^* h\|_{l^2} + \eta) + \frac{1}{\sqrt{t}}(\|D^*\hat{x}\|_{l^1} - \|D^*x\|_{l^1})\right]^2. \tag{SM4.4}$$

Using these results, we now depart from the argument in [SM10] and prove Theorem 6.2.

*Proof of Theorem 6.2.* To simplify the notation, we set

$$E_1 = \eta + \frac{1}{\sqrt{\rho t}}(\|D^*\hat{x}\|_{l^1} - \|D^*x\|_{l^1}).$$

If $h = 0$, then there is nothing to prove, so we assume that $\|h\|_{l^2} > 0$. The inequality (SM4.4) together with $\|D_{T_0}^* h\|_{l^2} \leq \|h\|_{l^2}$ then implies that

$$\|h\|_{l^2} \leq \|D_{T_{01}} D_{T_{01}}^* h\|_{l^2} + \rho\|h\|_{l^2} + 2\rho E_1 + \rho \frac{E_1^2}{\|h\|_{l^2}}. \tag{SM4.5}$$

If $\|h\|_{l^2} \geq (\rho + \sqrt{\rho^2 + \rho})E_1$, then we must have

$$2\rho E_1 + \rho \frac{E_1^2}{\|h\|_{l^2}} \leq (\rho + \sqrt{\rho^2 + \rho})E_1.$$

Combining with (SM4.5), it follows that (even in the case that $\|h\|_{l^2} < (\rho + \sqrt{\rho^2 + \rho})E_1$)

(SM4.6) $$\|h\|_{l^2} \leq \|D_{T_{01}} D_{T_{01}}^* h\|_{l^2} + \rho\|h\|_{l^2} + (\rho + \sqrt{\rho^2 + \rho})E_1.$$

Combining Lemma SM4.2 and (SM4.6),

(SM4.7) $$\|h\|_{l^2} \leq \left(\frac{\sqrt{\rho(1+\delta_t)}}{\sqrt{1-\delta_{s+t}}} + \rho\right)\|h\|_{l^2} + \frac{\|Ah\|_{l^2}}{\sqrt{1-\delta_{s+t}}} + \left(\rho + \sqrt{\rho^2 + \rho} + \frac{\sqrt{\rho(1+\delta_t)}}{\sqrt{1-\delta_{s+t}}}\right)E_1.$$

We then use $\|Ah\|_{l^2} \leq \|A\hat{x} - b\|_{l^2} - \epsilon + \|Ax - b\|_{l^2} + \epsilon$. Rearranging (SM4.7) gives (1.3) with the parameters in (6.3). ∎

## REFERENCES

[1] Antonin Chambolle and Thomas Pock. On the ergodic convergence rates of a first-order primal–dual algorithm. *Math. Program.*, 159(1-2):253–287, 2016.
[2] Ben Adcock, Vegard Antun, and Anders C Hansen. Uniform recovery in infinite-dimensional compressed sensing and applications to structured binary sampling. *arXiv:1905.00126*, 2019.
[3] Ben Adcock, Anders C Hansen, Clarice Poon, and Bogdan Roman. Breaking the coherence barrier: A new theory for compressed sensing. In *Forum Math. Sigma*, volume 5. CUP, 2017.
[4] Ben Adcock and Anders Hansen. *Compressive Imaging: Structure, Sampling, Learning*. CUP, 2021.
[5] Amirafshar Moshtaghpour, José M Bioucas-Dias, and Laurent Jacques. Close encounters of the binary kind: Signal reconstruction guarantees for compressive Hadamard sampling with Haar wavelet basis. *IEEE Trans. Inform. Theory*, 66(11):7253–7273, 2020.
[6] Matthew J. Colbrook, Vegard Antun, and Anders C. Hansen. Can stable and accurate neural networks be computed? - On the barriers of deep learning and Smale's 18th problem. *arXiv:2101.08286*, 2021.
[7] Ben Adcock, Claire Boyer, and Simone Brugiapaglia. On oracle-type local recovery guarantees in compressed sensing. *Inf. Inference*, 2018.
[8] Maryia Kabanava, Richard Kueng, Holger Rauhut, and Ulrich Terstiege. Stable low-rank matrix recovery via null space properties. *Inf. Inference*, 5(4):405–441, 2016.
[9] Tim Fuchs, David Gross, Peter Jung, Felix Krahmer, Richard Kueng, and Dominik Stöger. Proof methods for robust low-rank matrix recovery. *arXiv:2106.04382*, 2021.
[10] Emmanuel J Candes, Yonina C Eldar, Deanna Needell, and Paige Randall. Compressed sensing with coherent and redundant dictionaries. *Appl. Comput. Harmon. Anal.*, 31(1):59–73, 2011.



\end{document}